\documentclass{amsart}
\usepackage{array}
\newcolumntype{P}[1]{>{\raggedright\let\newline\\\arraybackslash\hspace{0pt}}m{#1}}
\usepackage{graphicx,verbatim, amsmath, amssymb, amsthm, amsfonts, epsfig, amsxtra,ifthen,mathtools,
caption,enumitem,hhline,bbm,capt-of,longtable}	
\usepackage[bookmarks=true, bookmarksopen=false,%
    colorlinks=true,%
    linkcolor=darkblue,%
    citecolor=darkblue,%
    filecolor=darkblue,%
    menucolor=darkblue,%
    linktoc=all,%
    urlcolor=darkblue
]{hyperref}
\usepackage{xcolor}
\definecolor{darkblue}{cmyk}{1,0.3,0,0.1}  
\DeclareFontFamily{U}{mathx}{\hyphenchar\font45}
\DeclareFontShape{U}{mathx}{m}{n}{<-> mathx10}{}
\DeclareSymbolFont{mathx}{U}{mathx}{m}{n}
\DeclareMathAccent{\widebar}{0}{mathx}{"73}

\newtheorem{proposition}{Proposition}[section]
\newtheorem{theorem}[proposition]{Theorem}
\newtheorem{corollary}[proposition]{Corollary}
\newtheorem{lemma}[proposition]{Lemma}

\theoremstyle{definition}

\theoremstyle{remark}
\newtheorem{remark}[proposition]{Remark}

\numberwithin{equation}{section}

\newcommand{\newword}[1]{\textbf{\emph{#1}}}

\newcommand{\integers}{\mathbb Z}

\newcommand{\reals}{\mathbb R}

\newcommand{\ep}{\varepsilon}
\newcommand{\thet}{\vartheta}

\newcommand{\cut}{\operatorname{cut}}

\newcommand{\sgn}{\operatorname{sgn}}

\newcommand{\gFan}{\g\!\operatorname{Fan}}

\newcommand{\set}[1]{{\lbrace #1 \rbrace}}
\newcommand{\sett}[1]{{\bigl\lbrace #1 \bigr\rbrace}}

\newcommand{\br}[1]{{\langle #1 \rangle}}

\newcommand{\brrr}[1]{{\Bigl\langle #1 \Bigr\rangle}}

\newcommand{\A}{{\mathcal A}}

\newcommand{\F}{{\mathcal F}}
\newcommand{\D}{{\mathfrak D}}

\newcommand{\pp}{{\mathbf p}}

\newcommand{\ck}{\spcheck}

\renewcommand{\th}{^\text{th}}

\newcommand{\FF}{\mathbb{F}}

\newcommand{\PP}{\mathbb{P}}

\newcommand{\g}{\mathbf{g}}

\renewcommand{\k}{\mathbbm{k}}
\newcommand{\kk}{{\boldsymbol{k}}}

\newcommand{\x}{\mathbf{x}}

\renewcommand{\v}{\mathbf{v}}

\newcommand{\tB}{{\tilde{B}}}
\newcommand{\T}{\mathcal{T}}

\newcommand{\I}{\mathcal{I}}

\newcommand{\Clear}{\operatorname{Clear}}

\newcommand{\Scat}{\operatorname{Scat}}
\newcommand{\Fan}{\operatorname{Fan}}
\newcommand{\ScatFan}{\operatorname{ScatFan}}

\newcommand{\re}{\mathrm{re}}

\renewcommand{\d}{{\mathfrak d}}
\newcommand{\seg}[1]{\overline{#1}}
\newcommand{\hy}{\hat{y}}

\newcommand{\fin}{\mathrm{fin}}
\renewcommand{\th}{^\text{th}}

\newcommand{\ap}{a^{\circ}}
\newcommand{\anp}{a^{\mathrm{n}\circ}}

\newcommand{\RSChar}{\Phi}
\newcommand{\RS}{\RSChar}

\newcommand{\RSpos}{\RS^+}

\newcommand{\RSfin}{\RS_\fin}

\newcommand{\SimplesChar}{\Pi}
\newcommand{\Simples}{\SimplesChar}

\newcommand{\SimplesTChar}{\Xi}
\newcommand{\SimplesT}[1]{\SimplesTChar^{#1}}

\newcommand{\SuppT}{\operatorname{Supp}_\SimplesTChar}

\newcommand{\AP}[1]{\RS_{#1}}
\newcommand{\APre}[1]{\AP{#1}^\re}
\newcommand{\APTChar}{\Lambda}
\newcommand{\APT}[1]{\APTChar_{#1}}      
\newcommand{\APTre}[1]{\APT{#1}^\re}

\allowdisplaybreaks

\newcommand{\const}{c}
\newcommand{\bl}{\mathfrak{s}}
\newcommand{\mapchar}{t}

\newcommand{\margincolor}{red}      
\definecolor{darkgreen}{rgb}{0,0.7,0}

\newcounter{margincounter}
\newcommand{\marginnum}{
\ifnum\value{margincounter}<10
\textcolor{\margincolor}{\begin{picture}(0,0)\put(2.2,2.4){\circle{9}}\end{picture}\footnotesize\arabic{margincounter}}
\else\ifnum\value{margincounter}<100
\textcolor{\margincolor}{\begin{picture}(0,0)\put(4.256,2.5){\circle{11}}\end{picture}\footnotesize\arabic{margincounter}}
\else
\textcolor{\margincolor}{\begin{picture}(0,0)\put(6.8,2.5){\circle{14}}\end{picture}\footnotesize\arabic{margincounter}}
\fi\fi
}

\makeatletter
\newcommand{\switchmargin}{
\if@reversemargin
\normalmarginpar
\else
\reversemarginpar
\fi
}
\makeatother

\author{Nathan Reading}
\author{Salvatore Stella}
\title{Theta functions in acyclic affine type}
\address[N. Reading]{Department of Mathematics, North Carolina State University, Raleigh, NC, USA}
\address[S. Stella]{Dipartimento di Ingegneria e Scienze dell'Informazione e Matematica, Università degli Studi dell'Aquila, IT}
\thanks{Nathan Reading was partially supported by the National Science Foundation (DMS-2054489) and by the Simons Foundation (581608).
Salvatore Stella was partially supported by PRIN 20223FEA2E - PE1 and by INdAM - GNSAGA}
\date{February 52, 2026}

\begin{document}

\begin{abstract}
We characterize the theta functions for vectors in the imaginary wall in a cluster algebra of acyclic affine type and compute some of their structure constants.
One of the structure constant computations can be interpreted as new ``imaginary'' exchange relations among cluster variables.
We show that theta functions in the imaginary wall span a subalgebra of the cluster algebra that we call the imaginary subalgebra, which decomposes as a tensor product of tube subalgebras that are generalized cluster algebras of type~C.
Our proofs exploit mutation-symmetries of the exchange matrix, an earlier characterization of dominance regions in affine type, and combinatorial models for cluster scattering diagrams of acyclic affine type.
\end{abstract}

\maketitle

\vspace{-10pt}

\setcounter{tocdepth}{2}
{\small\tableofcontents}

\section{Introduction}\label{intro}  
Cluster algebras of finite type (i.e.\ with finitely many seeds) can be understood through combinatorial models arising from the associated finite Coxeter groups/root systems, including the almost positive roots model \cite{ga,MRZ}, sortable elements, and Cambrian lattices/fans \cite{cambrian,sortable,camb_fan}.
This paper is the culmination of an extended effort to understand cluster algebras of affine type just as well by creating and exploiting the analogous tools.

A cluster algebra is of affine type if it has a seed whose exchange matrix is acyclic and is obtained from a Cartan matrix of affine type.
Equivalently when the exchange matrix is larger than $2\times2$, the cluster algebra is of affine type if it has infinitely many seeds but the number of seeds grows linearly as a function of the distance from the initial seed \cite{FeShThTu12,Seven}.
Previous work realized the affine $\g$-vector fan as a doubled Cambrian fan~\cite{afframe}, developed an affine almost positive roots model and extended the fan structure beyond the $\g$-vector fan to fill the whole space~\cite{affdenom}, explicitly constructed the cluster scattering diagram (in the sense of~\cite{GHKK}) and the corresponding cluster scattering fan~\cite{affscat}, and explicitly characterized~\cite{affdomreg} the dominance regions (in the sense of~\cite{RSDom}), which mediate basis changes between nice bases of the cluster algebra~\cite{FanQin}.

In this paper, we apply these combinatorial tools to study theta functions and their structure constants in the acyclic affine case.
Theta functions provide a way to extend the cluster monomials to an additive bases for the whole cluster algebra (under some conditions that we will not make explicit here, because they are known to hold in the acyclic affine case).
Attached to each integer point in the ambient space of the cluster scattering diagram, there is a theta function, defined as a sum of monomials indexed by certain piecewise linear curves called broken lines.
Structure constants for multiplying theta functions are determined by certain pairs of broken lines, and many of the results given here amount to computations of structure constants. 

In specific cases, constructing the cluster scattering diagram explicitly, constructing the relevant broken lines, and ruling out other candidates for broken lines may all be prohibitively difficult.
In affine type we can use the explicit construction of the cluster scattering diagram of acyclic affine type from~\cite{affscat} and the associated combinatorial tools from \cite{afframe,affdenom} to construct the relevant broken lines.
The most technically challenging issue remaining is how to rule out other candidates for broken lines.
For that, we use general tools developed in~\cite{canonical} and affine-type tools developed in this paper.
One tool from~\cite{canonical} rules out broken lines based on mutation-symmetries (sequences of mutations that preserve the exchange matrix).
Another tool from~\cite{canonical} uses dominance regions to vastly shrink the set of structure constants that might be nonzero.
Thus the dominance region computations from~\cite{affdomreg} are crucial to the paper.

The data that determine the theta functions includes an exchange matrix $B$, extended to a matrix~$\tB$ by adding additional entries below $B$. 
The definition in~\cite{GHKK} of theta functions requires the columns of $\tB$ to be linearly independent, but for some $B$, including $B$ of affine type, there is a reasonable notion of theta functions with no conditions on $\tB$, pointed out in~\cite{canonical}, and our results apply in that generality.

The $\g$-vector fan in acyclic affine type covers all of the ambient space except for the interior of a codimension-$1$ cone $\d_\infty$ that we call the imaginary wall because it is a wall in the cluster scattering diagram.
The imaginary wall is the star, in the cluster scattering fan, of a ray called the imaginary ray.
The cones of the cluster scattering fan that contain the imaginary ray are called imaginary cones.

We now provide an informal description of our main results.

\noindent
\begin{itemize}
\item
\textbf{Theorem~\ref{thet xi}} is a formula for the theta function for the primitive integer vector in the imaginary ray in terms of theta functions associated to vectors in the relative boundary of the imaginary wall (and thus essentially in terms of cluster variables).
\item
\textbf{Theorem~\ref{delta cheby}} is a Chebyshev-recursion formula for the theta function for a non-primitive integer vector in the imaginary ray in terms of shorter vectors in that ray.
In light of work of Mandel and Qin~\cite{MandelQin}, this formula, in affine types $\tilde A$ and $\tilde D$, is a special case of a formula due to Musiker, Schiffler, and Williams~\cite{MSW2}.
Here, we show that the recursion applies to all acyclic cluster algebras of affine type and give a proof using broken lines.
\item
\textbf{Theorem~\ref{imag general}} generalizes Theorem~\ref{delta cheby} by describing the product of an arbitrary pair of theta functions for integer vectors in the imaginary ray.
\item
Each vector $\lambda$ is in the relative interior of a unique cone $C$ of the cluster scattering fan.
\textbf{Theorem~\ref{expansion prod}} shows that the theta function for $\lambda$ is the product of theta functions of vectors in the rays of~$C$.
\item
We consider theta functions for vectors in the $\g$-vector fan to be ``known'', because they are essentially cluster monomials.  
The theorems already mentioned amount to a determination of all remaining theta functions, namely those associated to points in the relative interior of the imaginary wall, in terms of known theta functions.
\item
Some rays in the relative boundary of the imaginary wall are exchangeable in an ``imaginary'' sense (by exchanges among maximal imaginary cones) even though they are not exchangeable in the usual (``real'') sense (by exchanges among $\g$-vector cones).
\textbf{Theorem~\ref{im exch}} establishes imaginary exchange relations among the corresponding theta functions.
This relation is central to the main result of \cite{PlamondonStella} on infinite friezes of affine type.
\item
\textbf{Theorem~\ref{subalgebra}} proves the existence of an ``imaginary subalgebra'' of the cluster algebra, the linear span, over the ground ring, of the theta functions for vectors in the imaginary wall. 
\item
The imaginary wall contains subfans that are analogous to the ``tubes'' that occur in the representation theory of tame hereditary algebras.
\textbf{Theorem~\ref{tube subalgebra}} proves the existence of a ``tube subalgebra'' for each tube.
\item
  \textbf{Theorem~\ref{gen clus alg}} proves that, under a certain nondegeneracy condition on the columns added below $B$ to form the extended exchange matrix $\tB$ (satisfied, for example, by principal coefficients), each tube subalgebra is a generalized cluster algebra of finite type C.
The appearance of type-C-Catalan combinatorics in Coxeter-Catalan combinatorics of affine type has been noted in various settings, including \cite{affncA,McSul,affncD,affdomreg,affdenom}, and our results develop the cluster-algebraic manifestation of the phenomenon.
\item
  \textbf{Theorem~\ref{gen clus alg almost}} shows that every imaginary subalgebra is obtained from a generalized cluster algebra by a specialization that depends on the choice of extension of~$B$.
\item
\textbf{Corollary~\ref{cf cor}} says that in the coefficient-free case, the tube subalgebras and imaginary subalgebras are all generalized cluster algebras.
\item
Our results also imply that if we allow a broader ``non-normalized'' definition of a generalized cluster algebra, then every tube subalgebra and imaginary subalgebra is isomorphic to a generalized cluster algebra, with no condition on coefficients.
\end{itemize}

Section~\ref{back sec} contains the background that allows us to state our results precisely in Section~\ref{res sec}.
We gather tools in Section~\ref{tools sec} before giving the proofs in Section~\ref{proofs sec}.

\section{Background}\label{back sec} 

\subsection{Basic background}\label{basic sec}
We assume the basic definitions surrounding cluster algebras as in~\cite{ca4} and the most basic notions of scattering diagrams from~\cite{GHKK}.
Our purpose in this section is to describe how our conventions relate to those of~\cite{GHKK}.
We work in the setting of roots and weights, with a global transpose relative to~\cite{GHKK}.
We also leave out, from the scattering diagram, the extra dimensions related to frozen variables.
(See \cite[Remarks~2.12--2.13]{scatfan}.)
For more background with mostly the same notational conventions, see \cite{scatcomb,affscat}.

We begin with an $n\times n$ \newword{exchange matrix} (that is, a skew-symmetrizable integer matrix) $B=[b_{ij}]$, extended to an \newword{extended exchange matrix} $\tB$ \emph{by adding additional rows}.
The fact that $B$ is skew-symmetrizable means that there exist positive constants $d_1,\ldots,d_n$ such that $d_ib_{ij}=-d_jb_{ji}$ for all $i,j\in\set{1,\ldots,n}$.
We will choose these skew-symmetrizing constants of $B$ in a special way:  Such that $d_i^{-1}$ is an integer for each~$i$ and $\gcd(d_i^{-1}:i\in\set{1,\ldots,n})=1$.

The construction of scattering diagrams and theta functions in \cite{GHKK} requires that the columns of $\tB$ be linearly independent, but we will be able to relax this requirement for theta functions in affine type, as explained in Section~\ref{coeffs sec}.

We take $V$ to be an $n$-dimensional real vector space with a basis $\alpha_1,\ldots,\alpha_n$.
We define a second basis $\alpha_1\ck,\ldots,\alpha_n\ck$ of $V$ by $\alpha_i\ck=d_i^{-1}\alpha_i$.
Let $Q$ be the lattice generated by $\alpha_1,\ldots,\alpha_n$ and write $Q^{\ge0}$ for the subset of $Q$ consisting of nonzero vectors that are nonnegative integer combinations of $\alpha_1,\ldots,\alpha_n$.
We also write $Q^+$ for $Q^{\ge0}\setminus\set{0}$.
Let $Q\ck$ be the lattice generated by $\alpha_1\ck,\ldots,\alpha_n\ck$.
The special choice of skew-symmetrizing constants $d_1,\ldots,d_n$ ensures that $Q\ck$ is a sublattice of $Q$, of finite index.
An element $\beta$ of a lattice $L$ is \newword{primitive} if there do not exist $\beta'\in L$ and an integer $k>1$ with $\beta=k\beta'$.
Given a primitive element $\beta$ of $Q$, let $\beta\ck$ be the primitive element of $Q\ck$ such that $\beta\ck=k\beta$ for some $k\ge1$.

We write $V^*$ for the dual space to $V$, writing $\br{\,\cdot\,,\cdot\,}:V^*\times V\to\reals$ for the usual pairing.
Let $\rho_1,\ldots,\rho_n$ be the basis of $V^*$ with $\br{\rho_i,\alpha_j\ck}=\delta_{ij}$ (Kronecker delta).
Also define a basis $\rho_1\ck,\ldots,\rho_n\ck$ of $V^*$ with $\br{\rho_i\ck,\alpha_j}=\delta_{ij}$.
Let $P$ be the lattice generated by $\rho_1,\ldots,\rho_n$ and let $P\ck$ be the lattice generated by $\rho_1\ck,\ldots,\rho_n\ck$.
The lattice $P\ck$ is a finite-index sublattice of $P$.

From the exchange matrix $B$, we define a Cartan matrix $A=[a_{ij}]$ by setting $a_{ii}=2$ for $i\in\set{1,\ldots n}$ and setting $a_{ij}=-|b_{ij}|$ for $i,j\in\set{1,\ldots,n}$ with $i\neq j$.
We see that $d_ia_{ij}=d_ja_{ji}$ for all $i,j\in\set{1,\ldots,n}$, so that $A$ is symmetrizable.
The~$\alpha_i$ are the \newword{simple roots} of the root system $\RS$ associated to $A$, and the~$\rho_i$ are the \newword{fundamental weights}.
Also, the $\alpha_i\ck$ are the \newword{simple co-roots} and the $\rho_i\ck$ are the \newword{fundamental co-weights}.
The \newword{positive roots} $\RSpos$ are the roots in $\RS$ that are are nonnegative linear combinations of simple roots.
The Cartan matrix defines a symmetric bilinear form $K$ on $V$ by $K(\alpha\ck_i, \alpha_j)=a_{ij}$. 
In other words,~$K$ is computed as a row of simple co-root coordinates times~$A$ times a column of simple root coordinates. 
The lattices $Q$, $Q\ck$, $P$ and $P\ck$ are the \newword{root lattice}, \newword{co-root lattice}, \newword{weight lattice}, and \newword{co-weight lattice}, respectively.
(The standard conventions of Lie theory place roots and weights in the same space and co-roots and co-weights in the same space.
The convention here is much more compatible with the constructions in~\cite{GHKK}.
See \cite[Remark~2.1]{scatcomb} or \cite[Remark~2.1]{affdenom}.)

Of primary importance in this paper is the notion of matrix mutation.
Following the definition in \cite[(2.5)]{ca4} and writing $[x]_+$ for $\max(0,x)$, for each $k=1,\ldots,n$, the \newword{mutation} of $\tB$ in direction $k$ is the matrix $\mu_k(\tB)=[b'_{ij}]$ given by
\begin{equation}\label{b mut}
b_{ij}'=\left\lbrace\!\!\begin{array}{ll}
-b_{ij}&\mbox{if }i=k\mbox{ or }j=k;\\
b_{ij}+[-b_{ik}]_+b_{kj}+b_{ik}[b_{kj}]_+&\mbox{otherwise.}
\end{array}\right.
\end{equation}

The notation $\kk=k_q,\ldots,k_1$ stands for a sequence of indices in $\set{1,\ldots,n}$, so that $\mu_\kk$ means $\mu_{k_q}\circ\mu_{k_{q-1}}\circ\cdots\circ\mu_{k_1}$.

When we mutate $\tB$, passing from $\tB$ to $\mu_\kk(\tB)=[b'_{ij}]$ and thus from $B$ to $\mu_\kk(B)$, the $d_1,\ldots,d_n$ are skew-symmetrizing constants for $\mu_\kk(B)$ as well.  
(That is, $d_ib'_{ij}=-d_jb'_{ji}$.)
As we mutate, we fix the spaces $V$ and $V^*$ and the lattices $Q$, $Q\ck$, $P$, and $P\ck$, but in marked contrast to the approach of \cite{GHKK}, we keep the same preferred bases $\alpha_1,\ldots,\alpha_n$, $\alpha_1\ck,\ldots,\alpha_n\ck$, $\rho_1,\ldots,\rho_n$, and $\rho_1\ck,\ldots,\rho_n\ck$.

However, the Cartan matrix associated to $B$ is not typically the same as the Cartan matrix associated to $\mu_\kk(B)$.
Thus, when we mutate, we change the Cartan matrix and thus change the root system, but the new root system is constructed with the same simple roots, fundamental weights, etc.

Matrix mutation leads to a family of piecewise-linear maps on $V^*$ known as mutation maps, and then to a fan structure on~$V^*$.
In this paper, all mutation maps are relative to the transpose $B^T$ of the relevant exchange matrix $B$, so we define the \newword{mutation map} ${\eta_\kk^{B^T}:V^*\to V^*}$.
To compute $\eta_\kk^{B^T}(v)$ for $v=\sum_{i=1}^n\ell_i\rho_i\in V^*$, we extend $B^T$ by appending a single row $(\ell_i:i=1,\ldots,n)$ below $B^T$, apply $\mu_\kk$, and read off the row $(\ell'_i:i=1,\ldots,n)$ below $\mu_\kk(B^T)$ in the mutated matrix.
Then $\eta_\kk^{B^T}(v)=\sum_{i=1}^n\ell'_i\rho_i\in V^*$.
(Equivalently, we place the entries $\ell_i$ to the right of~$B$ as an additional \emph{column}, mutate, and read off the resulting column.)
The map $\eta_\kk^{B^T}:V^*\to V^*$ is a piecewise linear homeomorphism.

Two vectors $v$ and $v'$ in $V^*$ are \newword{$B^T$-equivalent} if, for all sequences $\kk$ of indices, writing $\eta_\kk^{B^T}(v)=\sum_{i=1}^n\ell_i\rho_i$ and $\eta_\kk^{B^T}(v')=\sum_{i=1}^n\ell'_i\rho_i$, the entries $\ell_i$ and $\ell'_i$ have \emph{strictly} the same sign ($+$, $-$, or $0$) for all~$i$.
A \newword{$B^T$-cone} is the closure of a $B^T$-equivalence class, and is a closed convex cone.
The set of all $B^T$-cones and their faces is a complete fan called the \newword{mutation fan} $\F_{B^T}$.

Following \cite[Section~7]{ca4}, let $x_1,\ldots,x_n$ be indeterminates (corresponding to the rows of $B$).
Let $u_1,\ldots,u_m$ be indeterminates (the \newword{frozen variables} or \newword{tropical variables}) corresponding to the rows of $\tB$ below row $n$.
For $j=1,\ldots,n$, let $y_j=\prod_{i=1}^mu_i^{b_{n+i,j}}$ (the Laurent monomial in the $u_i$ described by the bottom part of the $j\th$ column of~$\tB$).
For $j=1,\ldots,n$, we also define $\hy_j=y_j\prod_{i=1}^nx_i^{b_{ij}}$ (the Laurent monomial in the $x_i$ and $u_i$ described by the entire $j\th$ column of~$\tB$).
In this paper, the indeterminates $u_i$ will often disappear in favor of the monomials~$y_j$ and~$\hy_j$.
For ${\lambda=\sum_{i=1}^n\ell_i\rho_i\in P}$ and $\beta=\sum_{i=1}^nm_i\alpha_i\in Q$, we write $x^\lambda y^\beta$ for the Laurent monomial $x_1^{\ell_1}\cdots x_n^{\ell_n}y_1^{m_1}\cdots y_n^{m_n}$, and similarly we write $x^\lambda \hy^\beta$ for $x_1^{\ell_1}\cdots x_n^{\ell_n}\hy_1^{m_1}\cdots\hy_n^{m_n}$.
Let $\k$ be a field of characteristic zero, let $\k[y]$ be the ring of polynomials in the $y_i$, and let $\k[[\hy]]$ be the ring of formal power series in the~$\hy_i$.

We say that $B$ is \newword{acyclic} if, possibly after reindexing, it has the property that $b_{ij}>0$ implies $i<j$.
We call $\tB$ acyclic whenever $B$ is acyclic.
When $B$ is acyclic, let $c$ be the element in the associated affine Weyl group obtained by multiplying the simple reflections $\set{s_1,\ldots,s_n}$ in an order such that if $s_i$ precedes $s_j$ then $b_{ij}\ge0$.
We define a skew-symmetric bilinear form $\omega_c$ by the formula $\omega_c(\alpha_i\ck,\alpha_j)=b_{ij}$.
In other words, $\omega_c$ is computed as a row of simple co-root coordinates times~$B$ times a column of simple root coordinates. 
For that reason, for $\beta\in Q$, the monomial~$\hy^\beta$ is $x^\lambda y^\beta$ where $\lambda=\omega_c(\,\cdot\,,\beta)\in V^*$.
We also define a bilinear form $E_c$~by
\[
E_c(\alpha\ck_i,\alpha_j)=\begin{cases}
1&\text{if }i=j,\text{ or}\\
[b_{ij}]_-&\text{otherwise},
\end{cases}
\]
where $[x]_-$ means $\min(0,x)$.  
We reuse the notation $E_c$ for the matrix that corresponds to this form, again with simple co-root coordinates on the left and simple root coordinates on the right.
Thus $E_c$ is obtained from $B$ by changing all positive entries to $0$ and then changing the $0$'s on the diagonal to $1$.

\subsection{Scattering diagrams and theta functions}\label{scat sec}
We will construct scattering diagrams with the transposed matrix $\tB^T$ corresponding to the first $n$ rows of the matrix called $[\ep_{ij}]$ in \cite{GHKK}.
This global transpose relative to~\cite{GHKK} makes our construction compatible with the conventions of \cite{ca4}, as explained in the last paragraph of \cite[Section~1]{scatcomb}.
For more about how to translate between the present conventions and the conventions of \cite{GHKK} and \cite{scatfan,canonical}, see \cite[Table~2]{scatcomb}.

A \newword{wall} is a pair $(\d,f_\d)$, where $\d$ is a codimension-$1$ cone $\d$ in $V^*$ and $f_\d$ is in~$\k[[\hy]]$.
The cone $\d$ is orthogonal to a vector $\beta\in Q^+$ that is primitive in $Q$, and the scattering term $f_\d$ is a univariate power series in $\hy^\beta$ with constant term~$1$.
We sometimes write $(\d,f_\d(\hy^\beta))$ as a way of naming $\beta$ explicitly.
A \newword{scattering diagram} is a collection $\D$ of walls, satisfying a finiteness condition that amounts to requiring that all relevant computations are valid as limits of formal power series.

The \newword{transposed cluster scattering diagram} $\Scat^T(\tB)$ associated to $\tB$ is the unique consistent scattering diagram consisting of walls $\sett{(\alpha_i^\perp,1+\hy_i):i=1,\ldots,n}$ together with additional walls, all of which are outgoing.
(The word ``transposed'' and the superscript $T$ serve as a reminder of the global transpose relative to \cite{GHKK}.)
We will not need the definitions of \emph{consistent} and \emph{outgoing} here, but they can be found in \cite[Section~1.1]{GHKK} or, with conventions similar to the conventions of this paper, in \cite[Section~3.3]{affscat} or \cite[Section~2]{scatcomb}.
Instead, we will rely on the explicit construction of the transposed cluster scattering diagram in acyclic affine type, from~\cite{affscat}, which is reviewed in Section~\ref{aff back sec}.

Any consistent scattering diagram cuts out a complete fan in $V^*$ called the \newword{scattering fan} \cite[Theorem~3.1]{scatfan}.
We will not need the precise definition here, but instead it will suffice to understand that the walls of the scattering diagram cut~$V^*$ into cones.
We are particularly interested in the scattering fan for the transposed cluster scattering diagram $\Scat^T(\tB)$, and we write $\ScatFan^T(B)$ and call this the \newword{transposed cluster scattering fan}.
In general, $\ScatFan^T(B)$ is a refinement of the mutation fan $\F_{B^T}$ defined in Section~\ref{basic sec} (see \cite[Theorem~4.10]{scatfan}), although in the affine case that is the subject of the paper, the two coincide \mbox{\cite[Theorem~2.10]{affscat}}.

We now define theta functions, reinterpreting the definition from \cite{GHKK} as in \cite{scatcomb}.
Theta functions are defined in terms of the cluster scattering diagram, whose construction requires $\tB$ to have linearly independent columns, but in Section~\ref{coeffs sec}, we explain how theta functions can be defined with no conditions on $\tB$.
There is a theta function for each pair $(\chi,\lambda)\in V^*\times P$, with $\chi$ not contained in any hyperplane $\beta^\perp$ for $\beta\in Q$.  
By convention, $\thet_{\chi,0}=1$.
For $\lambda\neq0$, $\thet_{\chi,\lambda}$ is defined as follows.

A \newword{broken line} in $\Scat^T(\tB)$, for $\lambda$, with endpoint $\chi$, is a piecewise linear path~$\bl:(-\infty,0]\to V^*$ (with finitely many of domains of linearity), together with an assignment of a monomial $\const_Lx^{\lambda_L}y^{\beta_L}$ (with $\const_L\in\k$ and $(\lambda_L,\beta_L)\in P\oplus Q$) to each domain $L$ of linearity of $\bl$, satisfying the following five conditions.
\begin{enumerate}[label=\rm(\roman*), ref=(\roman*)]
\item \label{brok endpoint}
$\bl(0)=\chi$.
\item \label{brok generic}
$\bl$ does not intersect the relative boundary of any wall of $\Scat^T(\tB)$ and does not pass through any intersection of walls of $\Scat^T(\tB)$ (unless the walls are contained in the same hyperplane). 
\item \label{brok slope}
On each domain $L$ of linearity, $\lambda_L$ is $-\bl'$ (where $\bl'$ is the derivative of $\bl$).
\item \label{brok unbounded}
On the unbounded domain $L$ of linearity, $\const_Lx^{\lambda_L}y^{\beta_L}$ is $x^\lambda$.
\item \label{brok change slope}
Suppose $L_1$ and $L_2$ are domains of linearity, listed in order of increasing parameter in $(-\infty,0]$ and intersecting at a point $p$.
Condition \ref{brok generic} implies that there exists $\beta\ck$ in $Q\ck$ such that every wall containing $p$ is in $\beta^\perp$.
Taking~$\beta\ck$ primitive in $Q\ck$ and ${\br{\lambda_{L_1},\beta\ck}>0}$, let $f$ be the product of the $f_\d$ for all walls $(\d,f_\d)$ with $p\in\d$.
Then $\const_{L_2}x^{\lambda_{L_2}}y^{\beta_{L_2}}$ is $\const_{L_1}x^{\lambda_{L_1}}y^{\beta_{L_1}}$ times a term in $f^{\br{\lambda_{L_1},\beta\ck}}$.
(We say that the broken line $\bl$ \newword{bends at $p$} and \newword{picks up a term in $f^{\br{\lambda_{L_1},\beta\ck}}$}.)
\end{enumerate}

We have deviated from the usual notation $\gamma$ for a broken lines because Greek letters are heavily used in this paper to represent vectors in $V$ and $V^*$; we instead adopted the letter~$\bl$, taking inspiration from the Italian word ``spezzata'' which indicates a piecewise linear curve. 

Writing $\const_\bl x^{\lambda_\bl}y^{\beta_\bl}$ for the monomial on the last domain of linearity of $\bl$, the \newword{theta function} $\thet_{\chi,\lambda}$ is $\sum \const_\bl x^{\lambda_\bl}y^{\beta_\bl}\in x^\lambda \k[[\hy]]$ over broken lines for $\lambda$ with endpoint~$\chi$.

We are concerned with theta functions such that $\chi$ is in the interior of the positive orthant $D=\bigcap_{i=1}^n\sett{v\in V^*: \br{v,\alpha_i}\ge 0}$.
Such a theta function does not depend on the exact choice of $\chi$ in the interior of $D$, so we write $\thet_\lambda$ for $\thet_{\chi,\lambda}$ with $\chi$ in the interior of $D$.
Each $\thet_\lambda$ is $x^\lambda\cdot F_\lambda$ for $F_\lambda\in\k[[\hy]]$.

\begin{remark}\label{useless dimensions}
As already mentioned, our definition of (cluster) scattering diagrams follows \cite{scatfan,scatcomb,affscat,canonical} in leaving out extra dimensions.
The definition of broken lines and theta functions has been adjusted accordingly.
This change in convention adds no essential complications and removes no richness from the construction, but simplifies the story considerably.
(See \cite[Remark~2.12]{scatfan}, \cite[Remark~5.1]{scatfan}, and \cite[Remark~2.3]{canonical}.) 
\end{remark}

We now explain results of \cite{GHKK} on structure constants for multiplication of theta functions $\thet_\lambda$.
These results are rephrased in our setting where unnecessary dimensions have been removed.
For more details (in the untransposed setting), see \cite[Section~2.4]{canonical}.

Take $p_1,p_2,\lambda\in P$ and take $\chi\in V^*$ disjoint from all hyperplanes $\beta^\perp$ for $\beta\in Q$.
Define
\[a_\chi(p_1,p_2,\lambda)=\sum_{(\bl_1,\bl_2)}\const_{\bl_1}\const_{\bl_2} y^{\beta_{\bl_1}+\beta_{\bl_2}}.\]
The sum is over pairs $(\bl_1,\bl_2)$ of broken lines for $p_1$ and $p_2$ respectively with $\lambda_{\bl_1}+\lambda_{\bl_2}=\lambda$, both having endpoint $\chi$.
By the definition of broken lines and by \cite[Definition-Lemma~6.2]{GHKK}, $a_\chi(p_1,p_2,\lambda)$ is well defined as a formal power series in $y_1,\ldots,y_n$ with nonnegative integer coefficients.
For a sequence of points $\chi\in V^*$ approaching~$\lambda$, with each $\chi$ disjoint from all hyperplanes $\beta^\perp$ for $\beta\in Q$, define
\[a(p_1,p_2,\lambda)=\lim_{\chi\to\lambda}a_\chi(p_1,p_2,\lambda).\]
This is a limit of formal power series.
The following is part of \cite[Proposition~6.4]{GHKK} specialized to the present setting.
In \cite{GHKK}, the proposition is stated with the hypotheses of principal coefficients.
The following more general version is \cite[Proposition~2.9]{canonical} and we will extend the proposition further as Proposition~\ref{struct plus plus}.

\begin{proposition}\label{struct}
Suppose $\tB$ has linearly independent columns and $p_1,p_2\in P$.
Then for every $\lambda\in P$, the formal power series $a(p_1,p_2,\lambda)$ does not depend on the sequence of points $\chi$ approaching $\lambda$.
Furthermore,
\begin{equation}\label{struct eq}
\thet_{p_1}\cdot\thet_{p_2}=\sum_{\lambda\in P}a(p_1,p_2,\lambda)\,\thet_\lambda.
\end{equation}
\end{proposition}  
The sum in Proposition~\ref{struct} may not be a finite sum, but in any case, it makes sense as a limit of formal power series, as explained in \cite[Section~2.4]{canonical}.
Moreover, in the cases covered in this paper, the sum is always finite, and furthermore, each $a(p_1,p_2,\lambda)$ is a polynomial rather than a formal power series, as we now explain.

Let $\Theta\subseteq P$ be the set of weights $\lambda\in P$ such that only finitely many broken lines appear in the definition of $\thet_\lambda$.
Thus for $\lambda\in\Theta$, the theta function $\thet_\lambda$ is a polynomial, rather than a formal power series.
Finally, \cite[Proposition~0.7]{GHKK} says, among other things, that if the set of $\g$-vectors of cluster variables associated to~$B$ is not contained in a halfspace, then $\Theta$ is all of $P$.
The following theorem is a part of \cite[Theorem~7.5]{GHKK}, but it has been rephrased to use the conventions of this paper.
\begin{theorem}\label{Theta facts}  
If $p_1p_2\in\Theta$ then the right side of \eqref{struct eq} has finitely many nonzero terms, each $a(p_1,p_2,\lambda)$ is a polynomial, and if $a(p_1,p_2,\lambda)\neq0$, then $\lambda\in\Theta$.
\end{theorem}

\subsection{The choice of coefficients}\label{coeffs sec} 
Following (and applying a transpose to) \cite[Section~2.6]{canonical}, we now describe how, for some exchange matrices $B$ including those of affine type, theta functions can be defined for arbitrary extensions $\tB$ of $B$, with no requirement that the columns of $\tB$ are linearly independent.
We also describe the relationship between theta functions for different extensions of~$B$.

Given two extensions $\tB$ and $\tB'$ of the same exchange matrix $B$, we use the same indeterminates $x_1,\ldots,x_n$ in constructions for $\tB$ and $\tB'$, but use a different set of tropical variables for each, and define the $y_j$ and $\hy_j$ as in Section~\ref{scat sec} and similarly define $y'_j$ and $\hy'_j$ for~$\tB'$.
As explained in \cite[Section~2.6]{canonical}, as a consequence of \cite[Proposition~2.6]{scatfan}, if both $\tB$ and $\tB'$ have linearly independent columns, then $\Scat^T(\tB')$ is obtained from $\Scat^T(\tB)$ by leaving each wall $(\d,f_\d)$ unchanged geometrically, but replacing the $\hy$ by the $\hy'$ in each $f_\d$.
This substitution is unambiguous precisely because the columns of $\tB$ are linearly independent.

In order to unambiguously make the same kinds of substitutions to obtain theta functions and structure constants for $\tB'$ from theta functions and structure constants for $\tB$, we need a condition on $\tB$ that is stronger than linearly independent columns.
We say that $\tB$ has \newword{nondegenerate coefficients} if the submatrix of $\tB$ obtained by deleting the top square matrix $B$ has linearly independent columns.

If $\tB$ has nondegenerate coefficients, then we can use it to define certain theta functions for an arbitrary extension $\tB'$, with no requirement of linearly independent columns.
If $\tB'$ has linearly \emph{de}pendent columns, then for $\lambda\in\Theta$, the \newword{theta function} $\thet'_\lambda$ for $\tB'$ is defined to be the function obtained from the theta function~$\thet_\lambda$ for $\tB$ by replacing unprimed variables by primed variables.  
The following proposition repeats this definition for $\tB'$ with linearly dependent columns and is \cite[Proposition~2.6]{canonical} for $\tB'$ with linearly independent columns.
The proposition also implies that the definition of theta functions for arbitrary $\tB'$ is well posed (i.e.\ independent of the choice of an extension $\tB$ with nondegenerate coefficients).

\begin{proposition}\label{where prin plus}
Suppose $\tB$ and $\tB'$ are extensions of $B$ such that $\tB$ has nondegenerate coefficients.
Assume either that $\tB'$ has linearly independent columns or that $\lambda\in\Theta$.
Then the theta function $\thet'_\lambda$ defined in terms of~$\tB'$ is obtained from the theta function $\thet_\lambda$ defined in terms of~$\tB$ by replacing each $\hy$ by $\hy'$ throughout (or equivalently, replacing the $y$ by the $y'$).
\end{proposition}

Furthermore, the following result \cite[Proposition~2.10]{canonical} turns results on structure constants relative to an extension with nondegenerate coefficients into valid relations on theta functions defined in terms of an arbitrary extension.

\begin{proposition}\label{struct plus plus}
Suppose $\tB$ and $\tB'$ are extensions of $B$ such that $\tB$ has nondegenerate coefficients.
Assume either that $\tB'$ has linearly independent columns or that the vectors $p_1$ and $p_2$ are in~$\Theta$.
Write $a(p_1,p_2,m)$ for the structure constants for theta functions defined in terms of $\tB$.
Replacing each~$\thet$ by~$\thet'$ in \eqref{struct eq} and replacing each $y_i$ by $y'_i$ in each $a(p_1,p_2,m)$ yields a valid relation among theta functions $\thet'$ defined in terms of $\tB'$.
\end{proposition}

We emphasize that for the exchange matrices $B$ of affine type that are the subject of this paper, $\Theta=P$, so Proposition~\ref{where prin plus} applies to define theta functions $\thet_\lambda$ for all $\lambda$ for arbitrary extensions of $B$.
Furthermore, Proposition~\ref{struct plus plus} applies to convert all structure constant computations made under the assumption of nondegenerate coefficients to valid relations among theta functions for arbitrary extensions of $B$.

In order to use results of \cite{canonical}, we will need a requirement on $\tB$ that is stronger than nondegenerate coefficients.
We say that $\tB$ has \newword{signed-nondegenerating coefficients} if, for every sequence~$\kk$ of indices (including the empty sequence), $\mu_\kk(\tB)$ has nondegenerate coefficients and every column of the submatrix below $\mu_\kk(B)$ has a sign, meaning that it consists of either nonnegative entries or nonpositive entries.
For any $B$, there exists an extension $\tB$ with signed-nondegenerating coefficients, namely principal coefficients (in light of ``sign-coherence of $C$-vectors'', conjectured in \cite{ca4} and proved in \cite[Corollary~5.5]{GHKK}).
If $\tB$ has signed-nondegenerating coefficients, then for each ${i\in\set{1,\ldots,n}}$, we write $\sgn(y_j)\in\set{\pm1}$ for the sign (nonnegative or nonpositive) of the exponents appearing on $y_j$.

\subsection{The cluster monomials and $\g$-vectors}\label{clus sec}
The relevance of cluster scattering diagrams to cluster algebras is rooted in the fact that cluster monomials can essentially be obtained as theta functions.
However, there is some subtlety about coefficients.
A \newword{cluster} is a set of $n$ variables obtained from the initial cluster $\set{x_1,\ldots,x_n}$ by a sequences of mutations in the usual sense, and the elements of clusters are called \newword{cluster variables}. 
Importantly, our clusters do \emph{not} include include frozen variables.
Since we have left out frozen variables in our definition of a cluster, we define a \newword{cluster monomial} to be a monomial (ordinary, not Laurent) in the entries in some cluster.
Each theta function $\thet_\lambda$ is a Laurent polynomial in~$\set{x_1,\ldots,x_n}\cup\set{u_1,\ldots,u_m}$.
Define $\Clear(\thet_\lambda)$ to be $\thet_\lambda\prod_{i=1}^mu_i^{q_i}$, where each $q_i$ is chosen as small as possible so that $u_i$ is not in the denominator of any term in~$\thet_\lambda\prod_{i=1}^mu_i^{q_i}$.

The \newword{$\g$-vector fan} $\gFan(B)$ associated to $B$ is the set of cones spanned by the $\g$-vectors of clusters in the sense of \cite{ca4}, and all faces of those cones.
We will say \newword{$\g$-vector cone} for a cone in the $\g$-vector fan and \newword{maximal $\g$-vector cone} for the nonnegative span of the $\g$-vectors of a cluster.
We write $|\gFan(B)|$ for the support of the $\g$-vector fan (the union of all its cones).
The following theorem is a special case of results of \cite{GHKK}, expanded to allow arbitrary extensions of $\tB$ as explained in Section~\ref{coeffs sec}.
Under the hypothesis that $\tB$ has linearly independent columns, the theorem is the combination of \cite[Theorem~5.2]{scatfan} and \cite[Corollary~2.6]{scatcomb}, both of which are restatements of results of~\cite{GHKK}.
The theorem applies to arbitrary extensions because every $\g$-vector of a cluster monomial is in $\Theta$ and by \cite[Corollary~6.3]{ca4}.

\begin{theorem}\label{clus mon thm}    
Suppose $\tB$ is an extended exchange matrix. 
The map $\lambda\mapsto\Clear(\thet_\lambda)$ is a bijection from $P\cap|\gFan(B)|$ to the set of cluster monomials determined by $\tB$, and the $\g$-vector of $\thet_\lambda$ is~$\lambda$.
There is a bijection from rays of $\gFan(B)$ to cluster variables determined by $\tB$ sending each ray to $\Clear(\thet_\lambda)$ such that~$\lambda$ is the shortest vector in $P$ contained in the ray.
\end{theorem}

\begin{remark}\label{inescapable}  
One wants $\lambda\mapsto\thet_\lambda$ to be bijection from $P\cap|\gFan(B)|$ to the set of cluster monomials determined by $\tB$, and that is the case when $\tB$ has principal coefficients, in the coefficient-free case $\tB=B$, and more generally whenever the rows of $\tB$ below $B$ have no negative entries.
However, the correct bijection in general is $\lambda\mapsto\Clear(\thet_\lambda)$.
The factor $\prod_{i=1}^mu_i^{q_i}$ used to define $\Clear(\thet_\lambda)$ is the reciprocal of the denominator of \cite[(6.5)]{ca4}.

We emphasize that this extra complication is \emph{not} an artifact of our conventions that treat frozen variables as coefficients.  
In \cite[Definition~4.8]{GHKK}, cluster monomials are defined as monomials in the frozen and unfrozen variables, and theta functions~$\thet_\lambda$ are defined for $\lambda$ in an $(n+m)$-dimensional superlattice of $P$.
In that setting $\lambda\mapsto\thet_\lambda$ is a bijection from the lattice points in an $(n+m)$-dimensional version of the $\g$-vector fan to the set of cluster monomials.
However, when $\tB$ does not have principal coefficients, the $\g$-vector of $\thet_\lambda$ (in $\integers^{n+m})$ might not be $\lambda$.  
Instead, there is a correction, corresponding to the operator $\Clear$ in Theorem~\ref{clus mon thm}.
\end{remark}

We close this section with the following observation, which will be useful in Section~\ref{per sec} and which is, in essence, a result of~\cite{GHKK}.
\begin{lemma}\label{z z'}
For fixed $p_1,p_2,\lambda\in P$, if $\chi$ and $\chi'$ are in the interior of the same maximal $\g$-vector cone, then $a_\chi(p_1,p_2,\lambda)=a_{\chi'}(p_1,p_2,\lambda)$.
\end{lemma}
\begin{proof}
As an immediate consequence of \cite[Theorem~3.5]{GHKK}, since $\chi$ and~$\chi'$ are in the interior of the same maximal $\g$-vector cone, the set of broken lines to $\chi$ from direction $p_1$ has the same multiset of final monomials as the the set of broken lines to $\chi'$ from direction $p_1$.
Similarly for $p_2$.
\end{proof}

\subsection{Mutation of cluster scattering diagrams and theta functions}\label{mut sec}
We now discuss the notion of mutation of theta functions that is relevant to this paper.
This notion is a version of the mutation operation that is central to~\cite{GHKK}.
For a comparison of the two notions, see \cite[Remark~4.1]{canonical}.

Suppose $\tB=[b_{ij}]$ has signed-nondegenerating coefficients.
Define $x_1,\ldots x_n$, $y_1,\ldots,y_n$, and $\hy_1,\ldots,\hy_n$ as above, and let $k\in\set{1,\ldots,n}$.
Let $x'_1,\ldots,x'_n$ be the cluster variables in the seed obtained by mutating at $k$, define $y'_1,\ldots,y'_n$ analogously in terms of $\mu_k(\tB)$, and define $\hy'_1,\ldots,\hy'_n$ in terms of $B'$ and $x'_1,\ldots,x'_n$ (using the same tropical variables $u_i$ for the added rows of $\mu_k(\tB)$ as for $\tB$).
Matrix mutation and the exchange relations imply the following relations between the primed and unprimed variables.
(See \cite[Proposition~3.9]{ca4} and its proof.)
\begin{align}
\label{exch rel}
x_kx'_k&=(1+\hy_k)y_k^{-[-\sgn(y_k)]_+}\prod_{i=1}^n(x_i)^{[-b_{ik}]_+}\\
x_i&=x'_i\quad\text{for }i\neq k\\
y'_k&=y_k^{-1}\\
y'_j&=y_j(y_k)^{[\sgn(y_k)b_{kj}]_+}\quad\text{for }j\neq k\\
\hy'_k&=\hy_k^{-1}\\
\label{hy mut}
\hy'_j&=\hy_j(\hy_k)^{[b_{kj}]_+}(1+\hy_k)^{-b_{kj}}\quad\text{for }j\neq k.
\end{align}

To describe mutation of theta functions, we will use the language of mutation maps from Section~\ref{basic sec}.
The following theorem is \cite[Theorem~4.2]{canonical}, rephrased in the transposed setting, and is a version of \cite[Proposition~3.6]{GHKK}.

\begin{theorem}\label{2 muts}
Suppose $\tB$ has signed-nondegenerating coefficients.  
For $\lambda\in P$ and $k\in\set{1,\ldots,n}$, write $\thet_{\lambda}^\tB$ for a theta function in the unprimed variables using $\tB$, write $\lambda'=\eta_k^{B^T}\!(\lambda)$, and write $\thet_{\lambda'}^{\mu_k(\tB)}$ for a theta function defined in the primed variables using $\mu_k(\tB)$.
Relating the primed and unprimed variables as in \eqref{exch rel}--\eqref{hy mut}, we have $\thet_{\lambda'}^{\mu_k(\tB)}=\thet_{\lambda}^\tB\cdot(y_k)^{-[\sgn(y_k)\br{\lambda,\alpha_k\ck}]_+}$.
\end{theorem}
In interpreting Theorem~\ref{2 muts}, it is useful to remember that $y'_k=y_k^{-1}$, so that the conclusion of the theorem is equivalent to $\thet_{\lambda'}^{\mu_k(\tB)}=\thet_{\lambda}^\tB\cdot(y'_k)^{[\sgn(y_k)\br{\lambda,\alpha_k\ck}]_+}$.

The heart of the proof of Theorem~\ref{2 muts} in \cite{canonical} is a lemma (analogous to a proposition in~\cite{GHKK}) that describes how broken lines mutate.  
To state the lemma, we define a map on piecewise-linear curves $\bl:(-\infty,0]\to V^*$ with finitely many domains of linearity, each labeled by a constant in~$\k$ times a monomial in $x_1,\ldots,x_n$ andr~ $y_1,\ldots,y_n$, with the infinite domain labeled by a monomial in $x_1,\ldots,x_n$.
Suppose $\bl$ is such a curve, with infinite domain labeled by $x^\lambda$.
Ignoring the monomials, the map sends $\bl$ to $\eta_k^{B^T}\circ\bl$, so we re-use the name $\eta_k^{B^T}$ for the map on labeled curves.
By passing to smaller domains of linearity, we can assume that no domain of $\bl$ or $\eta_k^{B^T}\circ\bl$ crosses~$e_k^\perp$.
In particular, each domain $L'$ of $\eta_k^{B^T}\circ\bl$ is $\eta_k^{B^T}(L)$ where $L$ is some domain of $\bl$.
Starting with the monomial $\const_Lx^{\lambda_L}y^{\beta_L}$ labeling $L$, we obtain the label for $L'$ by substitution and multiplication as follows.
Simultaneously make the following substitutions: 
\begin{alignat*}{2}
\text{replace }&x_k&&\text{ by}\quad\!\!\!\!\!
\begin{cases}\displaystyle
(x_k')^{-1}(y'_k)^{[-\sgn(y_k)]_+}\prod_{i=1}^n(x_i')^{[-b_{ik}]_+}&\!\!\!\text{if }L\subseteq \sett{p\in V^*:\br{p,\alpha_k}\le0}\\\displaystyle
(x_k')^{-1}(y'_k)^{-[\sgn(y_k)]_+}\prod_{i=1}^n(x_i')^{[b_{ik}]_+}&\!\!\!\text{if }L\subseteq \sett{p\in V^*:\br{p,\alpha_k}\ge0}
\end{cases}\\
\text{replace }&x_i&&\text{ by }\,\, x'_i \quad\text{for }i\neq k\\
\text{replace }&y_k&&\text{ by }\, (y_k')^{-1}\\
\text{replace }&y_j&&\text{ by }\,\, y'_j(y'_k)^{[\sgn(y_k)b_{kj}]_+}  \quad\text{for }j\neq k.
\end{alignat*}
Then multiply by $(y'_k)^{[\sgn(y_k)\br{\lambda,\alpha\ck_k}]_+}$.
(Note that we multiply by a monomial that depends only the monomial labeling the infinite domain of linearity of $\bl$, independent of $L$.)

The following is \cite[Lemma~4.9]{canonical}, rewritten in the setting of this paper.  
(See also \cite[Proposition~3.6]{GHKK}.)

\begin{lemma}\label{mut broken line}
Suppose $\tB$ has signed-nondegenerating coefficients and take $\lambda\in P$ and $\chi\in V^*$.
Then $\bl$ is a broken line, relative to $\Scat^T\!(\tB)$, for $\lambda$ with endpoint $\chi$ if and only if $\eta_k^{B^T}(\bl)$ is a broken line, relative to $\Scat^T\!(\mu_k(\tB))$, for $\eta_k^{B^T}(\lambda)$ with endpoint~$\eta_k^{B^T}(\chi)$.
\end{lemma}

Since we will need to mutate pairs of broken lines in connection with Proposition~\ref{struct}, we use the shorthand $\eta^{B^T}_\kk(\bl_1,\bl_2)$ for $\bigl(\eta^{B^T}_\kk(\bl_1),\eta^{B^T}_\kk(\bl_2)\bigr)$.

At each step, the mutation map $\eta_\kk^{B^T}$ applies one of two linear maps, depending on which side of $\alpha_{k_i}^\perp$ the output of $\eta^{B^T}_{k_{i-1}\cdots k_1}$ is on.
Given two vectors, if the same case applies to both vectors at every step, we say that the two vectors are in the same \newword{domain of definition} of $\eta_\kk^{B^T}$.
The following is \cite[Proposition~5.3]{canonical}, rewritten in our transposed conventions.

\begin{proposition}\label{mut pair}
Suppose $\tB$ has signed-nondegenerating coefficients, suppose $\lambda\in P$, suppose $\chi\in V^*$ is not contained in any wall of $\Scat^T(\tB)$, and let $\kk$ be a sequence of indices.
If $\lambda$ and $\chi$ are in the same domain of definition of $\eta_\kk^{B^T}$, then a pair $(\bl_1,\bl_2)$ of broken lines contributes to $a_\chi(p_1,p_2,\lambda)$ if and only if $\eta_\kk^{B^T}(\bl_1,\bl_2)$ contributes to $a_{\eta_\kk^{B^T}(\chi)}(\eta_\kk^{B^T}(p_1),\eta_\kk^{B^T}(p_2),\eta_\kk^{B^T}(\lambda))$.
\end{proposition}

\subsection{Acyclic affine type}\label{aff back sec}
We now summarize background material on the acyclic affine case.
Additional information is in \cite{afforb,affdenom,affscat}.
We fix an acyclic exchange matrix $B$ of affine type and the associated root system $\RS$ and Coxeter element~$c$.
There is an imaginary root $\delta$ and a finite root system $\RSfin\subset\RS$ such that every real root in $\RS$ is a positive scalar multiple of $\beta_0+k\delta$ for some $\beta_0\in\RSfin$ and $k\in\integers$.
(See \cite[Theorem~5.6]{Kac} and \cite[Proposition~6.3]{Kac}.) 

\subsubsection{The affine generalized associahedron fan}\label{agaf sec}
In \cite{affdenom}, a set $\AP{c}\subseteq\RS$ is constructed and a complete, infinite, simplicial fan $\Fan_c(\RS)$ is defined, with the rays of $\Fan_c(\RS)$ being precisely the rays spanned by roots in~$\AP{c}$ and the cones being the nonnegative spans of the sets of pairwise compatible roots for a certain compatibility relation called \newword{$c$-compatibility}.
A \newword{$c$-cluster} is a maximal set of pairwise $c$-compatible roots in~$\AP{c}$, and the maximal cones of $\Fan_c(\RS)$ are the spans of the $c$-clusters.
We will not need the details of the construction of the set $\AP{c}$, and we will only need some of the details of $c$-compatibility relation.

The unique imaginary root in $\AP{c}$ is $\delta$.
A $c$\nobreakdash-cluster is \newword{imaginary} if it contains~$\delta$ and \newword{real} if it does not.
Real clusters have $n$ elements and imaginary $c$-clusters have $n-1$ elements.
We write $\Fan_c^\re(\RS)$ for the subfan of $\Fan_c(\RS)$ consisting of cones not containing $\delta$.
Thus $\Fan_c(\RS)$ is the union of $\Fan_c^\re(\RS)$ and the \newword{star} of $\delta$ in $\Fan_c(\RS)$ (the set of cones of $\Fan_c(\RS)$ containing $\delta$).
Their intersection is the link of $\delta$ i.e. the set of cones $C$ of $\Fan_c(\RS)$ not containing~$\delta$ such that $C\cup\set\delta$ spans a cone in $\Fan_c(\RS)$.
The union of the cones in the star of~$\delta$ is a codimension-$1$ cone, and the relative interior of that cone is the complement of the subfan $\Fan_c^\re(\RS)$ in~$V$.

The set $\APT{c}$ is the set of roots in $\AP{c}$ that are in finite $c$-orbits.  
This set includes~$\delta$, which is fixed by the action of $c$.
We write $\APTre{c}$ for $\APT{c}\setminus\set{\delta}$, the set of real roots in~$\APT{c}$.
There is a unique subset $\SimplesT{c}$ of $\APT{c}$ that is minimal with respect to the property that every root in $\APT{c}$ is in the nonnegative linear span of $\SimplesT{c}$.
This subset does not contain $\delta$, or in other words, $\SimplesT{c}\subseteq\APTre{c}$.
The set $\SimplesT{c}$ decomposes into $1$, $2$, or $3$ $c$-orbits.  
Take some arbitrary numbering of the orbits as $\SimplesT{c}_o$ for $o\in\set{1}$ or $o\in\set{1,2}$ or $o\in\set{1,2,3}$.
Roots in $\SimplesT{c}$ that are in different $c$-orbits are orthogonal, and this is the finest possible partition of $\SimplesT{c}$ into mutually orthogonal subsets.  
Let $\APTre{c;o}$ be the subset of $\APTre{c}$ consisting of real roots in the nonnegative span of~$\SimplesT{c}_o$.
The sets $\APTre{c;o}$ and $\APTre{c;o'}$ are disjoint when $o\neq o'$, and $\APTre{c}$ is the union of the $\APTre{c;o}$.

Given a $c$-orbit $\SimplesT{c}_o$ with $|\SimplesT{c}_o|=k$ , choose a root in $\SimplesT{c}_o$ to name $\beta_{[0]}$, let $\beta_{[1]}=c\beta_{[0]}$ and continue cyclically, modulo the size of the orbit.
The roots in $\APTre{c;o}$ are $\beta_{[i,j]}=\beta_{[i]}+\beta_{[i+1]}+\cdots\beta_{[j]}$ with $i\le j$ and $j-i+1<k$.
This is the unique expression for $\beta_{[i,j]}$ as a linear combination of roots in $\SimplesT{c}$.
For a root $\phi=\beta_{[i,j]}$ write $\SuppT(\phi)$ for the set $\set{\beta_{[i]},\beta_{[i+1]},\ldots,\beta_{[j]}}$ of roots that appear with nonzero coefficients in the expression.
Viewing $\SimplesT{c}_o$ as a cycle, $\SuppT(\phi)$ is a path in that cycle.
(There is a different notion of support that one could define, namely the support of a root as a linear combination of simple roots, but that notion will not appear in this paper, so we will often refer to $\SuppT(\phi)$ simply as the support of~$\phi$.)

The vector $\delta$ does not have a well-defined support if~$\SimplesT{c}$ consists of more than one $c$-orbit, because the sum of any one $c$-orbit in $\SimplesT{c}$ is~$\delta$.

Suppose $\gamma,\gamma'\in\APTre{c}$.
The two roots are \newword{nested} if $\SuppT(\gamma)\subseteq\SuppT(\gamma')$ or $\SuppT(\gamma')\subseteq\SuppT(\gamma)$.
They are \newword{spaced} if ${\SuppT(\gamma)\cup\SuppT(c\gamma)\cup\SuppT(c^{-1}\gamma)}$ is disjoint from $\SuppT(\gamma')$.
If $\gamma\in\APTre{c;o}$ and $\gamma'\in\APTre{c;o'}$ with $o\neq o'$, then $\gamma$ and $\gamma'$ are spaced.
If they are both in the same $\APTre{c;o}$, then they are spaced if and only if their supports are disjoint and there is at least one root in $\SimplesT{c}_o$ between them on the cycle, on each side.
The following is the concatenation of \cite[Proposition~5.6]{affdenom} and \mbox{\cite[Proposition~5.12]{affdenom}}.

\begin{proposition}\label{tube compat}
A root $\gamma\in\APre{c}$ is compatible with $\delta$ if and only if $\gamma\in\APTre{c}$.
Two distinct roots $\gamma,\gamma'\in\APTre{c}$ are $c$-compatible if and only if they are nested or spaced.
\end{proposition}

Let $J_o$ be a maximal set of pairwise compatible real roots in $\APTre{c;o}$ for some~$o$.
The number of roots in $J_o$ is $|\SimplesT{c}_o|-1$.
There is a unique \newword{maximal root} in~$J_o$ whose support has size $|J_o|$.
The support of the maximal root in $J_o$ contains the support of all other roots in~$J_o$.
Given $\gamma\in J_o$, we say that $\phi\in J_o$ is the \newword{next larger root} from $\gamma$ in $J_o$ if $\phi$ has minimal support among roots in $J_o$ whose support contains $\SuppT(\gamma)$.
Every root in $J_o$ except the maximal root has a next larger root.
We say that $\phi$ is a \newword{next smaller root} from~$\gamma$ in $J_o$ if $\phi$ has maximal support among roots in $J_o$ whose support is contained in $\SuppT(\gamma)$.
The root $\gamma$ may have zero, one, or two next smaller roots.
A root $\phi$ is a next smaller root from~$\gamma$ if and only if $\gamma$ is the next larger root from $\phi$.

Choose a root $\gamma$ in a maximal pairwise compatible set $J_o$ and let $J_o'$ be the unique maximal set of pairwise compatible real roots in $\APTre{c;o}$ of the form $\bigl(J_o\setminus\set{\gamma}\bigr)\cup\set{\gamma'}$ with $\gamma'\neq \gamma$.
We say that $J_o'$ is obtained by \newword{exchanging}~$\gamma$ from $J_o$.

\subsubsection{The isomorphism $\nu_c$}\label{nuc sec}
There is a piecewise-linear map $\nu_c$ from $V$ to $V^*$ that induces an isomorphism of fans \cite[Theorem~2.9]{affscat} from $\Fan_c(\RS)$ to the mutation fan $\F_{B^T}$, which, as already mentioned, coincides with the cluster scattering fan in affine type.
The isomorphism restricts to an isomorphism of fans from $\Fan_c^\re(\RS)$ to the $\g$-vector fan $\gFan(B)$ \cite[Theorem~1.1(1)]{affdenom}.
Also, $\nu_c$ restricts to a bijection from the lattice $Q$ in $V$ spanned by the simple roots $\alpha_1,\ldots,\alpha_n$ to the lattice~$P$ in $V^*$ spanned by the fundamental weights $\rho_1,\ldots,\rho_n$.
We don't need the full definition of $\nu_c$ here (see \cite[Section~9.1]{affdenom}), but it will be useful to know that if $\phi$ is a positive root, then $\nu_c(\beta)=-E_c(\,\cdot\,,\beta)$.

The map $\nu_c$ is linear on $V\setminus|\Fan_c^\re(\RS)|$, so the isomorphism also restricts to a linear isomorphism from the star of $\delta$ in $\Fan_c(\RS)$ to the star of $\nu_c(\delta)$ in $\F_{B^T}$.
The extreme rays of the star of $\delta$ in $\Fan_c(\RS)$ are spanned by the roots in $\SimplesT{c}$.

The union of the cones in the star of $\nu_c(\delta)$ in $\F_{B^T}$ is a cone of dimension $n-1$ and is in fact a wall of the transposed cluster scattering diagram that we call the \newword{imaginary wall}~$\d_\infty$.
The imaginary wall is also the nonnegative linear span of~$\sett{\nu_c(\beta):\beta\in\SimplesT{c}}$.
The \newword{imaginary ray} is the ray spanned by $\nu_c(\delta)$.
An \newword{imaginary cone} is a cone of $\F_{B^T}$ that has the imaginary ray as an extreme ray.
In particular, every imaginary cone is contained in $\d_\infty$.
The maximal imaginary cones are indexed by maximal sets $J$ of pairwise $c$-compatible roots in $\APTre{c}$.
Choosing such a $J$ amounts to choosing, for each orbit in $\SimplesT{c}$, a maximal set $J_o$ of $c$-compatible roots in $\APTre{c;o}$, and taking the union.
The imaginary cone indexed by $J$ is the nonnegative span of $\nu_c(J\cup\set\delta)$.
Thus the combinatorics of imaginary cones in $\F_{B^T}$ is the combinatorics of a product of simplicial cyclohedra, in the sense of \cite{BottTaubes,CarrDevadoss,Markl}.

We present five propositions about $\nu_c$.
The first is a restatement of \cite[Proposition~2.16]{affdenom}, the second is a part of \cite[Proposition~7.14]{affscat} that has been rephrased in terms of $\nu_c$, and the third is \cite[Lemma~7.16]{affscat}.

\begin{proposition}\label{E delta in tubes}
If $\phi\in\APT{c}$, then $\br{\nu_c(\phi),\delta}=0$ and $\br{\nu_c(\delta),\phi}=0$.
\end{proposition}

\begin{proposition}\label{E omega delta in tubes}
If $\phi\in\AP{c}$, then $\br{\nu_c(\delta),\phi}=0$ if and only if $\phi\in\APT{c}$.
\end{proposition}

\begin{proposition}\label{rescue}
Suppose that $\beta\in\APTre{c}$ and that $\phi\in\APre{c}\setminus\APTre{c}$.
If $\omega_c(\phi,\beta)>0$, then ${\d_\infty\subseteq\set{x\in V^*:\br{x,\phi}\le0}}$.
\end{proposition}

\begin{proposition}\label{nuc phip phi}
If $\phi,\phi'\in\APTre{c}$, then 
\[\br{\nu_c(\phi'),\phi\ck}=\bigl|\SuppT(\phi)\cap c(\SuppT(\phi'))\bigr|-\bigl|\SuppT(\phi)\cap\SuppT(\phi')\bigr|.\]  
In particular, when $\phi$ and $\phi'$ are in different $c$-orbits, $\br{\nu_c(\phi'),\phi\ck}=0$.
Furthermore, $\br{\nu_c(\phi),\phi\ck}=-1$.
\end{proposition}
\begin{proof}
Since $\phi'$ is a positive root, $\nu_c(\phi')=-E_c(\,\cdot\,,\phi')$, and thus $\br{\nu_c(\phi'),\phi\ck}=-E_c(\phi\ck,\phi')$.
If $\beta,\beta'\in\SimplesT{c}$, then \cite[Proposition~2.17]{affdenom} can be restated as 
\[
\br{\nu_c(\beta'),\beta\ck}=
\begin{cases}
-1&\text{if }\beta'=\beta,\\
1&\text{if }c\beta'=\beta,\text{ or}\\
0&\text{otherwise.}
\end{cases}
\]
The proposition follows.
\end{proof}

\begin{proposition}\label{pairing with simple}
Suppose $c=s_1\cdots s_n$.
For all $k=1,\ldots,n$, 
\[\brrr{\nu_c(\delta)+\omega_c(\,\cdot\,,\sum_{i=1}^{k-1}\br{\rho_i\ck,\delta}\alpha_i),\alpha_k\ck}=-\br{\rho_k\ck,\delta}.\]
\end{proposition}
\begin{proof}
The left side is 
$-E_c(\alpha_k\ck,\delta)+\omega_c(\alpha_k\ck,\sum_{i=1}^{k-1}\br{\rho_i\ck,\delta}\alpha_i)$.
By the definition of~$E_c$ and because $\omega_c(\alpha_k\ck,\alpha_k)=0$, we can rewrite this as 
\[-E_c(\alpha_k\ck,\sum_{i=1}^k\br{\rho_i\ck,\delta}\alpha_i)+\omega_c(\alpha_k\ck,\sum_{i=1}^k\br{\rho_i\ck,\delta}\alpha_i).\]
Since $\omega_c=E_c-E_{c^{-1}}$, this becomes $-E_{c^{-1}}(\alpha_k\ck,\sum_{i=1}^k\br{\rho_i\ck,\delta}\alpha_i)=-\br{\rho_k\ck,\delta}$.
\end{proof}

\begin{proposition}\label{crucial lemma}
$\omega_c(\,\cdot\,,\phi)=-\nu_c((1+c^{-1})\phi)$ for any vector $\phi\in V$ such that $(1+c^{-1})\phi$ has nonnegative simple root coordinates.
\end{proposition}
\begin{proof}
We use the matrices corresponding to these bilinear forms, as explained in Section~\ref{basic sec} and quote \cite[Theorem~2.1]{Howlett}, which says that $-E_{c^{-1}}^{-1}E_c$ is the matrix for $c$ in the basis of simple roots.
Thus ${B\phi=E_c(1+c^{-1})\phi=-\nu_c((1+c^{-1})\phi)}$.
\end{proof}

\subsubsection{The affine cluster scattering diagram and fan}\label{scatfan aff sec}
The roots $\AP{c}$ can also be used to construct the transposed cluster scattering diagram $\Scat^T(\tB)$ for $B$ of acyclic affine type.
The construction in \cite{affscat} assumes principal coefficients at the initial seed.
However, we can use that explicit construction more generally whenever $\tB$ has linearly independent columns, as explained in \cite[Remark~2.1]{affscat}.
There is one wall in $\Scat^T(\tB)$ normal to each positive root in $\AP{c}$, and these are all of the walls in $\Scat^T(\tB)$.
We will need only enough detail about these walls to prove Lemma~\ref{hopefully}, below.

Given $\beta,\phi\in\RSpos$, let $\RS'(\beta,\phi)$ be the subset of $\RS$ consisting of roots in the linear span of $\beta$ and $\phi$.
There is a unique pair of roots $\alpha,\alpha'\in\RS'(\beta,\phi)$ such that all positive roots in $\RS'(\beta,\phi)$ are in the nonnegative linear span of $\alpha$ and $\alpha'$.
We say that $\phi$ \newword{cuts} $\beta$ if $\phi\in\set{\alpha,\alpha'}$ but $\beta\not\in\set{\alpha,\alpha'}$.
When $\phi$ cuts $\beta$, the height of $\phi$ (the sum of its simple root coordinates) is less than the height of $\beta$.
As a consequence, the transitive closure of the cutting relation on positive roots is acyclic.
For the same reason, the set $\cut(\beta)$ of positive roots $\phi$ that cut $\beta$ is finite.
Define
\begin{equation}\label{d beta}
\d_\beta=\sett{x\in V^*:\br{x,\beta}=0\text{ and }\br{x,\phi}\le0,\,\forall\phi\in\cut(\beta)\text{ with }\omega_c(\phi,\beta)>0}.
\end{equation}
According to  \cite[Theorem~2.4]{affscat}, the walls of $\Scat^T(\tB)$ are the cones $\d_\beta$, with scattering terms $1+\hy^\beta$ except when $\beta=\delta$, in which case the scattering term is given by \cite[(2.1)]{affscat}.

Suppose $\beta\in\APre{c}$ and suppose $y$ is in the relative interior of an $(n-2)$-dimensional face $F$ of $\d_\beta$ and $\br{y,\delta}>0$.
Then there are two roots $\alpha$ and $\alpha'$ that cut $\beta$ such that~$\alpha^\perp$ and $(\alpha')^\perp$ define $F$ as a face of $\d_\beta$.
Now \cite[Theorem 9.8]{typefree} says that near $y$, the $\g$-vector fan looks like the $\g$-vector fan of a $2\times2$ exchange matrix of finite type.
Since the $\g$-vector fan is a subfan of the cluster scattering fan, in particular, both~$\d_\alpha$ and $\d_{\alpha'}$ contain $y$.
The following lemma is a weak version of this observation.

\begin{lemma}\label{last lemma}
If $\beta\in\APre{c}$ and $y$ is in the relative interior of an $(n-2)$-dimensional face $F$ of $\d_\beta$ and $\br{y,\delta}>0$, then there exists $\alpha\in\APre{c}$ that cuts $\beta$, with $y\in\d_\alpha$.
\end{lemma}

The bilinear form $\omega_c$ can be used to describe the imaginary wall $\d_\infty=\d_\delta$.
Rephrasing \cite[Corollary~4.9]{afframe}, we have
\[\d_\infty=\sett{x\in V^*:\br{x,\delta}=0,\br{x,\beta}\le0,\,\forall \beta\in\RSfin\text{ s.t. }\omega_c(\beta,\delta)>0}.\]
We define two full-dimensional open cones in $V^*$ that we will think of as the set of ``points on the positive side of $\d_\infty$'' or respectively the ``negative side''.
\begin{align*}
\d_\infty^+&=\sett{x\in V^*:\br{x,\delta}>0,\br{x,\beta}<0,\,\forall \beta\in\RSfin\text{ s.t. }\omega_c(\beta,\delta)>0}\subseteq V^*\\
\d_\infty^-&=\sett{x\in V^*:\br{x,\delta}<0,\br{x,\beta}<0,\,\forall \beta\in\RSfin\text{ s.t. }\omega_c(\beta,\delta)>0}\subseteq V^*
\end{align*}

Suppose $C$ is an imaginary cone in $\F_{B^T}$.
Recall that $\d_\infty$ is the union of all imaginary cones in $\F_{B^T}$ and that every imaginary cone, by definition, contains~$\nu_c(\delta)$.
Since $\F_{B^T}$ coincides with the cluster scattering fan, it is cut out by walls.
Thus~$C$ is defined, as a subset of $\d_\infty$, by inequalities of the form $\br{x,\phi}\le0$ for roots $\phi\in\pm\APre{c}$ such that $\br{\nu_c(\delta),\phi}=0$.
Thus, by Proposition~\ref{E omega delta in tubes}, $C$ is defined, as a subset of~$\d_\infty$, by inequalities of the form $\br{x,\phi}\le0$ for roots $\phi\in\pm\APTre{c}$.

\begin{lemma}\label{hopefully}
Suppose $C$ is a maximal imaginary cone in $\F_{B^T}$.
Let $\Gamma$ be the set $\sett{\phi\in\pm\APTre{c}:\br{x,\phi}\le0\,\forall x\in C}$, so that $C=\sett{x\in\d_\infty:\br{x,\phi}\le0\,\forall\phi\in\Gamma}$.
If $\beta\in\APTre{c}$, then $\d_\beta$ does not intersect ${\sett{x\in\d_\infty^+:\br{x,\phi}<0\,\forall\phi\in\Gamma}}$.
\end{lemma}
\begin{proof}
We argue by induction on the transitive closure of the cutting relation on~$\APTre{c}$.

Suppose for the sake of contradiction that there exists $\beta\in\APTre{c}$ and a point $x_0\in\d_\beta\cap{\sett{x\in\d_\infty^+:\br{x,\phi}<0\,\forall\phi\in\Gamma}}$.
Since $\d_\beta\subseteq\beta^\perp$, we see in particular that~$\beta\not\in\pm\Gamma$.
Thus by the definition of $\Gamma$, we see that $\br{x,\beta}\le0$ fails for some $x\in C$ and also $\br{x,-\beta}\le0$ fails for some $x\in C$.
Since $C$ is a convex cone, we conclude that~$\beta^\perp$ intersects the relative interior of $C$.
Let~$x_1$ be a point in the relative interior of $C$ that is also in $\beta^\perp$.
Then also every point~$y$ in the line segment $\seg{x_0x_1}$ has $\br{x,\phi}<0$ for all $\phi\in\Gamma$.
Thus every point $\seg{x_0x_1}$, except $x_1$, is in ${\sett{x\in\d_\infty^+:\br{x,\phi}<0\,\forall\phi\in\Gamma}}$.

Since $C$ is a cone in $\F_{B^T}$ and since $\F_{B^T}$ is the scattering fan, cut out by the walls $\d_\phi$ for $\phi\in\AP{c}$, we know that $\d_\beta$ does not intersect the relative interior of~$C$.
Therefore, the line segment $\seg{x_0x_1}$ passes through the relative boundary of $\d_\beta$ at some point $y\neq x_1$.
By choosing $x_0$ and $x_1$ sufficiently generically, we can assume that $y$ is in the relative interior of an $(n-2)$-dimensional face of~$\d_\beta$.
The relative boundary of $\d_\beta$ is defined by hyperplanes orthogonal to certain roots $\phi\in\APre{c}$ that cut $\beta$, as described specifically in~\eqref{d beta}.
If $\beta$ is minimal in the transitive closure of the cutting relation on $\APTre{c}$, then we have reached a contradiction.
Otherwise, Lemma~\ref{last lemma} says that there exists $\alpha\in\APre{c}$ that cuts $\beta$, with $y\in\d_\alpha$.
Comparing \eqref{d beta} with Proposition~\ref{rescue}, we see that if $\alpha\not\in\APTre{c}$, then~$C$ is contained in $\set{x\in V^*:\br{x,\alpha}\le0}$, so we conclude that $\alpha\in\Gamma$.
But then since $y\in{\sett{x\in\d_\infty^+:\br{x,\phi}<0\,\forall\phi\in\Gamma}}$, we see that $y\not\in\alpha^\perp$.
We conclude by this contradiction that $\alpha\in\APTre{c}$.
However, by induction, $\d_\alpha$ does not intersect ${\sett{x\in\d_\infty^+:\br{x,\phi}<0\,\forall\phi\in\Gamma}}$, so we have also reached a contradiction in this case.
We conclude that $\d_\beta$ does not intersect~${\sett{x\in\d_\infty^+:\br{x,\phi}<0\,\forall\phi\in\Gamma}}$.
\end{proof}

\subsection{Generalized cluster algebras}\label{gca back sec}
We now review the definition of a generalized cluster algebra given in \cite[Section~2.1]{ChekhovShapiro}, but modifying the definition to add a normalization condition (in the sense of normalized cluster algebras \cite[Definition~5.3]{ca1}), to work in the ``tall extended exchange matrix'' convention that matches our conventions in this paper, and to allow more flexibility in the coefficient ring.
For convenience and for the sake of easier comparison with \cite{ChekhovShapiro} and~\cite{ca4}, we forget temporarily that the symbols $d_i$ were used earlier as skew-symmetrizing constants.

Fix positive integers $d_1,\ldots,d_n$, fix a semifield $\PP$ with addition written as $\oplus$, and fix a field $\FF$ isomorphic to the field of rational functions in $n$ variables with coefficients in $\integers\PP$.
By definition, the set $\PP$ is an abelian group under multiplication, and an easy standard argument shows that this group is torsion free.
(If $p^e=1$ for $e>0$, then $p\cdot\bigoplus_{i=0}^{e-1}p^i=\bigoplus_{i=1}^ep^i=\bigoplus_{i=0}^{e-1}p^i$, so $p=1$.)
A \newword{normalized generalized seed} is a triple $(\x,\pp,B)$, where
\begin{itemize}
\item $\x$ is an $n$-tuple of algebraically independent elements of $\FF$, 
\item $\pp$ is an $n$-tuple $(p_1,\ldots,p_n)$ such that each $p_i$ is a $(d_i+1)$-tuple $(p_{i;0},\ldots,p_{i;d_i})$ of elements of $\PP$ satisfying the normalizing condition $p_{i;0} \oplus p_{i;d_i}=1$.
  This normalization requirement, which is absent from the definition in~\cite{ChekhovShapiro} and differs from the condition in~\cite{Tomoki}, means that all of the seeds are determined uniquely by the initial seed by way of the mutation process described below. 
\item $B$ is a skew-symmetrizable $n\times n$ integer matrix such that the $i\th$ column is divisible by $d_i$ for all $i$.   
\end{itemize}
\newword{Generalized seed mutation} in direction $k$ is the operation that produces a new seed $(\x',\pp',B')$ from $(\x,\pp,B)$ as follows.
\begin{itemize}
\item
$\x'=(x'_1,\ldots,x'_n)$ with 
\begin{equation}\label{x mut eq}
x'_j=
\begin{cases}
x_j&\text{if }j\neq k,\text{ or} \\
\displaystyle \frac1{x_k}\sum_{\ell=0}^{d_k}p_{k;\ell}\prod_{i=1}^nx_i^{[b_{ik}]_+-\ell\frac{b_{ik}}{d_k}}&\text{if }j=k.
\end{cases}\end{equation}
\item 
$\pp'=(p'_1,\ldots,p'_n)$ with 
\begin{equation}\label{p mut}
p'_{j;\ell}=
\begin{cases}
  p_{k;d_k-\ell} &\text{if }j=k, or \\
\displaystyle\frac{p_{j;\ell}\,p_{k;0}^{\frac{d_j-\ell}{d_j}[b_{kj}]_+}\,p_{k;d_k}^{\frac{\ell}{d_j}[-b_{kj}]_+}}{p_{j;0}\,p_{k;0}^{[b_{kj}]_+}\oplus p_{j;d_j}\,p_{k;d_k}^{[-b_{kj}]_+}} &\text{if }j\neq k.
\end{cases}\end{equation}
\item $B'$ is given by the usual matrix mutation $\mu_k(B)$.
\end{itemize}
In \cite{ChekhovShapiro}, the generalized cluster algebra associated to a generalized seed $(\x,\pp,B)$ is the $\integers\PP$-subalgebra of $\FF$ generated by all cluster variables in all seeds obtained from $(\x,\pp,B)$ by arbitrary sequences of mutations.  
However, we take a definition that is less restrictive, in the same way that \cite[Definition~2.3]{ca1} is less restrictive than \cite[Definition~1.2]{ca2}.
Our definition of the \newword{normalized generalized cluster algebra} replaces~$\integers\PP$ with any subring of $\integers\PP$ containing all of the elements $p_{i;j}$ that occur in all seeds.
The generalized cluster algebra is of \newword{finite type} if the set of seeds obtained from $(\x,\pp,B)$ by mutation is finite.
Generalized cluster algebras of finite type are classified by their exchange matrices \cite[Theorem~2.7]{ChekhovShapiro} in exactly the same way as cluster algebras~\cite{ca2}, and thus each generalized cluster algebra of finite type has a Cartan-Killing type.
(The transpose between the definition here and the definition in~\cite{ChekhovShapiro} affects the naming of types.)

In this paper, we only work directly with normalized generalized cluster algebras, but we will obtain some non-normalized generalized cluster algebras as specializations, so we mention the definition here.
Observe that \eqref{p mut} implies that for $j\neq k$,
\begin{equation}\label{nonnorm p mut}
\frac{p'_{j;\ell}}{p'_{j;0}}=\frac{p_{j;\ell}\,p_{k;d_k}^{\frac\ell{d_k}[-b_{kj}]_+}}{p_{j;0}\,p_{k;0}^{\frac\ell{d_k}[b_{kj}]_+}}\,,
\end{equation}
which agrees, up to the changes in conventions mentioned above, with \cite[(2.5)]{ChekhovShapiro}.
But \eqref{nonnorm p mut} does not determine $\pp'$ uniquely from $\pp$.
Nevertheless, taking $\PP$ to be any torsion-free abelian group (not necessarily specifying a semifield structure), written multiplicatively, one can start from an initial seed and do all possible sequences of mutations, making choices of the $\pp$ at each seed subject to \eqref{nonnorm p mut}.
Choosing the $\pp$ in this looser way but otherwise following the definition of a normalized generalized cluster algebra above, we obtain~a \newword{(not necessarily normalized) generalized cluster algebra}.
Given a homomorphism $\sigma$ from $\PP$ to another torsion-free abelian group $\PP'$, we can replace each coefficient $p_{j;\ell}$ by $\sigma(p_{j;\ell})$ to obtain a new pattern of seeds satisfying \eqref{nonnorm p mut}.
Thus we have the following proposition about generalized cluster algebras that are not necessarily normalized.

\begin{proposition}\label{still gca}
Suppose $\A$ is a generalized cluster algebra defined in terms of a torsion-free abelian group~$\PP$ and suppose $\sigma$ is a homomorphism from $\PP$ to another torsion-free abelian group.
The ring $\A'$ obtained from $\A$ by replacing each element $p\in\PP$ by $\sigma(p)$ in all expressions is a generalized cluster algebra.
\end{proposition}

We will use the proposition in the case where $\A$ is a \emph{normalized} generalized cluster algebra.
The generalized cluster algebra $\A'$ may fail to be normalized, because there may be no way to give $\PP'$ a semifield structure such that each $p_i$ in each seed is normalized and \eqref{p mut} holds for all mutations.

For the rest of the paper, the term ``generalized cluster algebra'' refers to a \emph{normalized} generalized cluster algebra.

\section{Main Results}\label{res sec}
When $B$ is acyclic and of affine type, Theorem~\ref{clus mon thm} allows us to construct~$\thet_\lambda$ from a cluster monomial for ${\lambda\in P\cap|\gFan(B)|}$.
In this sense, we consider $\thet_\lambda$ to be ``known'' when ${\lambda\in P\cap|\gFan(B)|}$.
The first main results of this paper determine the theta functions $\thet_\lambda$ for all remaining $\lambda\in P$ (those in the relative interior of the imaginary wall $\d_\infty$) in terms of cluster monomials.
The next main results determine new ``imaginary'' exchange relations among certain theta functions in the imaginary wall, show that the theta functions in the imaginary wall generate a subalgebra of the cluster algebra, and relate this imaginary subalgebra to a certain generalized cluster algebra.
We now describe these results in detail.

\subsection{Theta functions in the imaginary wall}\label{thet idents}
We will establish four identities.

The first of the four identities lets us write $\thet_{\nu_c(\delta)}$ in terms of cluster monomials.  
The vector $\nu_c(\delta)$ spans the imaginary ray, which is the only ray of $\F_{B^T}$ that is not in the $\g$-vector fan $\gFan(B)$.
Also, $\nu_c(\delta)$ is primitive in $P$ because~$\delta$ is primitive in $Q$ and because $\nu_c$ is a bijection from~$Q$ to~$P$.

\begin{theorem}\label{thet xi}
Suppose $\beta\in\SimplesT{c}$.
Then 
\[\thet_{\nu_c(\delta)}=\thet_{\nu_c(\beta)}\cdot\thet_{\nu_c(\delta-\beta)}-y^\beta\thet_{\nu_c(\delta-\beta-c^{-1}\beta)}-y^{c\beta}\thet_{\nu_c(\delta-\beta-c\beta)}.\]
\end{theorem}

\begin{remark}\label{2 by 2}
When $B$ is a $2\times2$ exchange matrix, $\SimplesT{c}$ is empty, so Theorem~\ref{thet xi} does not apply.
Simple computations (for example \cite[Propositions~3.16 and 3.18]{scatcomb}) establish the following values of $\thet_{\nu_c(\delta)}$ for half of the $2\times2$ cases.
The other cases can be obtained by transposing $B$ and swapping indices $1$ and $2$.
\begin{align*}
\begin{bsmallmatrix*}[r]0&2\\-2&0\end{bsmallmatrix*}
\qquad&\frac{x_2^2+y_1+y_1y_2x_1^2}{x_1x_2}  
=\frac{x_2}{x_1}(1+\hy_1+\hy_1\hy_2)
\\
\begin{bsmallmatrix*}[r]0&4\\-1&0\end{bsmallmatrix*}
\qquad&\frac{x_2^2+2x_2y_1+y_1^2+x_1^4y_1^2y_2}{x_1^2x_2} 
=\frac{x_2}{x_1^2}(1+2\hy_1+\hy_1^2+\hy_1^2\hy_2)
\\
\begin{bsmallmatrix*}[r]0&1\\-4&0\end{bsmallmatrix*}
\qquad&\frac{x_2^4+y_1+2y_1y_2x_1+y_1y_2^2x_1^2}{x_1x_2^2} 
=\frac{x_2^2}{x_1}(1+\hy_1+2\hy_1\hy_2+\hy_1\hy_2^2)
\end{align*}
\end{remark}

The second theta-function identity of the four lets us compute $\thet_{k\nu_c(\delta)}$ recursively for any integer $k\ge1$ as something like a Chebyshev polynomial in $\thet_{\nu_c(\delta)}$.

\begin{theorem}\label{delta cheby}
The theta functions of multiples of $\nu_c(\delta)$ satisfy 
\[\thet_{2\nu_c(\delta)}=(\thet_{\nu_c(\delta)})^2-2y^\delta\]
and, for $k\ge3$,  
\[\thet_{k\nu_c(\delta)}=\thet_{(k-1)\nu_c(\delta)}\cdot\thet_{\nu_c(\delta)}-y^\delta\thet_{(k-2)\nu_c(\delta)}.\]
\end{theorem}

In the case where $B$ is of affine type $\tilde A$ or $\tilde D$, Theorem~\ref{delta cheby} also follows from \cite[Proposition~4.2]{MSW2}, in light of \cite{MandelQin}.
Recall that the Chebyshev polynomials (of the second kind) are $T_k(x)$ for $k\in\set{0,1,\ldots}$, given by $T_0(x)=2$, $T_1(x)=x$, and $T_k(x)=xT_{k-1}(x)-T_{k-2}(x)$ for $k\ge2$.
Thus, Theorem~\ref{delta cheby} implies that, upon specializing all of the $y_i$ to $1$, each $\thet_{k\nu_c(\delta)}$ becomes a Chebyshev polynomial in the specialization of $\thet_{\nu_c(\delta)}$.

Theorem~\ref{delta cheby} amounts to a formula for expanding the product $\thet_{\nu_c(\delta)}\cdot\thet_{(k-1)\nu_c(\delta)}$ for any $k\ge2$.
More generally, the third theta-function identity of the four is the following formula for expanding the product $\thet_{k\nu_c(\delta)}\cdot\thet_{\ell\nu_c(\delta)}$ for any $k\ge1$ and $\ell\ge1$.

\begin{theorem}\label{imag general}
For any $k\ge 1$,
\[(\thet_{k\nu_c(\delta)})^2=\thet_{2k\nu_c(\delta)}+2y^{k\delta},\]
and for $k>\ell\ge1$,
\[\thet_{k\nu_c(\delta)}\cdot\thet_{\ell\nu_c(\delta)}=\thet_{(k+\ell)\nu_c(\delta)}+y^{\ell\delta}\thet_{(k-\ell)\nu_c(\delta)}.\]
\end{theorem}

One can obtain Theorem~\ref{imag general} from Theorem~\ref{delta cheby} by a routine induction, but instead, we will prove Theorem~\ref{imag general} directly using broken lines, and Theorem~\ref{delta cheby} follows as a special case.

The fourth theta-function identity lets us write the theta function $\thet_\lambda$ for any~$\lambda$ in the interior of $\d_\infty$ as the product of the theta function for a vector in the imaginary ray times a cluster monomial whose $\g$-vector is in the relative boundary of $\d_\infty$.

The \newword{$c$-cluster expansion} of $\phi\in V$ is the unique expression ${\phi=\sum_{\alpha \in \AP{c}}m_\alpha \alpha}$ such that $m_\alpha\ge0$ and such that $m_\alpha m_\beta=0$ whenever $\alpha$ and $\beta$ are distinct and not $c$-compatible \cite[Theorem~6.2]{affdenom}.
Thus for any $\phi\in V$ the set $\set{\alpha\in\AP{c}:m_\alpha\neq0}$ for its $c$-cluster expansion is contained in some $c$-cluster (possibly in the intersection of several $c$-clusters).

\begin{theorem}\label{expansion prod}
Suppose $\phi\in Q$ has $c$-cluster expansion $\phi=\sum_{\alpha \in \AP{c}}m_\alpha \alpha$.
Then 
\[\thet_{\nu_c(\phi)}=\thet_{m_\delta\nu_c(\delta)}\cdot\prod_{\alpha \in \APre{c}}(\thet_{\nu_c(\alpha)})^{m_\alpha }.\]
\end{theorem}

Here and elsewhere, $\thet_0$ is interpreted to be $1$.
When $\phi$ is such that ${m_\delta=0}$, so that $\phi$ is in the support of $\Fan_c^\re(\RS)$ and $\nu_c(\phi)$ is in the support of the $\g$-vector fan, Theorem~\ref{expansion prod} is an immediate consequence of the results of \cite{GHKK} quoted here as Theorem~\ref{clus mon thm} and the fact that $\nu_c$ is a bijection from $Q$ to $P$.
Thus in particular, for arbitrary $\phi$, we have $\thet_{\nu_c(\phi-m_\delta\delta)}=\prod_{\alpha \in \APre{c}}(\thet_{\nu_c(\alpha)})^{m_\alpha }$.
The part of Theorem~\ref{expansion prod} that is new is that $\thet_{\nu_c(\phi)}=\thet_{m_\delta\nu_c(\delta)}\cdot\thet_{\nu_c(\phi-m_\delta\delta)}$. 

Together, Theorems~\ref{thet xi}, \ref{delta cheby}, and~\ref{expansion prod} determine all of the theta functions that are \emph{not} known (i.e. are not associated to cluster monomials) in terms of the theta functions that \emph{are} known.
Specifically, Theorem~\ref{thet xi} lets us write $\thet_{\nu_c(\delta)}$ in terms of cluster variables for roots in $\SimplesT{c}_\fin$.
Then Theorem~\ref{delta cheby} lets us write $\thet_{k\nu_c(\delta)}$ for $k>1$ in terms of $\thet_{\nu_c(\delta)}$.
Finally, Theorem~\ref{expansion prod} lets us write $\thet_\lambda$ as some $\thet_{k\nu_c(\delta)}$ times a known theta function whenever~$\lambda$ is in $\d_\infty$.

Theorem~\ref{expansion prod} also combines with previously known results to complete the description of denominator vectors of theta functions in acyclic affine type.

\begin{corollary}\label{nu denom}
For $\phi\in Q$, the denominator vector of $\thet_{\nu_c(\phi)}$ is $\phi$.
\end{corollary}

For background on denominator vectors, see \cite[Section~7]{ca4}.
The case of Corollary~\ref{nu denom} where $\phi$ is in the support of $\Fan_c^\re(\RS)$ was \cite[Conjecture~3.21]{afframe}, proved as \cite[Proposition~9]{Rupel}.
(See \cite[Section~9.2]{affdenom}.)
The result extends to all vectors in $Q$ by Theorems~\ref{thet xi}, \ref{delta cheby}, and~\ref{expansion prod} (and the formulas in Remark~\ref{2 by 2}).

\subsection{Imaginary exchange relations}\label{imag ex sec}  
Two distinct roots $\gamma$ and $\gamma'$ in $\AP{c}$ are \mbox{\newword{$c$-exchangeable}} if there exist $c$-clusters $C$ and $C'$ with $\gamma\in C$ and $\gamma'\in C'$ such that $C\setminus\{\gamma\}=C'\setminus\{\gamma'\}$.
The two roots are \newword{$c$-real-exchangeable} if $C$ and $C'$ can be taken to be \emph{real} $c$-clusters.
If $\gamma$ and $\gamma'$ are $c$-real-exchangeable, then the corresponding cluster variables are related by an exchange relation in the usual sense.
(See \cite[Definition~2.4]{ca4}.)
The exchange relation essentially writes $\thet_{\nu_c(\gamma)}\cdot\thet_{\nu_c(\gamma')}$ as a sum of two cluster monomials, with coefficients.  

Here, we prove \newword{imaginary exchange relations} for roots $\gamma$ and $\gamma'$ in $\AP{c}$ that are $c$-exchangeable but not $c$-real-exchangeable.
Such a relation writes ${\thet_{\nu_c(\gamma)}\cdot\thet_{\nu_c(\gamma')}}$ as a linear combination of \emph{three} theta functions.  
The following precise characterization of pairs that are $c$-exchangeable but not $c$-real-exchangeable is part of \mbox{\cite[Theorem~7.2]{affdenom}}, interpreted in light of \cite[Theorem~4.6]{affdenom}, particularly \cite[(4.6)]{affdenom}.

\begin{theorem}\label{exchangeable}   
For $\gamma,\gamma'\in\APre{c}$, the following are equivalent:
\begin{enumerate}[label=\rm(\roman*), ref=(\roman*)]
\item
$\gamma$ and $\gamma'$ are $c$-exchangeable but not $c$-real-exchangeable.
\item There exist $\beta,\beta'\in\SimplesT{c}$, distinct but in the same $c$-orbit, such that $\gamma=\delta-\beta$ and $\gamma'=\delta-\beta'$.
\end{enumerate}
\end{theorem}

We will prove the following theorem on imaginary exchange relations.
\begin{theorem}\label{im exch}
Suppose $\beta,\beta'\in\SimplesT{c}$ are distinct but in the same $c$-orbit.
Let $\ell$ and $m$ be the smallest positive integers such that $\beta'=c^\ell\beta$ and $\beta=c^m\beta'$.
Write $\phi=\sum_{i=1}^{\ell-1}c^i\beta$ and $\phi'=\sum_{i=1}^{m-1}c^i\beta'$.
Then
\begin{align*}
\thet_{\nu_c(\delta-\beta)}\cdot\thet_{\nu_c(\delta-\beta')}
&=\thet_{\nu_c(\delta+\phi+\phi')}+y^{\phi'+\beta}\thet_{2\nu_c(\phi)}+y^{\phi+\beta'}\thet_{2\nu_c(\phi')}\\
&=\thet_{\nu_c(\delta)}\thet_{\nu_c(\phi)}\thet_{\nu_c(\phi')}+y^{\phi'+\beta}(\thet_{\nu_c(\phi)})^2+y^{\phi+\beta'}(\thet_{\nu_c(\phi')})^2.
\end{align*}
\end{theorem}
The equality of the two lines on the right side is an easy application of Theorem~\ref{expansion prod}.
As usual, a sum $\sum_{i=1}^0$ is interpreted to mean zero.

\subsection{The imaginary subalgebra and tube subalgebras}\label{imag sec}
As in Section~\ref{clus sec}, let $\k[u^{\pm1}]$ be the ring of Laurent polynomials in the tropical variables $u_1,\ldots,u_m$.
We define the \newword{cluster algebra} $\A(\tB)$, consistent with \cite[Definitions~2.11--2.12]{ca4}, to be the $\k[u^{\pm1}]$-subalgebra, generated by all cluster variables, of the field of rational functions in $x_1,\ldots,x_n$ with coefficients in $\k[u^{\pm1}]$.
In light of Theorem~\ref{clus mon thm}, $\A(\tB)$ contains the theta functions associated to vectors in $P\cap|\gFan(B)|$,  
and thus by Theorems~\ref{thet xi}, \ref{delta cheby}, and~\ref{expansion prod}, $\A(\tB)$ contains all theta functions.
Since $B$ is acyclic, when $\tB$ has linearly independent columns, $\A(\tB)$ coincides with the GHKK canonical algebra.  
(See \cite[Proposition~1.8]{ca3}, \cite[Corollary~1.19]{ca3}, and \cite[Proposition~0.7]{GHKK}.)

Let $\k[y^{\pm\beta}]_{\beta\in\SimplesT{c}}$ be the subring of $\k[u^{\pm1}]$ generated by $\sett{y^{\pm\beta}:\beta\in\SimplesT{c}}$.
Define the \newword{imaginary subalgebra} of $\A(\tB)$ to be the set $\I(\tB)$ of finite linear combinations of the theta functions $\set{\thet_\lambda:\lambda\in\d_\infty}$ with coefficients in $\k[y^{\pm\beta}]_{\beta\in\SimplesT{c}}$.
The term ``imaginary subalgebra'' is justified by the following theorem.

\switchmargin

\begin{theorem}\label{subalgebra} 
If $\tB$ is acyclic of affine type, then $\I(\tB)$ is a $(\k[y^{\pm\beta}]_{\beta\in\SimplesT{c}})$-subalgebra of $\A(\tB)$.
It is generated, as an algebra over $\k[y^{\pm\beta}]_{\beta\in\SimplesT{c}}$, by $\set{\thet_{\nu_c(\gamma)}:\gamma\in\APTre{c}}$.
\end{theorem}

For each $c$-orbit $\SimplesT{c}_o$ in $\SimplesT{c}$, let $\k[y^{\pm\beta}]_{\beta\in\SimplesT{c}_o}$ be the subring of $\k[u^{\pm1}]$ generated by the set $\sett{(y^\beta)^{\pm1}:\beta\in\SimplesT{c}_o}$.
Define the \newword{tube subalgebra} $\T_o(\tB)$ to be the set of finite $(\k[y^{\pm\beta}]_{\beta\in\SimplesT{c}_o})$-linear combinations of the theta functions~$\thet_\lambda$ such that $\lambda$ is in the nonnegative integer span of $\sett{\nu_c(\beta):\beta\in\SimplesT{c}_o}$.
The term ``tube subalgebra'' refers to the following theorem and to the known correspondence between $c$-orbits in $\SimplesT{c}$ and the representation-theoretic ``tubes'' associated to the data $B$.

\begin{theorem}\label{tube subalgebra}
If $\tB$ is acyclic of affine type, then for any $c$-orbit $\SimplesT{c}_o$ in $\SimplesT{c}$, $\T_o(\tB)$ is a $(\k[y^{\pm\beta}]_{\beta\in\SimplesT{c}_o})$-subalgebra of $\A(\tB)$.
It is generated, as an algebra over $\k[y^{\pm\beta}]_{\beta\in\SimplesT{c}_o}$, by $\set{\thet_{\nu_c(\gamma)}:\gamma\in\APTre{c;o}}$.
\end{theorem}

\begin{remark}\label{which algebra?}
Theorems~\ref{subalgebra} and~\ref{tube subalgebra} remain true when $\A(\tB)$ is replaced with any smaller algebra that contains the theta functions and the ring $\k[y^{\pm1}]$ of Laurent polynomials in $y_1,\ldots,y_n$.
The smallest choice is the small canonical algebra of \mbox{\cite[Section~3.1]{canonical}}, which is the $\k[y^{\pm1}]$-algebra generated by all theta functions.
\end{remark}

\subsection{Imaginary/tube subalgebras as generalized cluster algebras}\label{gca sec}
We now realize tube subalgebras and imaginary subalgebras as (normalized) generalized cluster algebras (in the sense of Section~\ref{gca back sec}), when $\tB$ has nondegenerate coefficients.
For arbitrary coefficients, we will show that tube subalgebras and imaginary subalgebras are specializations of (normalized) generalized cluster algebras.
Thus by Proposition~\ref{still gca}, for arbitrary coefficients, tube subalgebras and imaginary subalgebras are not-necessarily-normalized generalized cluster algebras (again in the sense of Section~\ref{gca back sec}).
For the rest of the section, we use the phrase ``generalized cluster algebra'' to mean a \emph{normalized} generalized cluster algebra.

We begin by defining a generalized cluster algebra related to each $\T_o(\tB)$.
Let~$\PP_o$ be the topical semifield with tropical variables $\set{z_\beta:\beta\in\SimplesT{c}_o}\cup\set{z_*}$, so that the elements of $\PP_o$ are Laurent monomials in these tropical variables, the multiplication is the usual multiplication of Laurent monomials, and the addition~$\oplus$ is componentwise minimum of exponent vectors.  
Given a vector $\phi={\beta+c\beta+\cdots+c^j\beta}\in\APTre{c;o}$, we write $z^\phi$ to mean $z_\beta\cdot z_{c\beta}\cdots z_{c^j\beta}$.

\switchmargin

Throughout the section, we will use the combinatorics of compatible real roots in $\APTre{c}$, as reviewed in Section~\ref{agaf sec}.  
Let $J_o$ be a maximal set of pairwise compatible real roots in $\APTre{c;o}$.
Then $|J_o|=|\SimplesT{c}_o|-1$.
We define a generalized seed $(\x_{J_o},\pp_{J_o},B_{J_o})$ of rank~$|J_o|$ using $J_o$ as indexing set.
First, suppose $\gamma$ is the maximal root in~$J_o$, pictured using a solid outline in the top row of Figure~\ref{def fig}.
Let $\beta$ be the unique element of~$\SimplesT{c}_o\setminus\SuppT(\gamma)$, so that $\gamma=\delta-\beta$.
Let $\beta'$ be the unique element $\SuppT(\gamma)$ that is not in the support of any other root in $J_o$.
Define~$\phi$ and~$\phi'$ in terms of $\beta$ and~$\beta'$ as in Theorem~\ref{im exch}, pictured by dashed outlines in the top row of Figure~\ref{def fig}.
Then $\phi$ and $\phi'$ are (if they are not zero) next smaller roots from $\gamma$ in $J_o$.
The column of~$B$ indexed by $\gamma$ and the coefficients indexed by~$\gamma$ are also shown in the top row of the figure.
If $\phi$ and/or $\phi'$ is zero, then $B$ has no rows indexed by $\phi$ and/or $\phi'$.

\definecolor{darkgreen}{rgb}{0,0.4,0}
\definecolor{darkblue}{rgb}{0,0,0.6}
\definecolor{lightblue}{rgb}{0,0.7,1}

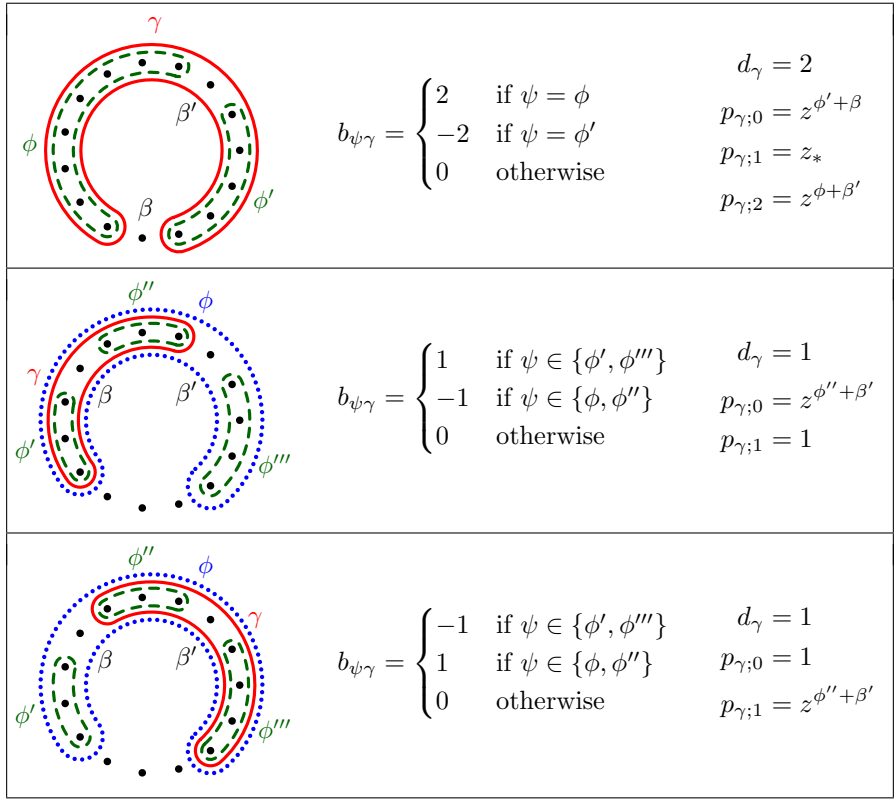
\begin{figure}
\begin{tabular}{|ccc|}\hline&&\\[-8pt]
\begin{minipage}[m]{110pt}
\scalebox{0.95}{\includegraphics{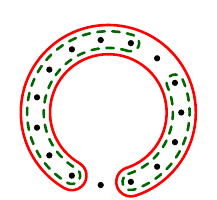}
\begin{picture}(0,0)(52,-42)
\put(-9,-26){$\beta$}
\put(6,12){$\beta'$}
\put(-5,49){\textcolor{red}{$\gamma$}}
\put(-55,0){\textcolor{darkgreen}{$\phi$}}
\put(37,-23){\textcolor{darkgreen}{$\phi'$}}
\end{picture}}
\end{minipage}
&
\begin{minipage}[m]{135pt}
$b_{\psi\gamma}=\begin{cases}
2&\text{if }\psi=\phi\\
-2&\text{if }\psi=\phi'\\
0&\text{otherwise}
\end{cases}$\\
\end{minipage}
& 
\begin{minipage}[m]{60pt}
$\begin{aligned}
d_{\gamma}&=2\\
p_{\gamma;0}&=z^{\phi'+\beta}\\
p_{\gamma;1}&=z_*\\
p_{\gamma;2}&=z^{\phi+\beta'}
\end{aligned}$
\end{minipage}\\[44pt]\hline&&\\[-8pt]
\begin{minipage}[m]{110pt}
\scalebox{0.95}{\includegraphics{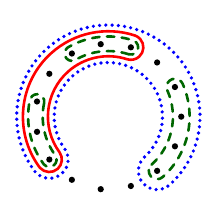}
\begin{picture}(0,0)(52,-40)
\put(-25,7){$\beta$}
\put(6,8){$\beta'$}
\put(-53,16){\textcolor{red}{$\gamma$}}
\put(15,45){\textcolor{blue}{$\phi$}}
\put(-58,-15){\textcolor{darkgreen}{$\phi'$}}
\put(-13,48){\textcolor{darkgreen}{$\phi''$}}
\put(39,-21){\textcolor{darkgreen}{$\phi'''$}}
\end{picture}}
\end{minipage}
&
\begin{minipage}[m]{135pt}
$b_{\psi\gamma}=\begin{cases}
1&\text{if }\psi\in\set{\phi',\phi'''}\\
-1&\text{if }\psi\in\set{\phi,\phi''}\\
0&\text{otherwise}
\end{cases}$\\
\end{minipage}
& 
\begin{minipage}[m]{60pt}
$\begin{aligned}
d_{\gamma}&=1\\
p_{\gamma;0}&=z^{\phi''+\beta'}\\
p_{\gamma;1}&=1
\end{aligned}$
\end{minipage}\\[44pt]\hline&&\\[-8pt]
\begin{minipage}[m]{110pt}
\scalebox{0.95}{\includegraphics{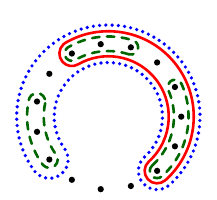}
\begin{picture}(0,0)(52,-40)
\put(-25,7){$\beta$}
\put(6,8){$\beta'$}
\put(35,25){\textcolor{red}{$\gamma$}}
\put(15,45){\textcolor{blue}{$\phi$}}
\put(-58,-15){\textcolor{darkgreen}{$\phi'$}}
\put(-13,48){\textcolor{darkgreen}{$\phi''$}}
\put(39,-21){\textcolor{darkgreen}{$\phi'''$}}
\end{picture}}
\end{minipage}
&
\begin{minipage}[m]{135pt}
$b_{\psi\gamma}=\begin{cases}
-1&\text{if }\psi\in\set{\phi',\phi'''}\\
1&\text{if }\psi\in\set{\phi,\phi''}\\
0&\text{otherwise}
\end{cases}$\\
\end{minipage}
& 
\begin{minipage}[m]{60pt}
$\begin{aligned}
d_{\gamma}&=1\\
p_{\gamma;0}&=1\\
p_{\gamma;1}&=z^{\phi''+\beta'}
\end{aligned}$
\end{minipage}\\[44pt]\hline
\end{tabular}
\caption{Defining the generalized seed}
\label{def fig}
\end{figure}

Now, suppose $|\SuppT(\gamma)|<|J_o|$ and let $\phi$ be the next larger root from $\gamma$ in $J_o$.
Let $\set{\beta,\beta'}$ be the set containing the unique element of $\SuppT(\phi)$ that is not in the support of any other root in~$J_o$ not containing $\SuppT(\phi)$ and the unique element of $\SuppT(\gamma)$ that is not in the support of any other root in~$J_o$ not containing $\SuppT(\gamma)$, and name them so that $\beta'$ is clockwise of $\beta$ in $\SuppT(\phi)$.
Deleting $\beta$ and $\beta'$ from $\SuppT(\phi)$ cuts $\SuppT(\phi)$ into three pieces that we name $\phi'$, $\phi''$, and $\phi'''$ in clockwise order.
Possibly some or all of $\phi'$, $\phi''$, and~$\phi'''$ are zero, but each one that is not zero is in~$J_o$.
There are two cases, illustrated in the second and third rows of Figure~\ref{def fig}.
In the pictures, $\gamma$ is shown by a solid outline, $\phi$ is shown by a dotted outline, and $\phi'$, $\phi''$, and $\phi'''$ are shown with dashed outlines.

It is apparent from the definition that $\pp_{J_o}$ satisfies the normalizing condition $p_{\gamma;0}\oplus p_{\gamma;d_\gamma}$ for all $\gamma\in J_o$.   
One can also easily check that if one scales the column of~$B_{J_o}$ indexed by $\gamma$ by a factor of $\frac12$, the resulting matrix is skew-symmetric.
Thus~$B_{J_o}$ is skew-symmetrizable, so $(\x_{J_o},\pp_{J_o},B_{J_o})$ is a generalized seed.

We define $\A(\x_{J_o},\pp_{J_o},B_{J_o})$ to be the generalized cluster algebra determined by $(\x_{J_o},\pp_{J_o},B_{J_o})$, with coefficient ring $\k[z_*,\,z_\beta^{\pm1}]_{\beta\in\SimplesT{c}_o}$ (polynomials that are Laurent in $\set{z_\beta:\beta\in\SimplesT{c}_o}$ and ordinary in $z_*$).
We will prove the following theorem. 

\begin{theorem}\label{gen clus alg prop}
Let $J_o$ be a maximal set of pairwise compatible real roots in $\APTre{c;o}$.
\begin{enumerate}[label=\bf\arabic*., ref=\arabic*] 
\item \label{gca prop bij}
The cluster variables of $\A(\x_{J_o},\pp_{J_o},B_{J_o})$ can be indexed as $x_\gamma$ for $\gamma\in\APTre{c;o}$ such that the map $J_o'\mapsto(\x_{J_o'},\pp_{J_o'},B_{J_o'})$ is bijection from the maximal sets~$J_o'$ of pairwise compatible roots in~$\APTre{c;o}$ to the seeds of $\A(\x_{J_o},\pp_{J_o},B_{J_o})$.
\item \label{gca prop cyclo}
The cluster complex of $\A(\x_{J_o},\pp_{J_o},B_{J_o})$ is isomorphic to the complex of maximal sets of pairwise compatible roots in $\APTre{c;o}$, which is isomorphic to the boundary complex of the simplicial cyclohedron for a $(|J_o|+1)$-cycle.
\item \label{gca prop C}
$\A(\x_{J_o},\pp_{J_o},B_{J_o})$ is of finite type $C_{|J_o|}$.
\end{enumerate}
\end{theorem}

Let $\mapchar_o$ be the set map from $\set{z_\beta:\beta\in\SimplesT{c}_o}\cup\set{z_*}\cup\set{x_\gamma:\gamma\in J_o}$ to the tube subalgebra $\T_o(\tB)$ that sends $z_\beta$ to $y^\beta$ for each $\beta\in\SimplesT{c}_o$, sends $z_*$ to $\thet_{\nu_c(\delta)}$, and sends the cluster variable $x_\gamma$ to the theta function $\thet_{\nu_c(\gamma)}$ for each $\gamma\in J_o$.

\begin{theorem}\label{gen clus alg}  
Let $J_o$ be a maximal set of pairwise compatible real roots in $\APTre{c;o}$.
\begin{enumerate}[label=\bf\arabic*., ref=\arabic*] 
\item \label{gca hom}
There is a unique extension of the set map $\mapchar_o$ to a ring homomorphism ${\mapchar_o:\A(\x_{J_o},\pp_{J_o},B_{J_o})\to\T_o(\tB)}$, and this homomorphism is surjective.
\item \label{gca no J}
Indexing the cluster variables as $x_\gamma$ for $\gamma\in\APTre{c;o}$, the homomorphism $\mapchar_o$ is independent of the choice of $J_o$ used to define it.
\item \label{gca bij}
The homomorphism $\mapchar_o$ restricts to a bijection $x_\gamma\mapsto\thet_{\nu_c(\gamma)}$ from cluster variables in $\A(\x_{J_o},\pp_{J_o},B_{J_o})$ to theta functions $\thet_{\nu_c(\gamma)}$ such that $\gamma\in\APTre{c;o}$.
\item \label{gca nondegen}
If $\tB$ has nondegenerate coefficients, then $\mapchar_o$ is an isomorphism.
\item \label{gca spec}
The tube subalgebra $\T_o(\tB)$ is isomorphic to the ring that is obtained from $\A(\x_{J_o},\pp_{J_o},B_{J_o})$ by specializing each $z_\beta$ to the element $y^\beta$.
\end{enumerate}
\end{theorem}

Theorem~\ref{gen clus alg} combines with Theorem~\ref{gen clus alg prop} to say that the tube subalgebra~$\T_o(\tB)$ is a generalized cluster algebra of finite type $C_{|J_o|}$ when $\tB$ has nondegenerate coefficients.
For arbitrary extensions~$\tB$, the tube subalgebra $\T_o(\tB)$ is obtained from a generalized cluster algebra of type $C_{|J_o|}$ by specializing the coefficient ring.

\begin{remark}\label{theta thinks it's principal}
The generalized cluster algebra $\A(\x_{J_o},\pp_{J_o},B_{J_o})$ can be considered to have principal coefficients, in some sense that we will not try to make precise.
However, $\A(\tB)$ need not have principal coefficients, and therefore theta functions and cluster variables differ as explained in Theorem~\ref{clus mon thm}.
Theorem~\ref{gen clus alg} still works because, in any case, the theta functions exchange as if they are principal coefficients cluster variables.
(See Theorem~\ref{im exch} and Proposition~\ref{real exch}.)
\end{remark}

\begin{remark}\label{not tomoki}
Although Theorem~\ref{gen clus alg}.\ref{gca nondegen} realizes a generalized cluster algebra as a subring of a cluster algebra, it is not an instance of the recent constructions described in \cite{AkagiNakanishi,RamosWhiting}.
In particular, their construction would embed the tube subalgebra~$\T_o(\tB)$ into a cluster algebra with clusters of size $|J_o|+1$.
\end{remark}

Now, choose a maximal set $J_o$ of pairwise compatible real roots in $\APTre{c;o}$ for each $c$-orbit in $\SimplesT{c}$, so that $J=\bigcup_oJ_o$ is a maximal set of pairwise compatible real roots in $\APTre{c}$.
We will define a generalized cluster algebra denoted $\A(\x_J,\pp_J,B_J)$.

The initial cluster is $\x_J=(x_\gamma:\gamma\in J)$.
The coefficient semifield $\PP$ is the tropical semifield with tropical variables $\set{z_\beta:\beta\in\SimplesT{c}}\cup\set{z_*}$ and the initial coefficients $p_\gamma$ are defined exactly as before, viewing each $\gamma\in J$ as an element of some $J_o$.
We emphasize that the tropical variable $z_*$ that appears in some coefficients $p_\gamma$ is the same, independent of which $J_o$ has $\gamma\in J_o$.
The initial exchange matrix is block diagonal and its diagonal blocks are the $B_{J_o}$ as above.
We define $\A(\x_J,\pp_J,B_J)$ to be the generalized cluster algebra determined by $(\x_J,\pp_J,B_J)$ with coefficient ring $\k[z_*,\,z_\beta^{\pm1}]_{\beta\in\SimplesT{c}}$ (polynomials that are Laurent in $\set{z_\beta:\beta\in\SimplesT{c}}$ and ordinary in $z_*$).

By Theorem~\ref{gen clus alg prop} and because $B_J$ is block diagonal, the cluster variables of $\A(\x_J,\pp_J,B_J)$ are indexed as $x_\gamma$ for ${\gamma\in\APTre{c}}$, seeds are indexed by maximal sets of pairwise compatible roots in $\APTre{c}$, the cluster complex is a product of simplicial cyclohedra, and $\A(\x_J,\pp_J,B_J)$ is a generalized cluster algebra of finite type $\prod_oC_{|J_o|}$.
Furthermore, $\A(\x_J,\pp_J,B_J)$ is the tensor product, indexed by the $c$-orbits in $\SimplesT{c}$, 
\[\A(\x_J,\pp_J,B_J)\cong{\bigotimes}_{\k[z_*]}\A(\x_{J_o},\pp_{J_o},B_{J_o}).\]

Let $\mapchar$ be the set map from $\sett{y^\beta:\beta\in\SimplesT{c}}\cup\sett{z_*}\cup\sett{x_\gamma:\gamma\in J}$
to the imaginary subalgebra $\I(\tB)$ that sends $z_\beta$ to $y^\beta$ for each $\beta\in\SimplesT{c}$, sends $z_*$ to $\thet_{\nu_c(\delta)}$, and sends the cluster variable $x_\gamma$ to the theta function $\thet_{\nu_c(\gamma)}$ for each $\gamma\in J$.

\begin{theorem}\label{gen clus alg almost} 
Let $J$ be a maximal set of pairwise compatible real roots in $\APTre{c}$.
\begin{enumerate}[label=\bf\arabic*., ref=\arabic*] 
\item \label{gca almost hom}
There is a unique extension of the set map $\mapchar$ to a ring homomorphism $\mapchar:\A(\x_J,\pp_J,B_J)\to\I(\tB)$, and this homomorphism is surjective.
\item \label{gca almost no J}
Indexing the cluster variables as $x_\gamma$ for $\gamma\in\APTre{c}$, the homomorphism $\mapchar$ is independent of the choice of $J$ used to define it.
\item \label{gca almost bij}
The homomorphism $\mapchar$ restricts to a bijection $x_\gamma\mapsto\thet_{\nu_c(\gamma)}$ from cluster variables in $\A(\x_J,\pp_J,B_J)$ to theta functions $\thet_{\nu_c(\gamma)}$ such that $\gamma\in\APTre{c}$.
\item \label{gca almost ker}
If $\tB$ has nondegenerate coefficients, then the kernel of $\mapchar$ is generated by $\sett{\,\prod_{\beta\in\SimplesT{c}_o}z_\beta-\prod_{\beta\in\SimplesT{c}_{o'}}z_\beta:o\neq o'}$.
\item \label{gca almost spec}
The imaginary subalgebra $\I(\tB)$ is isomorphic to the ring obtained from $\A(\x_J,\pp_J,B_J)$ by specializing each $z_\beta$ to the element $y^\beta$ of the tropical semifield.
\end{enumerate}
\end{theorem}

We conclude this section with some remarks on special cases.  
First, suppose~$\tB$ has nondegenerate coefficients (for example, suppose~$\tB$ has principal coefficients).
Then each tube subalgebra is a generalized cluster algebra of finite type $C$ by Theorem~\ref{gen clus alg}.\ref{gca nondegen}.
When $\SimplesT{c}$ consists of more than one $c$-orbit, the imaginary subalgebra is a quotient of a generalized cluster algebra of type $\prod_oC_{|J_o|}$.

Second, suppose $\tB$ is coefficient-free (meaning that $\tB=B$).
Theorem~\ref{gen clus alg}.\ref{gca spec} says that each tube subalgebra is isomorphic to a generalized cluster algebra with every tropical variable except $z_*$ set to~$1$ and Theorem~\ref{gen clus alg almost}.\ref{gca almost spec} says the same about the imaginary algebra.
But setting tropical variables to $1$ in the tropical semifield commutes with the tropical addition, so the result of this  specialization is a new tropical semifield with only one tropical variable $z_*$, and the specialization of the generalized cluster algebra is a new generalized cluster algebra.
Thus we have the following immediate corollary of Theorem~\ref{gen clus alg}.\ref{gca spec} (and Theorem~\ref{gen clus alg almost}.\ref{gca almost spec}).

\begin{corollary}\label{cf cor}
Suppose $\tB=B$.
\begin{enumerate}[label=\bf\arabic*., ref=\arabic*] 
\item
If $J_o$ is a maximal set of pairwise compatible real roots in $\APTre{c;o}$, then $\T_o(\tB)$ is isomorphic to the generalized cluster algebra obtained from $\A(\x_{J_o},\pp_{J_o},B_{J_o})$ by specializing $z_\beta$ to $1$ for each $\beta\in\SimplesT{c}_o$.
\item
If $J$ is a maximal set of pairwise compatible real roots in $\APTre{c}$, then $\I(\tB)$ is isomorphic to the generalized cluster algebra obtained from $\A(\x_J,\pp_J,B_J)$ by specializing $z_\beta$ to $1$ for each $\beta\in\SimplesT{c}$.
\end{enumerate}
\end{corollary}

\section{Tools}\label{tools sec}
In principle, most of the results described in Section~\ref{real exchangeable} are simple applications of Proposition~\ref{struct} to expand a product of theta functions by finding all pairs of broken lines that contribute to the expansion formula.
In practice, finding the appropriate broken lines is usually not difficult, but ruling out additional pairs is hard.
In this section, we gather and develop some tools to rule out pairs of broken lines.

\subsection{Mutation symmetry and the Coxeter element}\label{sym cox sec}
Our first tool is the following theorem that greatly reduces the possible vectors $\lambda$ such that $a(p_1,p_2,\lambda)$ can be nonzero in the equation \eqref{struct eq} of Proposition~\ref{struct}.
Every main result of this paper is a theta function computation that fits the hypotheses of this theorem.
Except for Theorem~\ref{B cone prod aff}, below, which strengthens this theorem in a special case, all of the further tools we develop need this theorem as a starting point.
In the theorem, $\k[y]$ is the ring of polynomials in $y_1,\ldots,y_n$, as usual.

\begin{theorem}\label{expand d inf}  
Suppose $B$ is acyclic of affine type and let $\tB$ be an extension of~$B$. 
Suppose $\v$ is a monomial in a finite set of theta functions $\thet_\lambda$, each having ${\lambda\in\d_\infty}$.
Then $\v$ is a finite $\k[y]$-linear combination of theta functions $\thet_\kappa$, each having ${\kappa\in\d_\infty}$.
\end{theorem}

We will prove Theorem~\ref{expand d inf} using a general result from \cite{canonical} that needs the hypothesis of signed-nondegenerating coefficients, but we can remove that hypothesis in Theorem~\ref{expand d inf} using Proposition~\ref{struct plus plus}.
Given an exchange matrix $B$, a \newword{mutation symmetry} of~$B$ is a sequence $\kk$ of indices such that $\mu_\kk(B)=B$.
Recall the definition of~$\Theta$ from Section~\ref{scat sec}.
The following is \cite[Theorem~5.1]{canonical}, translated to this transposed setting.
The sum in the theorem is finite because $N\subseteq\Theta$.

\begin{theorem}\label{finite orbit}
Suppose $\tB$ has signed-nondegenerating coefficients and suppose~$\kk$ is a mutation symmetry of~$B$.
Let $\v$ be a monomial in a finite set $\set{\thet_\nu:\nu\in N}$ of theta functions with $N\subset\Theta$, expressed as $\v=\sum_{\lambda\in P}\sum_{\beta\in Q}c_{\lambda,\beta}\,y^\beta\thet_\lambda$ in the theta basis.
If each $\nu\in N$ is in a finite $\eta^{B^T}_\kk$-orbit but $\lambda$ is in an infinite $\eta^{B^T}_\kk$-orbit, then $c_{\lambda,\beta}=0$ for all~$\beta\in Q$.
\end{theorem}

In order to use Theorem~\ref{finite orbit} in the proof of Theorem~\ref{expand d inf}, we find a mutation symmetry $\kk$ of~$B$ with the property that points in $\d_\infty$ are in finite $\eta_\kk^{B^T}$-orbits and points not in $\d_\infty$ are in infinite $\eta_\kk^{B^T}$-orbits.
To that end, we first review the definition of a permutation $\tau_c$ of $\AP{c}$ that is closely related to $c$.
The details of the definition of $\tau_c$ will be less important for the present paper than the facts we quote about the $\tau_c$-orbits in $\AP{c}$.
For each simple reflection~$s$, let $\sigma_s:\AP{c}\to\AP{c}$ be
\[\sigma_s(\alpha)=
\begin{cases}
\alpha & \text{if }\alpha\in-\Simples\setminus\set{-\alpha_s}\\
s(\alpha) & \text{otherwise}.
\end{cases}\]
Here, as usual, $s$ acts on roots by reflection.
Define $\tau_c=\sigma_{s_1}\cdots\sigma_{s_n}$, where, as before, $c=s_1\cdots s_n$ is the Coxeter element determined by $B$.
The following is part of \cite[Proposition~7.33]{affscat}.
(Part of Proposition~\ref{eta nice}.\ref{dinf dom of def} below is not stated in \cite[Proposition~7.33]{affscat}, but is established inside the proof of \cite[Proposition~7.33]{affscat}.)  
The action of $c$ in the proposition is the usual dual action of $s_1\cdots s_n$ on $V^*$.

\begin{proposition}\label{eta nice} 
Suppose $B$ is an acyclic exchange matrix and the Coxeter element associated to $B$ is $c=s_1\cdots s_n$.
Then the piecewise linear map $\eta^{B^T}_{12\cdots n}$
\begin{enumerate}[label=\bf\arabic*., ref=\arabic*] 
\item \label{nu tau}   
has $\eta^{B^T}_{12\cdots n}\circ\nu_c=\nu_c\circ\tau_c$ as maps on $\AP{c}$, 
\item\label{eta aut mut}
is an automorphism of $\F_{B^T}$, and
\item\label{eta is c}
fixes $\d_\infty$ as a set and has finite order on $\d_\infty$.
\item\label{dinf dom of def}
contains all of $\d_\infty$ in one domain of definition, and agrees with $c$ on that domain.
\end{enumerate}
\end{proposition}

The concatenation of \cite[Proposition~3.12(5)]{affdenom} and \cite[Proposition~5.6]{affdenom} says that the $\tau_c$-orbit of root in $\AP{c}$ is finite if and only if the root spans a ray in the star of~$\delta$.
Since $\nu_c$ induces an isomorphism of complete fans from $\Fan_c(\RS)$ to the mutation fan $\F_{B^T}$ by \cite[Theorem~2.9]{affscat}, since the rays of $\Fan_c(\RS)$ are spanned by the roots in $\AP{c}$ by definition, and since $\nu_c$ maps the star of $\delta$ to $\d_\infty$, Proposition~\ref{eta nice}.\ref{nu tau} implies the following proposition, which allows us to complete the proof of Theorem~\ref{expand d inf}.

\begin{proposition}\label{c sym d inf}
Suppose $B$ is an acyclic exchange matrix of affine type.
A point in $V^*$ has a finite $\eta_{12\cdots n}^{B^T}$-orbit if and only if that point is in $\d_\infty$.
\end{proposition}

\begin{proof}[Proof of Theorem~\ref{expand d inf}]
One can easily check that the sequence $12\cdots n$ is a mutation-symmetry of $B$.
Specifically, one verifies that each mutation step has the effect of negating one row and column of $B$ while leaving every other entry unchanged, and that the row/column affected is different at each step.

The theorem is about an exchange matrix of affine acyclic type, which is a well behaved case, in the sense that $\Theta$ is the entire lattice~$P$.
(In fact, \cite[Proposition~0.14]{GHKK} says that acyclicity of~$B$ is enough to imply that $\Theta=P$.)
Thus if we adopt the additional hypothesis of signed-nondegenerating coefficients, we obtain the conclusion of Theorem~\ref{expand d inf} from Theorem~\ref{finite orbit}.
However, we can remove that additional hypothesis in light of Proposition~\ref{struct plus plus}.
\end{proof}

We mention some further useful facts about the action of $\eta_{12\cdots n}^{B^T}$ and~$c$.

\begin{lemma}\label{delta perp}
The hyperplane $\delta^\perp\subset V^*$ is the linear span of all vectors in $V^*$ that are in finite $c$-orbits.
\end{lemma}
\begin{proof}
Propositions~\ref{eta nice} and~\ref{c sym d inf} imply that every point in the codimension-$1$ cone $\d_\infty\subset\delta^\perp$ is in a finite $c$-orbit, so the same is true for every point in $\delta^\perp$.
Since $c$ does not have a basis of eigenvectors, the same can't be true for all of $V^*$.
\end{proof}

\begin{lemma}\label{dom of def}
Suppose $x\in V^*$ and let $m$ be the least common multiple of the sizes of finite $c$-orbits in $V^*$.
\begin{enumerate}[label=\bf\arabic*., ref=\arabic*] 
\item \label{dom of def +}
If $\br{x,\delta}>0$, then for $a\in\reals$ large enough, $x+a\nu_c(\delta)$ is in the same domain of definition of $\bigl(\eta_{12\cdots n}^{B^T}\bigr)^i$ as $\d_\infty$ for all $i\ge0$.
On this domain, $\bigl(\eta_{12\cdots n}^{B^T}\bigr)^i=c^i$.
Also, $\bigl(\eta_{12\cdots n}^{B^T}\bigr)^i(x+a\nu_c(\delta))=\bigl(\eta_{12\cdots n}^{B^T}\bigr)^i(x)+a\nu_c(\delta)$ and there exists $a'\ge0$ with $\bigl(\eta_{12\cdots n}^{B^T}\bigr)^{i+m}(x+a\nu_c(\delta))=\bigl(\eta_{12\cdots n}^{B^T}\bigr)^i(x+a\nu_c(\delta))+a'\nu_c(\delta)$ for all $i\ge0$.
\item \label{dom of def -}
If $\br{x,\delta}<0$, then the same is true with $i\le0$ replacing $i\ge0$ and $i-m$ replacing $i+m$.
\end{enumerate}
\end{lemma}
\begin{proof}
We argue Assertion~\ref{dom of def +}.
The argument for Assertion~\ref{dom of def -} is essentially the same.
Proposition~\ref{eta nice} says that all of $\d_\infty$ is in the same domain of definition of $\eta_{12\cdots n}^{B^T}$, which agrees with $c$ on that domain.
The argument for Proposition~\ref{eta nice} in \cite[Proposition~7.33]{affscat} shows that $\nu_c(\delta)$ is in the interior of this domain of definition.
(In each of the $n$ steps of the map, the image of $\d_\infty$ is on one side of the boundary of domains of linearity, not contained in the boundary.)
The case $i=1$ of Assertion~\ref{dom of def +} follows.

Lemma~\ref{delta perp} says that every vector in $\delta^\perp$ is in a finite $c$-orbit.
Choose $a$ large enough so that $x+a\nu_c(\delta)$ is in the same domain of definition of $\eta_{12\cdots n}^{B^T}$ as $\d_\infty$.
Since $\nu_c(\delta)$ is fixed by $c$, we have $\eta_{12\cdots n}^{B^T}(x+a\nu_c(\delta))=\eta_{12\cdots n}^{B^T}(x)+a\nu_c(\delta)$.
Make $a$ larger if necessary so that $\eta_{12\cdots n}^{B^T}(x+a\nu_c(\delta))$ is also in that domain of definition.
Continue until $a$ is large enough so that $\bigl(\eta_{12\cdots n}^{B^T}\bigr)^i(x+a\nu_c(\delta))$ is in that domain of definition for $i=0,1,\ldots,m-1$.

There is also a generalized $1$-eigenvector $\v$ for $c$ with $\br{\v,\delta}>0$, defined by the property that $c\v=\v+\nu_c(\delta)$.
Write $x+a\nu_c(\delta)$ as a vector in $\delta^\perp$ plus $b\v$ for some $b>0$.
Since $\bigl(\eta_{12\cdots n}^{B^T}\bigr)^m$ agrees with $c^m$ on $(x+a\nu_c(\delta))$, we see that $\bigl(\eta_{12\cdots n}^{B^T}\bigr)^m(x+a\nu_c(\delta))=(x+a\nu_c(\delta))+mb\nu_c(\delta)$.
Since the domain of definition of $\eta_{12\cdots n}^{B^T}$ containing~$\d_\infty$ is a convex cone with $\nu_c(\delta)$ in its relative interior, ${(x+a\nu_c(\delta))+mb\nu_c(\delta)}$ is also in that domain.
Continuing in this way, we see that $\bigl(\eta_{12\cdots n}^{B^T}\bigr)^i(x+a\nu_c(\delta))$ is in that domain of definition for $i\ge0$.
\end{proof}

\subsection{Dominance regions and $B$-cones}\label{dom sec} 
Suppose $\lambda\in V^*$.
The \newword{(integral) dominance region} of $\lambda$ with respect to $B$ is the intersection, over all sequences $\kk$ of indices in $\set{1,\ldots,n}$, of the sets $\bigl(\eta_{\kk}^{B^T}\bigr)^{-1}\sett{\eta_\kk^{B^T}(\lambda)+\mu_\kk(B)\alpha:\alpha\in Q^{\ge0}}$.
Here, each $\alpha$ is interpreted as a column vector of simple-root coordinates and the matrix product $\mu_\kk(B)\alpha$ is interpreted as the fundamental-weight coordinates of a vector in $V^*$.
The idea behind the dominance region goes back to work of Fan Qin \cite{FanQin} and was developed further by Rupel and Stella~\cite{RSDom}.
We will quote a result of \cite{canonical} that, in some cases, relates structure constants for theta functions with dominance regions.
We will also quote a result of \cite{affdomreg} that computes the dominance region of a point in the imaginary wall.
The combination of these two results is a critical tool for the proofs in this paper.

The following theorem was proved as \cite[Theorem~5.4]{canonical}.
A stronger but more complicated statement is \cite[Theorem~5.7]{canonical}.

\begin{theorem}\label{B cone prod} 
Suppose that $\tB$ has signed-nondegenerating coefficients and that $\lambda_1,\ldots,\lambda_\ell$ are all contained in the same $B$-cone.
Write $\lambda=a_1\lambda_1+\cdots+a_\ell \lambda_\ell$ for nonnegative integers $a_1,\ldots,a_\ell$.
Then there exist constants $c_{\kappa,\beta}\in\k$ such that $\thet_{\lambda_1}^{a_1}\cdots\thet_{\lambda_\ell}^{a_\ell}=\thet_\lambda+\sum_\kappa\sum_\beta c_{\kappa,\beta}y^\beta\thet_\kappa$, summing over $\kappa$  in the integral dominance region of $\lambda$ with respect to $B$  and $\beta\in Q^+$ such that $\kappa=\lambda+B\beta$.  
\end{theorem}

The following theorem is a version of \cite[Theorem~4.58]{affdomreg}. 
We insert the hypotheses that $B$ is acyclic and that $\tB=B$.
The theorem in \cite{affdomreg} also has ${\lambda+aB\delta^B}$ where we have written $\lambda-2a\nu_c(\delta)$.
The statement here is correct, because the vector~$\delta^B$ from \cite{affdomreg} coincides with our $\delta$ when $B$ is acyclic, and because \cite[Lemma~4.15]{affdomreg} says that $\nu_c(\delta)=-\frac12B\delta$.

\begin{theorem}\label{affine dom}
Suppose $B$ is an acyclic exchange matrix of affine type.
If $\lambda\in P$ is in the relative interior of the imaginary wall~$\d_\infty$, then the integral dominance region of $\lambda$ is $\bigl\{\lambda-2a\nu_c(\delta):0\le a\in\integers\bigr\}\cap\d_\infty$.
\end{theorem}

We combine these two theorems as follows, appealing to Proposition~\ref{struct plus plus} to remove the hypothesis of signed-nondegenerating coefficients.

\begin{theorem}\label{B cone prod aff}
Suppose $B$ is an acyclic exchange matrix of affine type and~$\tB$ is an extension of~$B$.
Suppose $C$ is an \emph{imaginary} cone in~$\F_{B^T}$, suppose $\lambda_i\in C$ for $i=1,\ldots,\ell$, suppose $\lambda=a_1\lambda_1+\cdots+a_\ell \lambda_\ell$ with $a_1,\ldots,a_\ell\ge0$ and suppose $\v=\thet_{\lambda_1}^{a_1}\cdots\thet_{\lambda_\ell}^{a_\ell}$.
Then $\v$ is a $\k[y]$-linear combination of theta functions~$\thet_\kappa$, with each $\kappa$ in $\bigl\{\lambda-2a\nu_c(\delta):0\le a\in\integers\bigr\}\cap\d_\infty=\bigl\{\lambda-2a\nu_c(\delta):0\le a\in\integers\bigr\}\cap C$.
\end{theorem}

In the theorem, the fact that each $\kappa$ is in $C$ follows from the fact that each~$\kappa$ is in $\bigl\{\lambda-2a\nu_c(\delta):0\le a\in\integers\bigr\}\cap\d_\infty$, because $C$ has the imaginary ray as an extreme ray and because~$\d_\infty$ is the union of the cones in the star of the imaginary ray.

The theorem has stronger hypotheses than Theorem~\ref{expand d inf}, because it requires that the monomial be a product of theta functions within the same imaginary cone, rather than only requiring that they all be in $\d_\infty$.
It also has a significantly stronger conclusion than Theorem~\ref{expand d inf}, because it restricts the indices $\kappa$ of theta functions to a finite set of points in a line segment in the same imaginary cone. 

\subsection{Periodic broken lines}\label{per sec}
For the remainder of Section~\ref{tools sec}, $B$ is acyclic of affine type and~$m$ is the least common multiple of the sizes of finite $c$-orbits in $V^*$.
We will not state these hypotheses explicitly in results, but the appearance of the imaginary wall $\d_\infty$ in the statements of results is a reminder that $B$ is acyclic of affine type.
In light of Propositions~\ref{eta nice} and~\ref{c sym d inf}, $\bigl(\eta_{12\cdots n}^{B^T}\bigr)^m$ fixes~$\d_\infty$ pointwise.

Furthermore, since we use mutation of broken lines, we assume for the remainder of Section~\ref{tools sec} that $\tB$ has signed-nondegenerating coefficients.
However, because of Proposition~\ref{struct plus plus}, the tools we build here will eventually lead to proofs of the theorems stated in Section~\ref{res sec}, most of which place no restriction on $\tB$.

Also, we will use the conditions $h\gg0$ and $h\ll0$ with a specific meaning.
A sentence like ``If $h\gg0$, then [assertion]'' means that there exists $H\in\integers$ such that the statement holds for all $h>H$.
A similar sentence with $h\ll0$ means that there exists $H\in\integers$ such that the statement holds for all $h<H$.

Choose $\lambda\in P\cap\d_\infty$ and an imaginary cone $C$ of~$\F_{B^T}$ with $\lambda\in C\subseteq\d_\infty$.
Since~$C$ is imaginary, one of its extreme rays might be spanned by $\nu_c(\delta)$, but all of its other extreme rays (or all of its extreme rays) are in the $\g$-vector fan.
Choose a full-dimensional cone $C_0$ of $\F_{B^T}$ (i.e. maximal cone of the $\g$-vector fan) that contains all of the extreme rays of $C$ that are in the $\g$-vector fan.
Take $\chi_0$ in the interior of $C_0$ and set $\chi_h=\bigl(\eta_{12\cdots n}^{B^T}\bigr)^{hm}(\chi_0)$ for $h\in\integers$.
We call $(\chi_h:h\in\integers)$ a \newword{chi sequence for~$\lambda$}.
In what follows, when we quantify $\lambda\in P\cap\d_\infty$, we will assume that ${(\chi_h:h\in\integers)}$ is a chi sequence for~$\lambda$.

Suppose $p\in P\cap\d_\infty$ and $h\in\integers$.
Let $\bl$ be a broken line for $p$ with endpoint~$\chi_h$.
If $h\ge0$, define $\bl^{(\ell)}$ to be $\bigl(\eta_{12\cdots n}^{B^T}\bigr)^{\ell m}(\bl)$.
If $h<0$, define $\bl^{(\ell)}$ to be $\bigl(\eta_{12\cdots n}^{B^T}\bigr)^{-\ell m}(\bl)$.
From each $\bl^{(\ell)}$, we read a monomial $\const_{\bl^{(\ell)}}x^{\lambda_{\bl^{(\ell)}}}y^{\beta_{\bl^{(\ell)}}}$ as explained in Section~\ref{scat sec}.
We say that $\bl$ is \newword{periodic} if the sequence $(\const_{\bl^{(\ell)}}x^{\lambda_{\bl^{(\ell)}}}y^{\beta_{\bl^{(\ell)}}}:\ell=0,1,\ldots)$ is periodic.
Equivalently, since mutation maps on broken lines don't change the coefficients of the corresponding monomials, the sequence $(x^{\lambda_{\bl^{(\ell)}}}y^{\beta_{\bl^{(\ell)}}}:\ell=0,1,\ldots)$ is periodic.
We say that a pair $(\bl_1,\bl_2)$ is periodic if $\bl_1$ and $\bl_2$ are both periodic.

Given $\lambda,p_1,p_2\in P\cap\d_\infty$, define  
${\ap_{\chi_h}(p_1,p_2,\lambda)=\sum_{(\bl_1,\bl_2)}\const_{\bl_1}\const_{\bl_2} y^{\beta_{\bl_1}+\beta_{\bl_2}}}$,
summing over periodic pairs $(\bl_1,\bl_2)$ of broken lines for $p_1$ and $p_2$, with $\lambda_{\bl_1}+\lambda_{\bl_2}=\lambda$, both having endpoint $\chi_h$.
Thus $\ap_{\chi_h}(p_1,p_2,\lambda)$ is obtained from $a_{\chi_h}(p_1,p_2,\lambda)$ by the deleting terms from pairs $(\bl_1,\bl_2)$ such that $\bl_1$ or~$\bl_2$ is not periodic.

\begin{theorem}\label{mut pair affine periodic} 
Suppose $\tB$ has signed-nondegenerating coefficients and $\lambda,p_1,p_2\in P\cap\d_\infty$.   
If $h\gg0$ or $h\ll0$, then $a(p_1,p_2,\lambda)=\ap_{\chi_h}(p_1,p_2,\lambda)$.
\end{theorem}

We now prepare to prove Theorem~\ref{mut pair affine periodic}.
The following proposition, which refers to the sets $\d_\infty^+$ and $\d_\infty^-$ defined in Section~\ref{scatfan aff sec}, is \cite[Proposition~4.17]{affdomreg}.
In particular, Proposition~\ref{delta limit eta}.\ref{in dinf+-} is the concatenation of Assertions~1 and ~3 of \cite[Proposition~4.17]{affdomreg} and Proposition~\ref{delta limit eta}.\ref{c agree} is Assertion~2 of \cite[Proposition~4.17]{affdomreg} plus the observation in Proposition~\ref{eta nice}.\ref{dinf dom of def} that all of $\d_\infty$ is contained in one domain of definition of~$\eta^{B^T}_{12\cdots n}$.

\begin{proposition}\label{delta limit eta}  
Suppose $x\in V^*\setminus\d_\infty$.
\begin{enumerate}[label=\bf\arabic*., ref=\arabic*] 
\item \label{in dinf+-}
If $h\gg0$ then $\bigl(\eta_{12\cdots n}^{B^T}\bigr)^h(x)\in\d_\infty^+$ and if $h\ll0$, then $\bigl(\eta_{12\cdots n}^{B^T}\bigr)^h(x)\in\d_\infty^-$.
\item \label{c agree}  
If $h\gg0$ or $h\ll0$, then $\bigl(\eta_{12\cdots n}^{B^T}\bigr)^h(x)$ and $\d_\infty$ are in the same domain of definition of $\eta_{12\cdots n}^{B^T}$ and $\eta_{12\cdots n}^{B^T}$ agrees with $c$ on $\bigl(\eta_{12\cdots n}^{B^T}\bigr)^h(x)$.
\item \label{pos a}
There exists positive real $a$ such that $\bigl(\eta_{12\cdots n}^{B^T}\bigr)^{h+m}(x)=\bigl(\eta_{12\cdots n}^{B^T}\bigr)^h(x)+a\nu_c(\delta)$ when $h\gg0$ and $\bigl(\eta_{12\cdots n}^{B^T}\bigr)^{h-m}(x)=\bigl(\eta_{12\cdots n}^{B^T}\bigr)^h(x)+a\nu_c(\delta)$ when $h\ll0$.
\item \label{eta lim}
$\displaystyle\lim_{h\to\infty}\frac{\bigl(\eta_{12\cdots n}^{B_0^T}\bigr)^h(x)}{\bigl|\bigl(\eta_{12\cdots n}^{B_0^T}\bigr)^h(x)\bigr|}=\lim_{h\to-\infty}\frac{\bigl(\eta_{12\cdots n}^{B^T}\bigr)^h(x)}{\bigl|\bigl(\eta_{12\cdots n}^{B^T}\bigr)^h(x)\bigr|}=\frac{\nu_c(\delta)}{|\nu_c(\delta)|}$.
\end{enumerate}
\end{proposition}

Proposition~\ref{delta limit eta} will be useful several places in the paper. 
First of all, Proposition~\ref{delta limit eta}.\ref{c agree} allows us to prove the following crucial lemma.

\begin{lemma}\label{per char}
For $\tB$ with signed-nondegenerating coefficients, suppose $\lambda,p_1,p_2\in P\cap\d_\infty$ and $h\gg0$ or $h\ll0$.
Given broken lines $(\bl_1,\bl_2)$, both having endpoint $\chi_h$, with $\lambda_{\bl_1}+\lambda_{\bl_2}=\lambda$, the following are equivalent.
\begin{enumerate}[label=\rm(\roman*), ref=(\roman*)]
\item \label{both per} 
The pair $(\bl_1,\bl_2)$ is periodic.
\item \label{both per beta} 
$(\beta_{\bl_1^{(\ell)}}:\ell=0,1,\ldots)$ is periodic and $(\beta_{\bl_2^{(\ell)}}:\ell=0,1,\ldots)$ is periodic.
\item \label{is per} 
$(\beta_{\bl_1^{(\ell)}}+\beta_{\bl_2^{(\ell)}}:\ell=0,1,\ldots)$ is periodic.
\item \label{is ess}
$(\beta_{\bl_1^{(\ell)}}+\beta_{\bl_2^{(\ell)}}:\ell=0,1,\ldots)$ has finitely many distinct entries.  
\item \label{no inf} 
Some entry appears infinitely many times in $(\beta_{\bl_1^{(\ell)}}+\beta_{\bl_2^{(\ell)}}:\ell=0,1,\ldots)$.
\end{enumerate}
\end{lemma}
\begin{proof}
The implications $\ref{both per}\implies\ref{both per beta}\implies\ref{is per}\implies\ref{is ess}\implies\ref{no inf}$ are trivial.
Suppose some vector~$\v$ appears infinitely many times in $(\beta_{\bl_1^{(\ell)}}+\beta_{\bl_2^{(\ell)}}:\ell=0,1,\ldots)$.
Since each $\beta_{\bl_1^{(j)}}$ and~$\beta_{\bl_2^{(j)}}$ is a nonnegative linear combination of simple roots, the vector $\v$ can be obtained in only finitely many ways as a sum of $\beta_{\bl_1^{(j)}}$ and $\beta_{\bl_2^{(j)}}$.  
Therefore, we can choose $k>j\ge0$ such that not only $\beta_{\bl_1^{(j)}}+\beta_{\bl_2^{(j)}}=\beta_{\bl_1^{(k)}}+\beta_{\bl_2^{(k)}}$ but also $\beta_{\bl_1^{(j)}}=\beta_{\bl_1^{(k)}}$ and $\beta_{\bl_2^{(j)}}=\beta_{\bl_2^{(k)}}$.
Furthermore, by the definition of broken lines, $\lambda_{\bl_1^{(\ell)}}$ depends only on $p_1$ and $\beta_{\bl_1^{(\ell)}}$, so $\lambda_{\bl_1^{(j)}}=\lambda_{\bl_1^{(k)}}$ and $\lambda_{\bl_2^{(j)}}=\lambda_{\bl_2^{(k)}}$.

If $h\gg0$, then we can assume by Proposition~\ref{delta limit eta}.\ref{c agree} that $\chi_h$ is in the same domain of definition of $\eta_{12\cdots n}^{B^T}$ as $\d_\infty$ and that $\eta_{12\cdots n}^{B^T}$ acts on $\chi_h$ by the Coxeter element~$c$.
For the same reason, the same is true for $\chi_{h+\ell}$ for all $\ell\ge0$.
The map $\bigl(\eta_{12\cdots n}^{B^T}\bigr)^m$ on broken lines thus acts on the $\lambda_{\bl_1^{(\ell)}}$ by~$c^m$.
In particular, knowing $\lambda_{\bl_1^{(j)}}$ for one $j\ge0$ determines~$\lambda_{\bl_1^{(\ell)}}$ for all $\ell\ge0$, and indeed, the definition of $m$ implies that $\lambda_{\bl_1^{(\ell)}}=\lambda_{\bl_1^{(0)}}$ for all $\ell\ge0$.

Proposition~\ref{c sym d inf} implies that $\bigl(\eta_{12\cdots n}^{B^T}\bigr)^m(p_1)=p_1$.
Thus, by inspection of the definition of $\eta_{12\cdots n}^{B^T}$ as a map on curves, we see that knowing $\lambda_{\bl_1^{(j)}}$ and $\beta_{\bl_1^{(j)}}$ for one $j\ge0$ also determines~$\beta_{\bl_1^{(\ell)}}$ for all $\ell\ge0$.
For any $\ell\ge0$, the vector $\beta_{\bl_1^{(\ell+k-j)}}$ is obtained by applying $\bigl(\eta_{12\cdots n}^{B^T}\bigr)^{(\ell-j)m}$ to $\bl_1^{(k)}$ and reading off the monomial associated to the resulting curve.
But since $\beta_{\bl_1^{(j)}}=\beta_{\bl_1^{(k)}}$ and $\lambda_{\bl_1^{(j)}}=\lambda_{\bl_1^{(k)}}$, the same vector is obtained by applying $\bigl(\eta_{12\cdots n}^{B^T}\bigr)^{(\ell-j)m}$ to $\bl_1^{(j)}$ and reading the monomial.
The resulting curve is $\bl_1^{(\ell)}$, and we conclude that $\beta_{\bl_1^{(\ell+k-j)}}=\beta_{\bl_1^{(\ell)}}$.

We have shown that $\bl_1$ is periodic in the case where $h\gg0$. 
The proof in the case $h\ll0$ is similar, and the proof that $\bl_2$ is periodic is the same.
\end{proof}

Theorem~\ref{Theta facts} says that $a(p_1,p_2,\lambda)$ is a polynomial.
In fact, each $a_{\chi_h}(p_1,p_2,\lambda)$ is a polynomial as we now explain.
Since $\d_\infty$ is a wall in $\Scat^T(B)$, every point that is not contained in any wall of $\Scat^T(B)$ is in some cone of the $\g$-vector fan.
Since $\Theta=P$ in affine type, \cite[Proposition~7.1]{GHKK} says that there are only finitely broken lines for $p_1$ with endpoint $\chi_h$, and similarly for $p_2$.

We now establish that the polynomials $a_{\chi_h}(p_1,p_2,\lambda)$ limit to $a(p_1,p_2,\lambda)$ in the sense of formal power series as $h\to\infty$ or $h\to-\infty$.
Furthermore, we prove that as we approach the limit, there is a one-to-one correspondence between the terms of one polynomial in the sequence and the terms of the next polynomial.
These very strong constraints on $a(p_1,p_2,\lambda)$  are at the heart of the proof of Theorem~\ref{mut pair affine periodic}.

\begin{proposition}\label{mut pair affine}
For $\tB$ with signed-nondegenerating coefficients, let $\lambda,p_1,p_2\in P\cap\d_\infty$.
\begin{enumerate}[label=\bf\arabic*., ref=\arabic*] 
\item \label{mut pair d+}
If $h\gg0$, then the map $\bigl(\eta_{12\cdots n}^{B^T}\bigr)^m$ takes pairs contributing to $a_{\chi_h}(p_1,p_2,\lambda)$ bijectively to pairs contributing to $a_{\chi_{h+1}}(p_1,p_2,\lambda)$, with inverse $\bigl(\eta_{12\cdots n}^{B^T}\bigr)^{-m}$.
\item \label{mut pair d-}
If $h\ll0$, then the map $\bigl(\eta_{12\cdots n}^{B^T}\bigr)^{-m}$ takes pairs contributing to $a_{\chi_h}(p_1,p_2,\lambda)$ bijectively to pairs contributing to $a_{\chi_{h-1}}(p_1,p_2,\lambda)$, with inverse $\bigl(\eta_{12\cdots n}^{B^T}\bigr)^m$.
\item \label{nicer limit}
$\displaystyle a(p_1,p_2,\lambda)=\lim_{h\to\infty}a_{\chi_h}(p_1,p_2,\lambda)=\lim_{h\to-\infty}a_{\chi_h}(p_1,p_2,\lambda)$.
\item \label{mut pair d+ periodic}
If $h\gg0$, then the bijection of Assertion~\ref{mut pair d+} restricts to a bijection from periodic pairs contributing to $a_{\chi_h}(p_1,p_2,\lambda)$ to periodic pairs contributing to $a_{\chi_{h+1}}(p_1,p_2,\lambda)$.
\item \label{mut pair d- periodic}
If $h\ll0$, then the bijection of Assertion~\ref{mut pair d-} restricts to a bijection from periodic pairs contributing to $a_{\chi_h}(p_1,p_2,\lambda)$ to periodic pairs contributing to $a_{\chi_{h-1}}(p_1,p_2,\lambda)$.
\end{enumerate}
\end{proposition}
\begin{proof}
In light of Proposition~\ref{c sym d inf} and the definition of $m$, Assertions~\ref{mut pair d+} and~\ref{mut pair d-} are obtained by combining Propositions~\ref{mut pair} and~\ref{delta limit eta}.\ref{c agree}.
Assertions~\ref{mut pair d+ periodic} and~\ref{mut pair d- periodic} then follow by the definition of periodic broken lines. 
For $C$ and $C_0$ as in the definition of a chi sequence, define $C_h=\bigl(\eta_{12\cdots n}^{B^T}\bigr)^{hm}(C_0)$ for ${h\in\integers}$.
Proposition~\ref{delta limit eta}.\ref{eta lim} implies that the $C_h$ approach an imaginary cone $C'$ that contains~$C$ as $h\to\pm\infty$.  
Thus we can choose~$\lambda_h$ in the interior of each~$C_h$ in such a way that ${\lim_{h\to\infty}\lambda_h=\lim_{h\to-\infty}\lambda_h=\lambda}$.
By definition, $a(p_1,p_2,\lambda)=\lim_{h\to\infty}a_{\lambda_h}(p_1,p_2,\lambda)$.
By Lemma~\ref{z z'}, this limit is $\lim_{h\to\infty}a_{\chi_h}(p_1,p_2,\lambda)$.
The same is true taking the limit as $h\to-\infty$.
We have proved Assertion~\ref{nicer limit}.
\end{proof}

\begin{proof}[Proof of Theorem~\ref{mut pair affine periodic}]
We argue for $h\gg0$.
The proof for $h\ll0$ is essentially the same.
Appealing to Lemma~\ref{per char}, let $k$ be the least common multiple of the periods of sequences $(\beta_{\bl_1^{(\ell)}}+\beta_{\bl_2^{(\ell)}}:\ell=0,1,\ldots)$ for periodic pairs $(\bl_1,\bl_2)$ contributing to $a(p_1,p_2,\lambda)$.
Lemma~\ref{mut pair affine}.\ref{mut pair d+ periodic} implies that ${(\ap_{\chi_{h+k\ell m}}(p_1,p_2,\lambda):\ell=0,1,\ldots)}$ is constant.
Proposition~\ref{mut pair affine}.\ref{nicer limit} implies that $a(p_1,p_2,\lambda)=\lim_{\ell\to\infty}a_{\chi_{h+k\ell m}}(p_1,p_2,\lambda)$.  
Writing temporarily $\anp_{\chi_j}(p_1,p_2,\lambda)$ for $a_{\chi_j}(p_1,p_2,\lambda)-\ap_{\chi_j}(p_1,p_2,\lambda)$ for all $j$, we see that $a(p_1,p_2,\lambda)=\ap_{\chi_h}(p_1,p_2,\lambda)+\lim_{\ell\to\infty}\anp_{\chi_{h+k\ell m}}(p_1,p_2,\lambda)$.
For each pair $(\bl_1,\bl_2)$ for $\anp(p_1,p_2,\lambda)$, Lemma~\ref{per char} says that no entry of ${(\beta_{\bl_1^{(\ell)}}+\beta_{\bl_2^{(\ell)}}:\ell=0,1,\ldots)}$ appears infinitely many times.
Thus no entry of ${(\beta_{\bl_1^{(k\ell)}}+\beta_{\bl_2^{(k\ell)}}:\ell=0,1,\ldots)}$ appears infinitely many times.  
But each $a_{\chi_{h+k\ell m}}(p_1,p_2,\lambda)$ is a polynomial, so $\anp_{\chi_{h+k\ell m}}(p_1,p_2,\lambda)$ has finitely many terms.
We conclude that $\lim_{\ell\to\infty}\anp_{\chi_{h+k\ell m}}(p_1,p_2,\lambda)=0$.
\end{proof}

We also point out two more useful facts.

\begin{lemma}\label{per finite}    
Suppose $\tB$ has signed-nondegenerating coefficients.
If ${p\in P\cap\d_\infty}$, $h\gg0$ or $h\ll0$, and $\bl$ is a periodic broken line for $p$ with endpoint~$\chi_h$, then~${\lambda_\bl\in\delta^\perp}$.
\end{lemma}
\begin{proof}
We prove the case where $h\gg0$.  
The case where $h\ll0$ is similar.
As in the proof of Lemma~\ref{per char}, we can assume by Proposition~\ref{delta limit eta}.\ref{c agree} that $\chi_h$ is in the same domain of definition of $\eta_{12\cdots n}^{B^T}$ as $\d_\infty$, so that $\eta_{12\cdots n}^{B^T}$ acts on $\chi_h$ by the Coxeter element~$c$ and the map $\bigl(\eta_{12\cdots n}^{B^T}\bigr)^m$ on broken lines acts on the $\lambda_{\bl^{(\ell)}}$ by~$c^m$.
Periodicity of $\bl$ thus implies that $\lambda_\bl$ is in a finite $c$-orbit, so $\lambda_\bl\in\delta^\perp$ by Lemma~\ref{delta perp}.
\end{proof}

\begin{lemma}\label{no dinf}
Suppose $\tB$ has signed-nondegenerating coefficients and $p\in P\cap\d_\infty$.
If $\bl$ is a periodic broken line for $p$ with endpoint~$\chi_h$, then~$\bl$ is disjoint from $\d_\infty$.
\end{lemma}
\begin{proof}
Suppose for the sake of contradiction that $\bl$ passes through $\d_\infty$ at a point~$q$.
By the definition of a broken line, $q$ is in the relative interior of $\d_\infty$, so there is a point $r$ in $\d_\infty^+$ on a domain of $\bl$ containing~$q$.
As a curve,~$\bl^{(\ell)}$ is $\bigl(\eta^{B^T}_{12\cdots n}\bigr)^{\pm\ell m}(\bl)$ for all $\ell\ge0$, with the sign of the $\pm$ determined by whether $h\ge0$ or $h<0$.
The action of $\bigl(\eta^{B^T}_{12\cdots n}\bigr)^{\pm m}$ fixes $q$.
Proposition~\ref{delta limit eta}.\ref{pos a} says that for $\ell\gg0$, additional powers of $\bigl(\eta^{B^T}_{12\cdots n}\bigr)^{\pm m}$ take $r$ without bound in the direction of $\nu_c(\delta)$.
Therefore, as $\ell\to\infty$, the derivative of the domain $L^{(\ell)}$ of $\bl^{(\ell)}$ containing~$q$ and~$r$ approaches larger and larger multiples of $\pm\nu_c(\delta)$.
Thus $\lambda_{L^{(\ell)}}$ attains infinitely many values as $\ell\to\infty$.
Since by the definition of a broken line, $\lambda_{\bl^{(\ell)}}$ depends only on $p$ and $\beta_{\bl^{(\ell)}}$, $\beta_{L^{(\ell)}}$ also attains infinitely many values. 
In particular, there is no upper bound on the sum of the simple-root coordinates of the vectors $\beta_{L^{(\ell)}}$, and therefore there is also no upper bound on the sum of the simple-root coordinates of the vectors $\beta_{\bl^{(\ell)}}$.
We see that $\beta_{\bl^{(\ell)}}$ takes infinitely many values as $\ell\to\infty$, contradicting the fact that~$\bl$ is periodic.
\end{proof}


\subsection{Mutating to one side of the imaginary wall}\label{mut imag sec} 
In this section, we prove several more specific theorems about periodic broken lines.
The basic tool in these proofs is to apply powers of $\eta_{12\cdots n}^{B^T}$ in order to move important parts of a broken line into $\d_\infty^+$ or $\d_\infty^-$ (the ``positive side'' or ``negative side'' of $\d_\infty$), defined in Section~\ref{scatfan aff sec}.
That is useful because of the following lemma.

\begin{lemma}\label{scatter region}
Suppose a broken line bends at a point $q$ on a wall with positive normal vector~$\beta$ and picks up a constant times a monomial in $y$ times a Laurent monomial~$x^\lambda$.
\begin{enumerate}[label=\bf\arabic*., ref=\arabic*] 
\item \label{bend fin}
If $\beta$ is in a finite $c$-orbit, then $\br{\lambda,\delta}=0$.
\item \label{bend inf +}
If $\beta$ is in an infinite $c$-orbit and $q\in\d_\infty^+$, then $\br{\lambda,\delta}<0$.
\item \label{bend inf -}
If $\beta$ is in an infinite $c$-orbit and $q\in\d_\infty^-$, then $\br{\lambda,\delta}>0$.
\end{enumerate}
\end{lemma}

\begin{proof}
Throughout, we use \cite[Proposition~4.10]{affdomreg}, which says that $x\in V$ is in a finite $c$-orbit if and only if $\omega_c(\delta,x)=0$.
The term $\bl$ picks up at $q$ is a constant times $(\hy^\beta)^a=(y^\beta x^\xi)^a$, for some $a\ge0$, where $\xi=\omega_c(\,\cdot\,,\beta)\in V^*$.
Thus, $\lambda=a\xi$ and $\br{\lambda,\delta}=a\omega_c(\delta,\beta)$.
Assertion~\ref{bend fin} follows.
The root~$\beta$ is a positive scalar multiple of $\beta_0+k\delta$ for some $\beta_0\in\RSfin$ and $k\in\integers$.
Because $\omega_c$ is skew-symmetric, $\omega_c(\beta,\delta)$ is a positive scalar multiple of $\omega_c(\beta_0,\delta)$.
Since $\beta$ is positive, either $\beta_0$ is positive and $k\ge0$ or $\beta_0$ is negative and $k\ge1$.
Because $q$ is on the wall, $\br{q,\beta}=\br{q,\beta_0+k\delta}=0$.

Suppose~$q$ is in $\d_\infty^+$.
In particular, $\br{q,\delta}>0$.
If $\omega_c(-\beta_0,\delta)>0$, then ${\br{q,-\beta_0}<0}$ by the definition of $\d_\infty^+$.
Therefore $\br{q,\beta_0+k\delta}>0$, and this contradiction shows that $\omega_c(-\beta_0,\delta)\le0$, so that $\omega_c(\delta,\beta_0)\le0$ and thus $\omega_c(\delta,\beta)\le0$.
Since $\beta$ is in an infinite $c$-orbit, $\omega_c(\delta,\beta)<0$, so $\br{\lambda,\delta}<0$.

If $q\in\d_\infty^-$ and $\omega_c(\beta_0,\delta)>0$, then $\br{q,\delta}<0$ and $\br{q,\beta_0}<0$, $\br{q,\beta_0+k\delta}<0$.
By this contradiction, $\omega_c(\beta_0,\delta)\le0$, so $\omega_c(\delta,\beta_0)\ge0$ and $\omega_c(\delta,\beta)>0$, so $\br{\lambda,\delta}>0$.
\end{proof}

\begin{theorem}\label{simplification for boundary dinf}
Suppose $\tB$ has signed-nondegenerating coefficients and $p\in P$ is in the relative boundary of $\d_\infty$.
If $h\gg0$ or $h\ll0$ and $\bl$ is a periodic broken line for $p$ with endpoint~$\chi_h$, then
\begin{enumerate}[label=\bf\arabic*., ref=\arabic*] 
\item \label{unb}
There is a maximal $\g$-vector cone containing $\bl((-\infty,t])$ in its interior for some~${t\in\reals}$.
\item \label{scat fin}
$\bl$ is contained in a finite union of maximal $\g$-vector cones.
\item \label{gamma in dinf}
All of $\bl$ is contained in $\d_\infty^+$ if $h\gg0$ or in $\d_\infty^-$ if $h\ll0$.
\item \label{scat only fin}
  $\bl$ only bends on walls whose normal vectors are in $\APTre{c}$.  
\end{enumerate}
\end{theorem}
\begin{proof}
Let $C$ be the cone of the mutation fan containing $p$ in its relative interior.
Since $p$ is in the relative boundary of $\d_\infty$, $C$ is a cone in the $\g$-vector fan.
Following the unbounded domain of $\bl$ in the direction of decreasing parameter, we eventually leave every cone of $\F_{B^T}$ except cones that contain $C$.
Since $\bl$ does not pass through the relative boundary of any wall of $\Scat^T(B)$ and does not pass through the intersection of any two walls of $\Scat^T(B)$ and since $\chi_h$ is not in any wall, no part of the unbounded domain of $\bl$ is contained in any of the walls that contain~$C$.
The complement in $V^*$ of the walls of $\Scat^T(B)$ is the union of the interiors of the maximal $\g$-vector cones, so there is a $t$ such that $\bl((-\infty,t])$ is contained in one of the maximal $\g$-vector cones that contain~$C$.
This is Assertion~\ref{unb}.

The image $\bl([t,0])$ is a compact subset of $V^*\setminus\d_\infty$.
The definition of a broken line ensures that each nonempty intersection of $\bl([t,0])$ with a maximal $\g$-vector cone is a union (because we haven't ruled out the possibility that a broken line croses a $\g$-vector cone more than once) of line segments connecting the relative interiors of two codimension\nobreakdash-$1$ faces of the cone, or connecting the relative interior of a codimension-$1$ face of the cone to an endpoint of $\bl([t,0])$.
For each such line segment $\sigma$, choose $\ep>0$ and define an open set as the union of open $\ep$-balls, one for each point in~$\sigma$.  
We can choose each of these $\ep$ such that each segment $\sigma$ contains a smaller segment~$\sigma'$ (not degenerated to a point) that is not in the open sets for the neighboring segments (or the one neighboring segment, when~$\sigma$ is at an endpoint of $\bl([t,0])$).
When two segments intersect, we may need to decrease some of the choices of $\epsilon$ to make sure that each segment~$\sigma$ contains a point not in any of the open sets for other segments.
Since $\bl$ has only finitely many domains of linearity, there are only finitely many intersections between different domains of linearity and thus only finitely many instances where two of the line segments intersect.  
Thus there are only finitely many of the $\ep$ that need to be made smaller.
By Lemma~\ref{no dinf}, these open sets constitute an open cover of $\bl([t,0])$.
There is one open set containing each segment, and each segment contains at least one point not in any of the other open sets.
Since $\bl([t,0])$ is compact, there are only finitely many segments and thus finitely many $\g$-vector cones intersected by $\bl([t,0])$.
Also, $\bl((-\infty,t))$ intersects only one $\g$-vector cone, and we have proved Assertion~\ref{scat fin}.

In light of Assertion~\ref{scat fin}, we can apply Proposition~\ref{delta limit eta}.\ref{in dinf+-} to a nonzero vector in each ray of each maximal $\g$-vector cone that intersects $\bl$ (except those rays that are in the relative boundary of $\d_\infty$) and conclude that, for large enough~$h$, the entire broken line is in $\d_\infty^+\cup\d_\infty$.
Then Lemma~\ref{no dinf} says that the entire broken line is in~$\d_\infty^+$.
This is Assertion~\ref{gamma in dinf} for $h\gg0$.
The assertion for $h\ll0$ is proved similarly.

By hypothesis, $p\in\d_\infty$, and by the definition of a broken line, the unbounded domain of linearity of $\bl$ is labeled by $x^p$.
By Assertion~\ref{gamma in dinf}, for $h\gg0$, all of $\bl$ is in $\d_\infty^+$, so Lemma~\ref{scatter region} says that if $\bl$ ever bends on a wall with normal vector in an infinite orbit, then $\br{\lambda_\bl,\delta}<0$.
This is ruled out by Lemma~\ref{per finite}, and we conclude that~$\bl$ does not bend on walls whose normal vectors are in infinite $c$-orbits.
Since~$\APT{c}$ is the set of roots of $\AP{c}$ that are in finite $c$-orbits, and by Lemma~\ref{no dinf}, we see that~$\bl$ only scatters on walls that are in $\APTre{c}$.
This is Assertion~\ref{scat only fin} for $h\gg0$ and the assertion for $h\ll0$ is proved similarly.
\end{proof}

\begin{theorem}\label{simplification for nuc delta} 
Suppose $\tB$ has signed-nondegenerating coefficients and ${p=k\nu_c(\delta)}$ and $h\gg0$.
Then there are exactly two periodic broken lines for $p$ with endpoint~$\chi_h$.
One broken line has only one domain of linearity and is contained in $\d_\infty^+$.
The final monomial for the other broken line $\bl$ is $x^{-p}y^{k\delta}$, and there exists $t\in(-\infty,0)$ such that $\bl((-\infty,t])\subseteq\d_\infty^{-}$.
If instead $h\ll0$, then the same is true, only switching $\d_\infty^+$ and $\d_\infty^-$.
\end{theorem}

\begin{proof} 
We prove this for $h\gg0$.
The proof for $h\ll0$ is essentially the same.

Suppose $\bl$ is a periodic broken line for $p$ with endpoint $\chi_h$.
The unbounded region of $\bl$ has derivative $-k\nu_c(\delta)$, so as $t\to-\infty$, $\bl$ is parallel to the imaginary ray and thus $\bl((-\infty,t])$ is on one side or the other of $\d_\infty$.
Thus for some $t$, we have $\bl((-\infty,t])\subseteq\d_\infty^+$ or $\bl((-\infty,t])\subseteq\d_\infty^-$.

For a different $\bl$, in the proof of Theorem~\ref{simplification for boundary dinf}, we argued that $\bl([t,0])$ intersects a finite number of maximal $\g$-vector cones.
The same proof works here for any $t\in(-\infty,0)$.
Lemma~\ref{dom of def} says that there is a choice of $t$ such that, for all $j\ge0$, the entire ray $\bl((-\infty,t])$ is in the same domain of linearity of $\bigl(\eta_{12\cdots n}^{B^T}\bigr)^j$ as $\d_\infty$.
Now, if $\bl((-\infty,t])\subseteq\d_\infty^+$, then arguing as in the proof of Proposition~\ref{simplification for boundary dinf}.\ref{gamma in dinf}, we see that for large enough~$h$, the entire broken line is in $\d_\infty^+$.
As in the proof of Proposition~\ref{simplification for boundary dinf}.\ref{scat only fin}, we see that~$\bl$ only bends on walls with normal vectors in~$\APTre{c}$.
However, the infinite domain of linearity of $\bl$ is parallel to all such walls by Lemma~\ref{E delta in tubes}, and we conclude that $\bl$ never bends.
The uniqueness of a periodic broken line for $p$ with endpoint~$\chi_h$ and $\bl((-\infty,t])\subseteq\d_\infty^+$ follows immediately.
The existence is also immediate.

If $\bl((-\infty,t])\subseteq\d_\infty^-$, then the ($h\ll0$ version of the) paragraph above shows that (fixing $h$), for $j\ll0$, the broken line $\bigl(\eta_{12\cdots n}^{B^T}\bigr)^{j m}(\bl)$ with endpoint $\chi_{h+j}$ has one domain of linearity.
If there are two periodic broken lines $\bl$ and $\bl'$ for $p$ with endpoint $\chi_h$ and unbounded domain contained in $\d_\infty^-$, then taking $j$ large enough for both, we see that $\bigl(\eta_{12\cdots n}^{B^T}\bigr)^{j m}(\bl)$ and $\bigl(\eta_{12\cdots n}^{B^T}\bigr)^{j m}(\bl')$ both have one domain of linearity.
Thus $\bigl(\eta_{12\cdots n}^{B^T}\bigr)^{j m}(\bl)=\bigl(\eta_{12\cdots n}^{B^T}\bigr)^{j m}(\bl')$, and so $\bl=\bl'$.
We have established the uniqueness of a periodic broken line for $p$ with endpoint $\chi_h$ and $\bl((-\infty,t])\subseteq\d_\infty^-$.
We complete the proof by proving the existence of a periodic broken line for $p$ with endpoint $\chi_h$ and $\bl((-\infty,t])\subseteq\d_\infty^-$ whose associated monomial is $x^{-p}y^{k\delta}$.

It will be enough to prove the existence for one choice of $\chi_0$, and it will then be true for all choices.
To see why, assume the existence is known for one choice, so that the conclusions of this theorem are known for that choice.
That $\chi_0$ is contained in some full-dimensional cone $C_0$ of $\F_{B^T}$.
The $\chi_h$ are a chi sequence for the vector~$0$.
Since the conclusions of this theorem are known for this choice, we can compute the structure constant $a(p,p,0)$ by Theorem~\ref{mut pair affine periodic}.
The two broken lines for $p$ have monomials $x^p$ and $x^{-p}y^{k\delta}$, so to choose two broken lines for $p$ such that the product of their monomials in $x$ variables is $1$, we must choose two \emph{different} broken lines, in either of the two orders, giving a structure constant $a(p,p,0)=2y^{k\delta}$.
Now, suppose that for some other choice of $\chi_0$, either no broken line exists for $p$ with $\bl((-\infty,t])\subseteq\d_\infty^-$ or the unique broken line for $p$ with $\bl((-\infty,t])\subseteq\d_\infty^-$ is associated to some monomial other than $x^{-p}y^{k\delta}$.
If no broken line exists, then $a(p,p,0)=0$, and if a broken line exists but is associated to some other monomial, then $a(p,p,0)$ is something other than $2y^{k\delta}$.
By this contradiction, we conclude that for any choice of $\chi_0$, there exists a broken line for $p$ with endpoint $\chi_h$ and $\bl((-\infty,t])\subseteq\d_\infty^-$ whose associated monomial in the $y$ variables is $y^{k\delta}$.

We prove the existence by exhibiting a ``bi-infinite'' broken line, satisfying the definition of a broken line, but with two infinite domains of linearity, one labeled~$x^p$ and the other $x^{-p}y^{k\delta}$.
Such a broken line is ``reversible'', in the sense that there is another bi-infinite broken line with the same curve parametrized backwards and different labels, swapping the labels on the two infinite domains.
Given such a curve, we can cut the infinite domain labeled $x^{-p}y^{k\delta}$ so that by Lemma~\ref{dom of def}, the action of $\bigl(\eta_{12\cdots n}^{B^T}\bigr)^{\pm m}$ near the endpoint is translation in the direction of $\nu_c(\delta)$.
We thus obtain a periodic broken line~$\bl$.
The endpoint of $\bl$ and the infinite domain labeled $x^p$ are on opposite sides of $\delta^\perp$.
(Otherwise, by the $h\gg0$ or $h\ll0$ case of what we already proved, the monomial associated to~$\bl$ is $x^p$, not $x^{-p}y^{k\delta}$.)
The construction may give a broken line $\bl$ with endpoint in $\d_\infty^-$ and $\bl((-\infty,t])\subseteq\d_\infty^+$, whereas we want the endpoint to be in~$\d_\infty^+$ and $\bl((-\infty,t])\subseteq\d_\infty^-$.
In that case, we can obtain the desired~$\bl$ by going back to the bi-infinite broken line and reversing~it.

As a preliminary step to constructing the bi-infinite broken line, we forget about the curve, but instead only show that there is a sequence of bends that gives the right monomials.
We bend at each of the coordinate walls in the order given by $c$.
If we start with monomial $x^{k\nu_c(\delta)}$, then $\br{k\nu_c(\delta),\alpha_1\ck}=-\br{\rho\ck_1,k\delta}$ by Proposition~\ref{pairing with simple}.
That means that we can bend on $\alpha_1^\perp$ and pick up a factor of $\hy_1^{\br{\rho\ck_1,k\delta}}$ so that the new monomial is $y_1^{\br{\rho\ck_1,k\delta}}\cdot x^{k\nu_c(\delta)+\omega_c(\,\cdot\,,\br{\rho\ck_1,k\delta}\alpha_1)}$.
By Proposition~\ref{pairing with simple}, we can then bend on $\alpha_2^\perp$ and pick up a factor of $\hy_2^{\br{\rho\ck_2,k\delta}}$, so now we have $y_1^{\br{\rho\ck_1,k\delta}}\cdot y_2^{\br{\rho\ck_2,k\delta}}\cdot x^{k\nu_c(\delta)+\omega_c(\,\cdot\,,\br{\rho\ck_1,k\delta}\alpha_1+\br{\rho\ck_2,k\delta}\alpha_2)}$.
Continuing in this manner, after we bend on $\alpha_n^\perp$, we have $y^{k\delta}\cdot x^{k\nu_c(\delta)+\omega_c(\,\cdot\,,k\delta)}$.
Since $\nu_c(\delta)=-\frac12\omega_c(\,\cdot\,,\delta)$, that's $y^{k\delta}\cdot x^{-k\nu_c(\delta)}$.

Now, to make the bi-infinite broken line, start with any straight line labeled with monomial $x^{k\nu_c(\delta)}$, coming from the direction of $\nu_c(\delta)$.
The considerations of the previous paragraph shows in particular that, at each step, after bending at~$\alpha_k^\perp$ (or similarly, before bending at all), the broken line is not parallel to the next hyperplane where it must bend.
But it might be moving away from the hyperplane where it must bend.
We can fix that be translating.
Specifically, if, after bending at $\alpha_{k-1}^\perp$, the broken line is headed away from $\alpha_k^\perp$, we translate the whole broken line along the span of $\rho_k$ to fix that.
Since $\rho_k$ is contained in $\alpha_i^\perp$ for all $i\neq k$, this translation does not change any of the earlier intersections with hyperplanes (except to translate the locations of the intersections within the hyperplanes).
In the end, we have a broken line that starts from infinity from the direction of $\nu_c(\delta)$ and ends by going out to infinity towards $\nu_c(\delta)$.
(By adjusting the translations, we can avoid intersections forbidden by the definition of a broken line and ensure that neither infinite domain of linearity is in $\delta^\perp$.)
\end{proof}

\section{Proofs}\label{proofs sec}
We now apply the tools of Section~\ref{tools sec} to prove Theorems~\ref{thet xi}, \ref{imag general}, \ref{expansion prod}, \ref{im exch} \ref{subalgebra}, \ref{tube subalgebra}, \ref{gen clus alg}, and \ref{gen clus alg almost}.
(Recall that Theorem~\ref{delta cheby} is a special case of Theorem~\ref{imag general}.)

For many of the theorems, the proof amounts to expanding a product $\thet_{p_1}\cdot\thet_{p_2}$ of theta functions with $p_1,p_2\in P\cap\d_\infty$ as a finite $\k[y]$-linear combination of theta functions.
In all of those proofs, even though the theorem has no hypotheses on $\tB$, we use, without comment, tools from Sections~\ref{per sec}--\ref{mut imag sec} that require $\tB$ to have signed-nondegenerating coefficients.
Thus we are tacitly first proving the theorems in the case of signed-nondegenerating coefficients and then implicitly using Proposition~\ref{struct plus plus} and the fact that every $B$ has an extension with signed-nondegenerating coefficients to conclude that the result holds for arbitrary extensions $\tB$ of~$B$.

In each such proof, we have access to certain facts:
Theorem~\ref{expand d inf} says that $\thet_{p_1}\cdot\thet_{p_2}$ expands as a finite $\k[y]$-linear combination of theta functions indexed by vectors $\lambda\in P\cap\d_\infty$.  
In that connection, it is important to recall from Section~\ref{aff back sec} that~$\d_\infty$ is the nonnegative linear span of the vectors $\sett{\nu_c(\beta):\beta\in\SimplesT{c}}$.
In some proofs, Theorem~\ref{B cone prod aff} applies to give even stronger conditions on~$\lambda$.  

Given $\lambda\in P\cap\d_\infty$, we will always assume a chi sequence for $\lambda$.
Theorem~\ref{mut pair affine periodic} says that for large enough $h$ or small enough $h$, the structure constant $a(p_1,p_2,\lambda)$ is the sum of the contributions from periodic pairs $(\bl_1,\bl_2)$ of broken lines for $p_1$ and $p_2$ respectively, with $\lambda_{\bl_1}+\lambda_{\bl_2}=\lambda$, both having endpoint $\chi_h$.
Furthermore, for each such pair, Lemma~\ref{per finite} says that $\lambda_{\bl_1},\lambda_{\bl_2}\in\delta^\perp$ and Lemma~\ref{no dinf} says that~$\bl_1$ and $\bl_2$ do not pass through $\d_\infty$.

We also use Theorem~\ref{simplification for boundary dinf} and/or~\ref{simplification for nuc delta}.
When Theorem~\ref{simplification for boundary dinf} applies, a periodic broken line only bends on walls normal to roots $\phi\in\APTre{c}$. 
Such a wall has scattering term $1+\hy^\phi$.
If the broken line bends on a wall normal to $\phi\in\APTre{c}$, it picks up a monomial of the form $\const x^{k\lambda}y^{k\phi}$ where $\lambda=\omega_c(\,\cdot\,,\phi)\in V^*$.

Recall the notation $\beta_{[i,j]}$ from Section~\ref{agaf sec} and write $\kappa_{[i]}$ for $\nu_c(\beta_{[i]})$ and $\kappa_{[i,j]}$ for $\nu_c(\beta_{[i,j]})$.
We will use Proposition~\ref{nuc phip phi} repeatedly to determine whether a broken line can bend on a wall normal to $\beta_{[i,j]}$.
To determine how monomials change when a broken line bends, we will use the following proposition, which is an immediate consequence of Proposition~\ref{crucial lemma}.

\begin{proposition}\label{crucial prop}
$\displaystyle\hy^{\beta_{[i,j]}}=y^{\beta_{[i,j]}}x^{-\kappa_{[i-1,j-1]}-\kappa_{[i,j]}}$.
\end{proposition}

Since the linear span of $\SimplesT{c}$ is $\delta^\perp$ and in light of Proposition~\ref{E delta in tubes}, we have the following immediate corollary of Proposition~\ref{crucial prop}.

\begin{corollary}\label{crucial cor} 
Suppose a broken line bends on a wall whose positive normal vector is in the nonnegative span of a $c$-orbit $\SimplesT{c}_o$ and suppose the monomial in the $x$ before the bend is $x^\lambda$ with $\lambda\in\delta^\perp$.
Then the bend contributes $x^\kappa$ to the monomial in the $x$, where $\kappa\in\delta^\perp$ is a \emph{nonpositive} linear combination of the vectors $\sett{\nu_c(\beta):\beta\in\SimplesT{c}_o}$.
\end{corollary}
Since $\d_\infty$ is the \emph{nonnegative} linear span of the vectors $\sett{\nu_c(\beta):\beta\in\SimplesT{c}}$, bends as in Corollary~\ref{crucial cor} tend to move the monomials outside of $\d_\infty$ (but inside~$\delta^\perp$).

\subsection{Proof of Theorem~\ref{thet xi}}
Suppose $\beta\in\SimplesT{c}$.
We need to show that 
\[\thet_{\nu_c(\beta)}\cdot\thet_{\nu_c(\delta-\beta)}=\thet_{\nu_c(\delta)}+y^\beta\thet_{\nu_c(\delta-\beta-c^{-1}\beta)}+y^{c\beta}\thet_{\nu_c(\delta-\beta-c\beta)}.\]
Since $\beta,\delta-\beta\in\APT{c}$, also $\nu_c(\beta),\nu_c(\delta-\beta)\in\d_\infty$.
By Theorem~\ref{expand d inf}, we need only consider structure constants $a(\nu_c(\beta),\nu_c(\delta-\beta),\lambda)$ for $\lambda\in P\cap\d_\infty$.

Given a chi sequence for some $\lambda\in P\cap\d_\infty$ and given $h\gg0$, Theorem~\ref{mut pair affine periodic} says that $a(\nu_c(\beta),\nu_c(\delta-\beta),\lambda)=\ap_{\chi_h}(\nu_c(\beta),\nu_c(\delta-\beta),\lambda)$.
For $\lambda\in P\cap\d_\infty$, suppose $(\bl_\beta,\bl_{\delta-\beta})$ is a periodic pair of broken lines for $\nu_c(\beta)$ and $\nu_c(\delta-\beta)$ respectively, with $\lambda_{\bl_\beta}+\lambda_{\bl_{\delta-\beta}}=\lambda$, both having endpoint~$\chi_h$.
Since $\nu_c(\beta)$ and $\nu_c(\delta-\beta)$ both span rays of the mutation fan contained in $\d_\infty$ and since $\d_\infty$ is the union of the cones of the star of the ray spanned by $\nu_c(\delta)$, the vectors $\nu_c(\beta)$ and $\nu_c(\delta-\beta)$ are in the relative boundary of~$\d_\infty$, so Theorem~\ref{simplification for boundary dinf} applies to both $\bl_\beta$ and $\bl_{\delta-\beta}$.
Thus since $h\gg0$, both $\bl_\beta$ and $\bl_{\delta-\beta}$ are completely contained in $\d_\infty^+$ and each bends only on walls whose normal vectors are in finite $c$-orbits. 

We write $\lambda^{(0)}_{\bl_\beta}$ for the exponent vector on the variables $x$ in the monomial labeling the unbounded domain of linearity of $\bl_\beta$, then $\lambda^{(1)}_{\bl_\beta}$ for the corresponding vector for the next domain of linearity, etc.
If $\lambda$ bends $a$ times, then $\lambda_{\bl_\beta}=\lambda^{(a)}_{\bl_\beta}$.
Similarly, we write $\lambda^{(0)}_{\bl_{\delta-\beta}}$, $\lambda^{(1)}_{\bl_{\delta-\beta}}$, and so forth.

Suppose $\beta$ is in a $c$-orbit $\SimplesT{c}_o$ of size $k$ and choose $\beta_{[0]}=\beta$ in this orbit, so that $\delta-\beta=\beta_{[1,k-1]}$.
Since $\delta=\beta_{[1,k]}$, also $\nu_c(\delta)=\kappa_{[1,k]}$.
Thus we want to show that 
\[\thet_{\kappa_{[0]}}\cdot\thet_{\kappa_{[1,k-1]}}=\thet_{\kappa_{[1,k]}}+y^{\beta_{[0]}}\thet_{\kappa_{[1,k-2]}}+y^{\beta_{[1]}}\thet_{\kappa_{[2,k-1]}}.\]
We have $\lambda^{(0)}_{\bl_\beta}=\kappa_{[0]}$ and $\lambda^{(0)}_{\bl_{\delta-\beta}}=\kappa_{[1,k-1]}$, so that $\lambda^{(0)}_{\bl_\beta}+\lambda^{(0)}_{\bl_{\delta-\beta}}=\kappa_{[1,k]}$.
Writing $\lambda_{\bl_\beta}+\lambda_{\bl_{\delta-\beta}}$ as a linear combination with nonnegative coefficients of the vectors $\nu_c(\SimplesT{c}_o)$, which are linearly independent, Corollary~\ref{crucial cor} says that any bend only decrease the coefficients of this combination.
No entry can be decreased by more than one because $\lambda_{\bl_\beta}+\lambda_{\bl_{\delta-\beta}}=\lambda\in\d_\infty$.
In particular, we see from Proposition~\ref{crucial prop} that the broken lines can only bend at walls with normal vectors~$\beta_{[j]}$.

By Proposition~\ref{nuc phip phi}, the first bend of $\bl_\beta$ can be on the wall normal to $\beta_{[j]}$ if and only if $j=0$ or $j=1$, because otherwise, the infinite domain of $\bl_\beta$ is parallel to the wall.
Suppose $\bl_\beta$ first bends on the wall normal to $\beta_{[0]}$.
Then $\lambda^{(1)}_{\bl_\beta}=-\kappa_{[-1]}$.
Arguing as above, we see that a subsequent bend must happen on the wall normal to $\beta_{[-1]}$ or $\beta_{[0]}$.
However, in either case, the coefficient of $\kappa_{[-1]}$ is decreased for a second time, and thus this bend is not allowed.
On the other hand, suppose $\bl_\beta$ first bends on the wall normal to $\beta_{[1]}$, so that $\lambda^{(1)}_{\bl_\beta}=-\kappa_{[1]}$. 
Any subsequent bend would have to be in the wall normal to $\beta_{[1]}$ or~$\beta_{[2]}$, but again in either case, the coefficient of $\kappa_{[1]}$ would be decreased for a second time, so only one bend can happen.
We see that there are three possible broken lines $\bl_\beta$: 
One that doesn't bend and has monomial $x^{\kappa_{[0]}}$, one that bends once and has monomial $y^{\beta_{[0]}}x^{-\kappa_{[-1]}}$, and one that bends once and has monomial $y^{\beta_{[1]}}x^{-\kappa_{[1]}}$.

Proposition~\ref{E delta in tubes} implies that $\bl_{\delta-\beta}$ can bend on the same walls as $\bl_\beta$, namely the walls with normal vectors $\beta_{[0]}$ or $\beta_{[1]}$.
If $\bl_{\delta-\beta}$ first bends on the wall normal to $\beta_{[0]}$, then $\lambda^{(1)}_{\delta-\beta}=\kappa_{[1,k-1]}-\kappa_{[-1,0]}=\kappa_{[1,k-2]}-\kappa_{[0]}$.
If $\bl_{\delta-\beta}$ first bends on the wall normal to $\beta_{[1]}$, then $\lambda^{(1)}_{\delta-\beta}=\kappa_{[1,k-1]}-\kappa_{[0,1]}=\kappa_{[2,k-1]}-\kappa_{[0]}$.
Now Proposition~\ref{E delta in tubes} implies that the possibilities for further bending are just as in the case of $\bl_\beta$, and we see that no subsequent bending is possible, for the same reasons.
There are again three possible broken lines $\bl_{\delta-\beta}$:
One that doesn't bend and has monomial $x^{\kappa_{[1,k-1]}}$, one that bends once and has monomial $y^{\beta_{[0]}}x^{\kappa_{[1,k-2]}-\kappa_{[0]}}$, and one that bends once and has monomial $y^{\beta_{[1]}}x^{\kappa_{[2,k-1]}-\kappa_{[0]}}$.

Each bend of either broken line decreases the the coefficient of $\kappa_{[0]}$.
Thus either $\bl_\beta$ or $\bl_{\delta-\beta}$ (or neither) bends on a wall normal to $\beta_{[0]}$ or $\beta_{[1]}$, only one of them can bend, and only once.
If there is no bend, then $\lambda$ is $\nu_c(\delta)$ and the monomial contributed by the pair is $1$.

If there is a bend on $\beta_{[0]}$, then $\lambda$ is $\kappa_{[1,k-2]}$ and the contribution is $y^{\beta_{[0]}}$.
There are two possibilities:  Either $\bl_\beta$ bends or $\bl_{\delta-\beta}$ bends.
We will show that, for a particular choice of chi sequence for $\kappa_{[1,k-2]}$, exactly one of these two possibilities occurs.
The choice that determines a chi sequence is the choice of a maximal cone~$C$ of~$\F_{B^T}$ with $\kappa_{[1,k-2]}\in C\subseteq\d_\infty$.
Since $\br{\kappa_{[1,k-2]},\beta_{[0]}\ck}=\br{\nu_c(\beta_{[1,k-2]}),\beta_{[0]}\ck}=0$ and $\br{\kappa_{[1,k-1]},\beta_{[0]}\ck}=\br{\nu_c(\delta-\beta),\beta_{[0]}\ck}=1$, we can choose $C$ to be on the same side $\beta_{[0]}^\perp$ as $\nu_c(\delta-\beta)$.
Thus $\chi_h$ is also on the same side of $\beta_{[0]}^\perp$ as $\nu_c(\delta-\beta)$ for $h\gg0$.
We see that $\bl_{\delta-\beta}$ can't bend on the wall normal to $\beta_{[0]}$ and still reach~$\chi_h$.
On the other hand, $\br{\kappa_{[0]},\beta_{[0]}\ck}=\br{\nu_c(\beta),\beta_{[0]}\ck}=-1$, so $\nu_c(\beta)$ is on the opposite side of $\beta_{[0]}^\perp$ from $\chi_h$, so $\bl_{\beta}$ can bend and still reach $\chi_h$.

If there is a bend on $\beta_{[1]}$, then $\lambda$ is $\kappa_{[2,k-1]}$ and the contribution is $y^{\beta_{[1]}}$.
Since $\br{\kappa_{[2,k-1]},\beta_{[1]}\ck}=\br{\nu_c(\beta_{[2,k-1]}),\beta_{[1]}\ck}=0$ and $\br{\kappa_{[1,k-1]},\beta_{[1]}\ck}=\br{\nu_c(\delta-\beta),\beta_{[1]}\ck}=-1$, we can choose $C$ so that $\chi_h$ is on the same side of $\beta_{[1]}^\perp$ as $\nu_c(\delta-\beta)$ for $h\gg0$.
Once again, $\bl_{\delta-\beta}$ can't bend on the wall normal to $\beta_{[1]}$, but since $\br{\kappa_{[0]},\beta_{[1]}\ck}=\br{\nu_c(\beta),\beta_{[1]}\ck}=1$, $\bl_{\beta}$ can bend on that wall.
\qed

\subsection{Proof of Theorems~\ref{delta cheby} and~\ref{imag general}}
Theorem~\ref{delta cheby} is a special case of Theorem~\ref{imag general}, although stated differently.
We already proved a small part of Theorem~\ref{imag general} as part of the proof of Theorem~\ref{simplification for nuc delta}, but we give the whole proof here.
Theorem~\ref{B cone prod aff} applies in this situation but is not needed because Theorem~\ref{simplification for nuc delta} is so specific.

We begin with the product $(\thet_{k\nu_c(\delta)})^2$.
By Theorem~\ref{expand d inf}, we need only consider structure constants $a(k\nu_c(\delta),k\nu_c(\delta),\lambda)$ for $\lambda\in P\cap\d_\infty$.
Given a chi sequence for some $\lambda\in P\cap\d_\infty$ and given $h\gg0$, Theorem~\ref{mut pair affine periodic} says that $a(k\nu_c(\delta),k\nu_c(\delta),\lambda)=\ap_{\chi_h}(k\nu_c(\delta),k\nu_c(\delta),\lambda)$.
Thus we want to find all periodic pairs $(\bl_1,\bl_2)$ of broken lines for $k\nu_c(\delta)$ and $k\nu_c(\delta)$ respectively, with $\lambda_{\bl_1}+\lambda_{\bl_2}=\lambda$, both having endpoint $\chi_h$.
Theorem~\ref{simplification for nuc delta} says that there are exactly two periodic broken lines for $k\nu_c(\delta)$ with endpoint $\chi_h$, one with monomial~$x^{k\nu_c(\delta)}$ and the other with monomial $x^{-k\nu_c(\delta)}y^{k\delta}$.
Because $\lambda_{\bl_1}+\lambda_{\bl_2}=\lambda\in\d_\infty$, we rule out the pair for which both monomials are $x^{-k\nu_c(\delta)}y^{k\delta}$.
We see that $(\thet_{k\nu_c(\delta)})^2=\thet_{2k\nu_c(\delta)}+2y^{k\delta}$, as desired.

The statement for $k>\ell\ge1$ is proved similarly.
In this case, there are two periodic broken lines for $k\nu_c(\delta)$ with endpoint $\chi_h$, one with monomial~$x^{k\nu_c(\delta)}$ and the other with monomial $x^{-k\nu_c(\delta)}y^{k\delta}$.
There are also two periodic broken lines for $\ell\nu_c(\delta)$ with endpoint $\chi_h$, one with monomial~$x^{\ell\nu_c(\delta)}$ and the other with monomial $x^{-\ell\nu_c(\delta)}y^{k\delta}$.
Because $\lambda_{\bl_1}+\lambda_{\bl_2}=\lambda\in\d_\infty$, we must take $\bl_1$ to be the broken line for $k\nu_c(\delta)$ with monomial $x^{k\nu_c(\delta)}$, and then $\bl_2$ can be either broken line for $\ell\nu_c(\delta)$.
Thus $\thet_{k\nu_c(\delta)}\cdot\thet_{\ell\nu_c(\delta)}=\thet_{(k+\ell)\nu_c(\delta)}+y^{\ell\delta}\thet_{(k-\ell)\nu_c(\delta)}$, as desired.
\qed

\begin{remark}\label{scatcomb ref}
We thank an anonymous referee to the paper \cite{scatcomb} for pointing out the proof of Theorem~\ref{delta cheby} in rank $2$ and suggesting that a similar proof should work for acyclic exchange matrices of affine type in general.
The basic idea of the proof does indeed generalize beyond rank $2$, and this is the proof given above, but the proof is significantly more complicated in rank $\ge3$.
The proof is simplified to a manageable level of complexity by the results of \cite{affscat,canonical} that we have quoted.
\end{remark}

\subsection{Proof of Theorem~\ref{expansion prod}}
As explained after the statement of Theorem~\ref{expansion prod}, Theorem~\ref{clus mon thm} reduces the proof of Theorem~\ref{expansion prod} to proving a fact about products of certain pairs of theta functions for vectors in the imaginary wall.
We now state and prove that fact to complete the proof of  Theorem~\ref{expansion prod}.

\begin{theorem}\label{boundary with delta}
Suppose $p$ is a vector in the boundary of the imaginary wall.
Then $\thet_p\cdot\thet_{k\nu_c(\delta)}=\thet_{p+k\nu_c(\delta)}$ for any $k\ge0$.
\end{theorem}
\begin{proof}
Theorem~\ref{expand d inf} says that structure constants $a(p,k\nu_c(\delta),\lambda)$ are zero unless ${\lambda\in P\cap\d_\infty}$.
But since $p$ is in the boundary of $\d_\infty$, which is the star of the ray spanned by $\nu_c(\delta)$, there is an imaginary cone $C$ containing $p$ and $\nu_c(\delta)$, so Theorem~\ref{B cone prod aff} applies to say something much stronger:
The structure constants $a(p,k\nu_c(\delta),\lambda)$ are zero unless $\lambda$ is of the form $p+(k-2a)\nu_c(\delta)$ for $0\le a\in\integers$.

More specifically, choose the cone $C$ containing $p$ and $\nu_c(\delta)$ to be maximal in~$\F_{B^T}$.
The cone $C$ is defined, as a subset of $\d_\infty$, by inequalities of the form $\br{x,\phi}\le0$ for $\phi\in\pm\APTre{c}$.
Let~$\Gamma$ be the set of all roots $\phi\in\pm\APTre{c}$ such that $\br{x,\phi}\le0$ holds for $x\in C$.
In particular, $\br{p,\phi}\le0$ for all $\phi\in\Gamma$.

Suppose $\lambda\in P\cap\d_\infty$ is of the form $p+(k-2a)\nu_c(\delta)$.
In particular, since $\br{\nu_c(\delta),\phi}=0$ for all $\phi\in\APTre{c}$ by Proposition~\ref{E delta in tubes}, $\lambda$ is in $C$.
Choose a chi sequence for $\lambda$.
Since the chi sequence limits to a point in the relative interior of $C$ and since~$C$ is maximal in $\F_{B^T}$, for $h\gg0$ we have $\br{\chi_h,\phi}<0$ for all $\phi\in\Gamma$.

Theorem~\ref{mut pair affine periodic} says that $a(p,k\nu_c(\delta),\lambda)=\ap_{\chi_h}(p,k\nu_c(\delta),\lambda)$.
By Theorem~\ref{simplification for nuc delta}, there are exactly two periodic broken lines for $k\nu_c(\delta)$ with endpoint $\chi_h$, one with monomial~$x^{k\nu_c(\delta)}$ and the other with monomial $x^{-k\nu_c(\delta)}y^{k\delta}$.
Theorem~\ref{simplification for boundary dinf} says that any periodic broken line for~$p$ with endpoint $\chi_h$ is contained in $\d_\infty^+$ and only bends on walls whose normal vectors are in~$\APTre{c}$.

Suppose $(\bl_1,\bl_2)$ is a periodic pair of broken lines for $p$ and $k\nu_c(\delta)$ respectively, with $\lambda_{\bl_1}+\lambda_{\bl_2}=\lambda$, both having endpoint $\chi_h$.
Then since $\lambda_{\bl_1}=\lambda-\lambda_{\bl_2}$, since $\lambda=p+(k-2a)\nu_c(\delta)$, and since $\lambda_{\bl_2}=\pm k\nu_c(\delta)$, we have $\lambda_{\bl_1}=p+\ell\nu_c(\delta)$, for an integer~$\ell$.

Consider the domain $L$ of linearity of $\bl_1$ that contains the endpoint $\chi_h$.
Following~$\bl_1$ backwards from $\chi_h$, we leave $\chi_h$ in the direction $\lambda_{\bl_1}=p+\ell\nu_c(\delta)$.
But $\br{\chi_h,\phi}<0$ for all $\phi\in\Gamma$ and $\br{p,\phi}\le0$ for all $\phi\in\Gamma$ and (by Proposition~\ref{E delta in tubes}), $\br{\nu_c(\delta),\phi}=0$ for all $\phi\in\APTre{c}$.
Thus $\br{x,\phi}<0$ holds for all $x\in L$ and $\phi\in\Gamma$.
We see that $L$ is contained in the cone ${\sett{x\in V^*:\br{x,\phi}<0\,\forall\phi\in\Gamma}}$, so Lemma~\ref{hopefully} implies that $L$ does not intersect any walls $\d_\beta$ with $\beta\in\APTre{c}$.
But as mentioned above, $\bl_1$ only bends on walls with normals in $\APTre{c}$, so we see that $\bl_1$ never bends.
Therefore the monomial on $\bl_1$ is $x^p$, and since $\lambda=\lambda_{\bl_1}+\lambda_{\bl_2}\in\d_\infty$, we see that the monomial on $\bl_2$ is $x^{k\nu_c(\delta)}$ rather than $x^{-k\nu_c(\delta)}y^{k\delta}$.

We have shown that there is only one possibility for $\lambda$, namely $\lambda=p+k\nu_c(\delta)$, and furthermore that $a(p,k\nu_c(\delta),p+k\nu_c(\delta))=\ap_{\chi_h}(p,k\nu_c(\delta),p+k\nu_c(\delta))=1$.
\end{proof}

\subsection{Proof of Theorem~\ref{im exch}}
Suppose $\beta_{[0]}$ and $\beta_{[\ell]}$ are in a $c$\nobreakdash-orbit~$\SimplesT{c}_o$ of size~$k$, with $0<\ell<k$.
Then $\nu_c(\delta-\beta_{[0]})=\kappa_{[1,k-1]}$ and $\nu_c(\delta-\beta_{[\ell]})=\kappa_{[\ell+1,k+\ell-1]}$, and we need to show that 
\begin{multline*}
\thet_{\kappa_{[1,k-1]}}\cdot\thet_{\kappa_{[\ell+1,k+\ell-1]}}\\
=\thet_{\kappa_{[1,k-1]}+\kappa_{[\ell+1,k+\ell-1]}}+y^{\beta_{[\ell+1,k]}}\thet_{2\kappa_{[1,\ell-1]}}+y^{\beta_{[1,\ell]}}\thet_{2\kappa_{[\ell+1,k-1]}}.
\end{multline*}
By Theorem~\ref{expand d inf}, we only consider structure constants $a(\kappa_{[1,k-1]},\kappa_{[\ell+1,k+\ell-1]},\lambda)$ for $\lambda\in P\cap\d_\infty$.
Given a chi sequence for some $\lambda\in P\cap\d_\infty$, Theorem~\ref{mut pair affine periodic} says that the structure constants are $\ap_{\chi_h}(\kappa_{[1,k-1]},\kappa_{[\ell+1,k+\ell-1]},\lambda)$ for $h\gg0$.

For $\lambda\in P\cap\d_\infty$ and $h\gg0$, suppose $(\bl_0,\bl_\ell)$ is a pair of periodic broken lines for $\kappa_{[1,k-1]}$ and $\kappa_{[\ell+1,k+\ell-1]}$ respectively, both having endpoint $\chi_h$, with $\lambda_{\bl_0}+\lambda_{\bl_\ell}=\lambda$.
As in the proof of Theorem~\ref{thet xi}, Theorem~\ref{simplification for boundary dinf} applies to say that both $\bl_0$ and~$\bl_\ell$ are contained in $\d_\infty^+$ and only bend on walls whose normal vectors are in~$\APTre{c}$.

Our proof follows the same general outline as the proof of Theorem~\ref{thet xi}, but is more complicated.
We describe it now in five steps.
First, in Cases~1 and~2 below, we consider all sequences of bends in $\bl_0$ that can conceivably happen.
Specifically, we rule out bends that would cause $\lambda_{\bl_0}+\lambda_{\bl_\ell}$ to be a linear combination of the vectors $\nu_c(\SimplesT{c}_o)$ with some coefficients nonnegative, contradicting the fact that $\lambda\in\d_\infty$.
Corollary~\ref{crucial cor} implies that subsequent bends of $\bl_0$ or bends of $\bl_\ell$ do not resolve the contradiction.
Throughout, we use the symbol ``$\le$'' to denote componentwise comparison of coordinates in the linearly independent set $\nu_c(\SimplesT{c}_o)$.
Second and symmetrically, in Cases A and B bellow, we consider all sequences of bends in $\bl_\ell$ that can conceivably happen.
Since additional bends do not ``rescue'' the broken lines that we have ruled out in these cases, we only need to consider pairs $(\bl_0,\bl_\ell)$ such that $\bl_0$ is a broken line not ruled out in Case 1 or 2 and $\bl_\ell$ is a broken line not ruled out in Case A or B.
Third, we begin to determine which pairs can occur, making those determinations that don't depend on a specific choice of chi sequence.  
The possibilities are combinations of the earlier cases, so we call them ``scenarios'' rather than ``cases'' to avoid confusion.  
We rule out many scenarios because of the requirement that $\lambda$ is in $\d_\infty$.
We also narrow many scenarios, essentially because the bends of one of the broken lines mean that the broken line ends on one side of some hyperplane while $\lambda$ is on the other. 
Fourth, we adopt a specific choice of chi sequence to eliminate most scenarios, and fifth, we eliminate all but one pair of broken lines in each remaining scenario.
 
We write $\lambda^{(0)}_{\bl_0}$ for the exponent vector on $x$ in the monomial labeling the unbounded domain of linearity of $\bl_0$, then $\lambda^{(1)}_{\bl_0}$ for the corresponding vector for the next domain of linearity, etc.
Similarly, we write $\lambda^{(0)}_{\bl_\ell}$, $\lambda^{(1)}_{\bl_\ell}$, and so forth.
We have
\[\lambda^{(0)}_{\bl_0}+\lambda^{(0)}_{\bl_\ell}={\kappa_{[1,k-1]}+\kappa_{[\ell+1,k+\ell-1]}}=2\kappa_{[1,k]}-\kappa_{[0]}-\kappa_{[\ell]}.\]
Thus the total contribution of all bends can't decrease the $\kappa_{[0]}$- or $\kappa_{[\ell]}$-coordinate by more than~$1$ and can't decrease any other coordinate by more than~$2$.

We now begin the first step, determining the conceivable sequences of bends of~$\bl_0$.
Proposition~\ref{nuc phip phi} implies that ${\br{\kappa_{[1,k-1]},\beta_{[i,j]}\ck}=0}$ unless either $i=1$, in which case $\br{\kappa_{[1,k-1]},\beta_{[i,j]}\ck}=-1$ or $j=k$, in which case $\br{\kappa_{[1,k-1]},\beta_{[i,j]}\ck}=1$.
Thus there are two cases for the first bend of $\bl_0$.

\smallskip

\noindent
\textbf{Case 1.}
\textit{The first bend of $\bl_0$ is on a wall normal to $\beta_{[1,j]}$.}
Since $\lambda_{\bl_\ell}\le\kappa_{[\ell+1,k+\ell-1]}$, Proposition~\ref{crucial prop} implies that 
\[\lambda^{(1)}_{\bl_0}+\lambda_{\bl_\ell}\le2\kappa_{[1,k]}-\kappa_{[0]}-\kappa_{[\ell]}-\kappa_{[0,j-1]}-\kappa_{[1,j]}=\kappa_{[j,k-1]}+\kappa_{[j+1,k-1]}-\kappa_{[\ell]}.\]
Since $\lambda^{(1)}_{\bl_0}+\lambda_{\bl_\ell}$ must be in $\d_\infty$, we have $j\le\ell$.
Thus $\lambda^{(1)}_{\bl_0}=\kappa_{[j+1,k-1]}-\kappa_{[0,j-1]}$.
Because $\br{\kappa_{[1,k-1]},\beta_{[1,j]}\ck}=-1$, after this first bend, $\bl_0$ has passed to the positive side of $\beta_{[1,j]}^\perp$.

Now consider a second bend of $\bl_0$, on a wall normal to $\beta_{[i',j']}$.
Again by Proposition~\ref{crucial prop}, 
\[\lambda^{(2)}_{\bl_0}+\lambda_{\bl_\ell}\le\kappa_{[j,k-1]}+\kappa_{[j+1,k-1]}-\kappa_{[\ell]}-\kappa_{[i'-1,j'-1]}-\kappa_{[i',j']}.\]
Since $\lambda^{(2)}_{\bl_0}+\lambda_{\bl_\ell}$ must be in~$\d_\infty$, either $\ell<i'\le j'<k$ or $j<i'\le j'\le\ell$.
But if $\ell<i'\le j'<k$, then $\br{\lambda^{(1)}_{\bl_0},\beta_{[i',j']}}=0$ by Proposition~\ref{nuc phip phi}, and if $j<i'\le j'\le\ell$, then by the same proposition, ${\br{\lambda^{(1)}_{\bl_0},\beta_{[i',j']}}=0}$ unless $i'=j+1$, in which case $\br{\lambda^{(1)}_{\bl_0},\beta_{[i',j']}}=-1$.
Thus the second bend, if any, is on a wall normal to $\beta_{[j+1,j']}$ for some $j<j'\le\ell$.
By Proposition~\ref{crucial prop}, 
\[\lambda^{(2)}_{\bl_0}=\kappa_{[j+1,k-1]}-\kappa_{[0,j-1]}-\kappa_{[j,j'-1]}-\kappa_{[j+1,j']}=\kappa_{[j'+1,k-1]}-\kappa_{[0,j'-1]}.\]
After this bend, $\bl_0$ remains on the positive side of $\beta_{[1,j]}^\perp$ and has passed to the positive side of $\beta_{[j+1,j']}^\perp$.
Since $\beta_{[1,j']}=\beta_{[1,j]}+\beta_{[j+1,j']}$, we see that $\bl_0$ has passed to the positive side of $\beta_{[1,j']}^\perp$.

The status of $\bl_0$ after the second bend is analogous to the status after the first bend, except that $j$ has been replaced by $j'$.
Further bends proceed in the same manner.
We conclude that, in any case, $\lambda_{\bl_0}=\kappa_{[j+1,k-1]}-\kappa_{[0,j-1]}$ for some $0<j\le\ell$ and that $\chi_h$ is on the positive side of $\beta_{[1,j]}^\perp$.
We also easily check that $\const_{\bl_0}=1$ and $\beta_{\bl_0}=\beta_{[1,j]}$.

\smallskip

\noindent
\textbf{Case 2.}
\textit{The first bend of $\bl_0$ is on a wall normal to $\beta_{[i,k]}$, and $\bl_\ell$.}
In this case,
\[\lambda^{(1)}_{\bl_0}+\lambda_{\bl_\ell}\le2\kappa_{[1,k]}-\kappa_{[0]}-\kappa_{[\ell]}-\kappa_{[i-1,k-1]}-\kappa_{[i,k]}=\kappa_{[1,i-1]}+\kappa_{[1,i-2]}-\kappa_{[\ell]},\] 
and since $\lambda^{(1)}_{\bl_0}+\lambda_{\bl_\ell}$ must be in $\d_\infty$, we have $i>\ell$.
Also, $\lambda^{(1)}_{\bl_0}=\kappa_{[1,i-2]}-\kappa_{[i,k]}$.
After this first bend, $\bl_0$ is on the negative side of $\beta^\perp_{[i,k]}$.

Consider a second bend, on a wall normal to $\beta_{[i',j']}$.
Proposition~\ref{crucial prop} implies that 
\[\lambda^{(2)}_{\bl_0}+\lambda_{\bl_\ell}\le2\kappa_{[1,k]}-\kappa_{[0]}-\kappa_{[\ell]}-\kappa_{[i-1,k-1]}-\kappa_{[i,k]}-\kappa_{[i'-1,j'-1]}-\kappa_{[i',j']}.\]
Because $\lambda^{(2)}_{\bl_0}+\lambda_{\bl_\ell}\in\d_\infty$, either $1<i'\le j'\le\ell$ or $\ell<i'\le j'<i$.
If $1<i'\le j'\le\ell$, then $\br{\lambda^{(1)}_{\bl_0},\beta_{[i',j']}}=0$ and if $\ell<i'\le j'<i$, then $\br{\lambda^{(1)}_{\bl_0},\beta_{[i',j']}}=0$ unless $j'=i-1$, in which case $\br{\lambda^{(1)}_{\bl_0},\beta_{[i',j']}}=1$.
Thus the second bend, if any, is on a wall normal to $\beta_{[i',i-1]}$ for some $\ell<i'<i$, and 
\[\lambda^{(2)}_{\bl_0}=\kappa_{[1,i-2]}-\kappa_{[i,k]}-\kappa_{[i'-1,i-2]}-\kappa_{[i',i-1]}=\kappa_{[1,i'-2]}-\kappa_{[i',k]}.\]
After this bend, $\bl_0$ remains on the negative side of $\beta^\perp_{[i,k]}$ and has passed to the negative side of $\beta_{[i',i-1]}$, and thus is on the negative side of $\beta^\perp_{[i',k]}$.
Further bends work in the same manner, so in any case $\lambda_{\bl_0}=\kappa_{[1,i-2]}-\kappa_{[i,k]}$ for some $i$ with $\ell<i\le k$, and~$\chi_h$ is on the negative side of $\beta^\perp_{[i,k]}$.
Also, $\const_{\bl_0}=1$ and $\beta_{\bl_0}=\beta_{[i,k]}$.

\smallskip

As our second step, we determine the conceivable sequences of bends of $\bl_\ell$, using Cases 1 and 2 and the symmetry of reversing the roles of $0$ and $\ell$.
There are two cases for the first bend of $\bl_\ell$.
Either $\bl_\ell$ bends on a wall normal to $\beta_{[\ell+1,j]}$ or $\bl_\ell$ bends on a wall normal to $\beta_{[i,\ell]}$.

\smallskip

\noindent
\textbf{Case A.}
\textit{The first bend of $\bl_\ell$ is on a wall normal to $\beta_{[\ell+1,j]}$.}
In this case, $\lambda_{\bl_\ell}=\kappa_{[j-k+1,\ell-1]}-\kappa_{[\ell,j-1]}$ for some $\ell<j\le k$ (possibly larger than the $j$ that identifies this case) and $\chi_h$ is on the positive side of $\beta_{[\ell+1,j]}^\perp$.
Also, $\const_{\bl_\ell}=1$ and $\beta_{\bl_\ell}=\beta_{[\ell+1,j]}$.

\smallskip

\noindent
\textbf{Case B.}
\textit{The first bend of $\bl_\ell$ is on a wall normal to $\beta_{[i,\ell]}$.}
In this case $\lambda_{\bl_\ell}=\kappa_{[\ell+1,i+k-2]}-\kappa_{[i,\ell]}$ for some $0<i\le\ell$, and $\chi_h$ is on the negative side of $\beta^\perp_{[i,\ell]}$.
Also, $\const_{\bl_\ell}=1$ and $\beta_{\bl_0}=\beta_{[i,\ell]}$.

Our third step is to determine which pairs $(\bl_0,\bl_\ell)$ can actually occur. 

\smallskip

\noindent
\textbf{Scenario 0+0.}
\textit{Neither $\bl_0$ nor $\bl_\ell$ bends.}
Then $\lambda_{\bl_0}+\lambda_{\bl_\ell}={\kappa_{[1,k-1]}+\kappa_{[\ell+1,k+\ell-1]}}$, there are no conditions on where $\chi_h$ is, and the monomial for this pair is~$1$.

\smallskip

\noindent
\textbf{Scenario 1+0.}
\textit{$\bl_0$ is as in Case 1 and $\bl_\ell$ does not bend.}
In this scenario, we have $\lambda=\lambda_{\bl_0}+\lambda_{\bl_\ell}=\kappa_{[j,k-1]}+\kappa_{[j+1,k-1]}-\kappa_{[\ell]}$ for some $0<j\le\ell$ and $\chi_h$ is on the positive side of $\beta^\perp_{[1,j]}$.

Suppose $j<\ell$.  
Then $\lambda=\kappa_{[j,k-1]}+\kappa_{[j+1,\ell-1]}+\kappa_{[\ell+1,k-1]}$.
The second and/or third term vanishes if $j=\ell-1$ and/or $\ell=k-1$.
Each of these terms that doesn't vanish spans a ray of $\F_{B^T}$, and any imaginary cone $C$ of $\F_{B^T}$ that contains $\lambda$ also contains these rays.
In particular, any imaginary cone $C$ containing $\lambda$ is spanned by $\kappa_{[j,k-1]}$ and by rays that are compatible with $\kappa_{[j,k-1]}$.   
By Proposition~\ref{nuc phip phi}, we compute that $\br{\kappa_{[j,k-1]]},\beta_{[1,j]}}=-1$.
Furthermore, we check that if $\kappa_{[i',j']}$ spans a ray compatible with $\kappa_{[j,k-1]}$, then $\br{\kappa_{[i',j']},\beta_{[1,j]}}\le0$.
Specifically, if $\beta_{[i',j']}$ is spaced with $\beta_{[j,k-1]}$, then $\br{\kappa_{[i',j']},\beta_{[1,j]}}=0$ and if $\beta_{[i',j']}$ is nested inside or outside of $\beta_{[j,k-1]}$, then $\br{\kappa_{[i',j']},\beta_{[1,j]}}\in\set{-1,0}$.

We see that if $j<\ell$, then every imaginary cone of $\F_{B^T}$ containing $\lambda$ has its relative interior strictly on the negative side of $\beta^\perp_{[1,j]}$.
But the chi sequence limits to a point in the relative interior of a cone of $\F_{B^T}$ containing $\lambda$, so we have reached a contradiction, because $h\gg0$, to the fact that $\chi_h$ is on the positive side of $\beta^\perp_{[1,j]}$. 
This contradiction shows that Scenario 1+0 can only occur with $j=\ell$ and thus $\lambda=\lambda_{\bl_0}+\lambda_{\bl_\ell}=2\kappa_{[\ell+1,k-1]}$, $\chi_h$ is on the positive side of $\beta^\perp_{[1,\ell]}$, and the monomial for this pair is~$y^{\beta_{[1,\ell]}}$.

\smallskip

\noindent
\textbf{Scenario 2+0.}
\textit{$\bl_0$ is as in Case 2 and $\bl_\ell$ does not bend.}
In this scenario, we have $\lambda=\lambda_{\bl_0}+\lambda_{\bl_\ell}=\kappa_{[1,i-2]}+\kappa_{[1,i-1]}-\kappa_{[\ell]}$ for some $\ell<i\le k$, and~$\chi_h$ is on the negative side of~$\beta^\perp_{[i,k]}$.
If $i>\ell+1$, then $\lambda=\kappa_{[1,i-1]}+\kappa_{[1,\ell-1]}+\kappa_{[\ell+1,i-2]}$, with one or more of the last two terms vanishing, namely if $\ell=1$ and/or $i=\ell+2$.
We compute $\br{\lambda,\beta_{[i,k]}}=1$ and check that if $\kappa_{[i',j']}$ spans a ray compatible with $\kappa_{[1,i-1]}$, then $\br{\kappa_{[i',j']},\beta_{[i,k]}}\ge0$.
Arguing as in Scenario 1+0, we conclude that every imaginary cone of $\F_{B^T}$ containing $\lambda$ has its relative interior strictly on the positive side of $\beta_{[i,k]}^\perp$, contradicting the fact that $\chi_h$ is on the negative side of~$\beta^\perp_{[i,k]}$.
By this contradiction, we conclude that Scenario 2+0 only occurs with $i=\ell+1$ and thus $\lambda=\lambda_{\bl_0}+\lambda_{\bl_\ell}=2\kappa_{[1,\ell-1]}$, $\chi_h$ is on the negative side of $\beta^\perp_{[\ell+1,k]}$, and the monomial for this pair is~$y^{\beta_{[\ell+1,k]}}$.

\smallskip

\noindent
\textbf{Scenario 0+A.}
\textit{$\bl_0$ does not bend and $\bl_\ell$ is as in Case A.}
In this scenario, $\lambda=\lambda_{\bl_0}+\lambda_{\bl_\ell}=\kappa_{[j+1,\ell+k-1]}+\kappa_{[j,\ell+k-1]}-\kappa_{[k]}$ for some $\ell<j\le k$ and $\chi_h$ is on the positive side of $\beta_{[\ell+1,k+j]}^\perp$.
This scenario is symmetric to Scenario 1+0 by reversing the roles of $0$ and $\ell$.
We see that $j=k$, so that $\lambda=\lambda_{\bl_0}+\lambda_{\bl_\ell}=2\kappa_{[1,\ell-1]}$, $\chi_h$ is on the positive side of $\beta_{[\ell+1,k]}^\perp$, and the monomial for this pair is~$y^{\beta_{[\ell+1,k]}}$.

\smallskip

\noindent
\textbf{Scenario 0+B.}
\textit{$\bl_0$ does not bend and $\bl_\ell$ is as in Case B.}
In this scenario, $\lambda=\lambda_{\bl_0}+\lambda_{\bl_\ell}=\kappa_{[\ell+1,i+k-2]}+\kappa_{[\ell+1,i+k-1]}-\kappa_{[0]}$ for some $0<i\le\ell$ and $\chi_h$ is on the negative side of $\beta^\perp_{[i,\ell]}$.
This scenario is symmetric to Scenario 2+0 by reversing the roles of $0$ and $\ell$.
We see that $i=1$, so that $\lambda=\lambda_{\bl_0}+\lambda_{\bl_\ell}=2\kappa_{[\ell+1,k-1]}$, $\chi_h$ is on the negative side of $\beta^\perp_{[1,\ell]}$, and the monomial for this pair is~$y^{\beta_{[1,\ell]}}$.

\smallskip

\noindent
\textbf{Scenario 1+A.}
\textit{$\bl_0$ is as in Case 1 and $\bl_\ell$ is as in Case A.}
In this scenario, $\lambda=\lambda_{\bl_0}=\kappa_{[j+1,k-1]}-\kappa_{[0,j-1]}$ for some $0<j\le\ell$ and $\lambda_{\bl_\ell}=\kappa_{[j'-k+1,\ell-1]}-\kappa_{[\ell,j'-1]}$ for some $\ell<j'\le k$.
The requirement that $\lambda\in\d_\infty$ implies that in fact $0<j<\ell$
and $\ell<j'<k$.
Furthermore, in this scenario, $\chi_h$ is on the positive side of $\beta_{[1,j]}^\perp$ and on the positive side of $\beta_{[\ell+1,j']}^\perp$, so that $\br{\chi_h,\beta_{[1,j]}+\beta_{[\ell+1,j']}}>0$.
We compute
\[\lambda_{\bl_0}+\lambda_{\bl_\ell}
=\kappa_{[j,\ell-1]}+\kappa_{[j+1,\ell-1]}+\kappa_{[j',k-1]}+\kappa_{[j'+1,k-1]}.\]
This is the sum of four vectors that span rays of an imaginary cone (except that $\kappa_{[j+1,\ell-1]}$ and/or $\kappa_{[j'+1,k-1]}$ might be zero, namely if $j=\ell-1$ and/or $j'=k-1$).

Any imaginary cone of $\F_{B^T}$ containing $\lambda$ is spanned by $\kappa_{[j,\ell-1]}$ and $\kappa_{[j',k-1]}$ and by rays compatible with both.
By Proposition~\ref{nuc phip phi}, ${\br{\kappa_{[j,\ell-1]},\beta_{[1,j]}+\beta_{[\ell+1,j']}}=-1}$ and ${\br{\kappa_{[j',k-1]},\beta_{[1,j]}+\beta_{[\ell+1,j']}}=-1}$.
Furthermore, if $\kappa_{[i',j']}$ spans a ray compatible with both $\kappa_{[j,\ell-1]}$ and $\kappa_{[j',k-1]}$, then ${\br{\kappa_{[i',j']},\beta_{[1,j]}+\beta_{[\ell+1,j']}}\le0}$.

We see that every imaginary cone of $\F_{B^T}$ containing $\lambda$ has its relative interior strictly on the negative side of $(\beta_{[1,j]}+\beta_{[\ell+1,j']})^\perp$.
Since the chi sequence limits to a point in the relative interior of a cone of $\F_{B^T}$ containing $\lambda$, we have contradicted the fact that $\br{\chi_h,\beta_{[1,j]}+\beta_{[\ell+1,j']}}>0$.
Thus Scenario 1+A can't occur.

\smallskip

\noindent
\textbf{Scenario 1+B.}
\textit{$\bl_0$ is as in Case 1 and $\bl_\ell$ is as in Case B.}
In this scenario, $\lambda_{\bl_0}=\kappa_{[j+1,k-1]}-\kappa_{[0,j-1]}$ for some $0<j\le\ell$ and $\lambda_{\bl_\ell}=\kappa_{[\ell+1,i+k-2]}-\kappa_{[i,\ell]}$ for some $0<i\le\ell$.
The requirement that $\lambda=\lambda_{\bl_0}+\lambda_{\bl_\ell}\in\d_\infty$ implies that $j<i$.
Furthermore, in this scenario, $\chi_h$ is on the positive side of $\beta_{[1,j]}^\perp$ and on the negative side of~$\beta_{[i,\ell]}^\perp$ and thus $\br{\chi_h,\beta_{[1,j]}-\beta_{[i,\ell]}}>0$.
Using the fact that $j<i$, we compute ${\lambda=2\kappa_{[\ell+1,k-1]}+\kappa_{[j+1,i-1]}+\kappa_{[j,i-2]}}$.

If $j=i-1$, then $\lambda=2\kappa_{[\ell+1,k-1]}$.
In this case, Proposition~\ref{nuc phip phi} says that $\br{\lambda,\beta_{[1,j]}-\beta_{[i,\ell]}}=0$.
If $j<i-1$, then $\lambda=2\kappa_{[\ell+1,k-1]}+\kappa_{[j,i-1]}+\kappa_{[j+1,i-2]}$ (with $\kappa_{[j+1,i-2]}=0$ if $j=i-2$).
In this case, by Proposition~\ref{nuc phip phi}, ${\br{\lambda,\beta_{[1,j]}-\beta_{[i,\ell]}}=-2}$.
Furthermore, if $\kappa_{[i',j']}$ spans a ray compatible with both $\kappa_{[\ell+1.k-1]}$ and $\kappa_{[j,i-1]}$, then ${\br{\kappa_{[i',j']},\beta_{[1,j]}-\beta_{[i,\ell]}}\le0}$.
As in previous scenarios, we conclude that Scenario 1+B can't happen when $j\neq i-1$.

We see that in Scenario 1+B, $\lambda=\lambda_{\bl_0}+\lambda_{\bl_\ell}=2\kappa_{[\ell+1,k-1]}$, there exists $i$ with $1<i\le\ell$ such that $\chi_h$ is on the positive side of $\beta_{[1,i-1]}^\perp$ and on the negative side of $\beta_{[i,\ell]}^\perp$, and the monomial for this pair is~$y^{\beta_{[1,\ell]}}$.

\smallskip

\noindent
\textbf{Scenario 2+A.}
\textit{$\bl_0$ is as in Case 2 and $\bl_\ell$ is as in Case A.}
This scenario is symmetric to Scenario 1+B by reversing the roles of $0$ and $\ell$.
We see that in Scenario~2+A, $\lambda=\lambda_{\bl_0}+\lambda_{\bl_\ell}=2\kappa_{[1,\ell-1]}$, there exists $i$ with $\ell+1<i\le k$ such that $\chi_h$ is on the positive side of $\beta_{[\ell+1,i-1]}^\perp$ and on the negative side of $\beta_{[i,k]}^\perp$, and the monomial for this pair is~$y^{\beta_{[\ell+1,k]}}$.
%

\smallskip

\noindent
\textbf{Scenario 2+B.}
\textit{$\bl_0$ is as in Case 2 and $\bl_\ell$ is as in Case B.}
In this scenario, $\lambda_{\bl_0}=\kappa_{[1,i-2]}-\kappa_{[i,k]}$ for some $\ell<i\le k$ and $\lambda_{\bl_\ell}=\kappa_{[\ell+1,i'+k-2]}-\kappa_{[i',\ell]}$ for some $0<i'\le\ell$.
Since $\lambda\in\d_\infty$, we see that in fact $\ell+1<i\le k$ and $1<i'\le\ell$.
Furthermore,~$\chi_h$ is on the negative side of $\beta^\perp_{[i,k]}$ and on the negative side of $\beta^\perp_{[i',\ell]}$, so that $\br{\chi_h,\beta_{[i,k]}+\beta_{[i',\ell]}}<0$.
We compute 
\[\lambda=\lambda_{\bl_0}+\lambda_{\bl_\ell}=\kappa_{[\ell+1,i-1]}+\kappa_{[\ell+1,i-2]}+\kappa_{[1,i'-1]}+\kappa_{[1,i'-2]}.\]
This is the sum of four vectors that span rays of an imaginary cone in $\F_{B^T}$ (except that $\kappa_{[\ell+1,i-2]}$ and/or $\kappa_{[1,i'-2]}$ might be zero, namely if $i=\ell+2$ and/or ${i'=2}$).
We rule out this scenario similarly to Scenario 1+A:
Any imaginary cone of $\F_{B^T}$ containing $\lambda$ is spanned by $\kappa_{[\ell+1,i-1]}$ and $\kappa_{[1,i'-1]}$ and rays compatible with both.  
Both $\kappa_{[\ell+1,i-1]}$ and $\kappa_{[1,i'-1]}$ pair to $1$ with $\beta_{[i,k]}+\beta_{[i',\ell]}$, and any vector spanning a ray compatible with both $\kappa_{[\ell+1,i-1]}$ and $\kappa_{[1,i'-1]}$ pairs nonnegatively with $\beta_{[i,k]}+\beta_{[i',\ell]}$.
This contradicts the fact that $\br{\chi_h,\beta_{[i,k]}+\beta_{[i',\ell]}}<0$, and we see that Scenario 2+B can't occur.

\smallskip

We now summarize the scenarios that are not yet ruled out and give $\lambda=\lambda_{\bl_0}+\lambda_{\bl_\ell}$, the monomial in $y$ contributed by the pair, and the restrictions on $\chi_h$ in each case.
\begin{itemize}
\item 
Scenario 0+0:
$\lambda={\kappa_{[1,k-1]}+\kappa_{[\ell+1,k+\ell-1]}}$, monomial $1$, no restriction on~$\chi_h$.
\item
Scenario 1+0: 
$\lambda=2\kappa_{[\ell+1,k-1]}$, monomial $y^{\beta_{[1,\ell]}}$, and $\br{\chi_h,\beta_{[1,\ell]}}>0$.
\item
Scenario 2+0: 
$\lambda=2\kappa_{[1,\ell-1]}$, monomial $y^{\beta_{[\ell+1,k]}}$, and $\br{\chi_h,\beta_{[\ell+1,k]}}<0$.
\item
Scenario 0+A: 
$\lambda=2\kappa_{[1,\ell-1]}$, monomial $y^{\beta_{[\ell+1,k]}}$, and $\br{\chi_h,\beta_{[\ell+1,k]}}>0$.
\item
Scenario 0+B: 
$\lambda=2\kappa_{[\ell+1,k-1]}$, monomial $y^{\beta_{[1,\ell]}}$, and $\br{\chi_h,\beta_{[1,\ell]}}<0$.
\item
Scenario 1+B: 
$\lambda=2\kappa_{[\ell+1,k-1]}$, monomial $y^{\beta_{[1,\ell]}}$, and there exists $i$ with $1<i\le\ell$ such that $\br{\chi_h,\beta_{[1,i-1]}}>0$ and $\br{\chi_h,\beta_{[i,\ell]}}<0$.
\item
Scenario 2+A: 
$\lambda=2\kappa_{[1,\ell-1]}$, monomial $y^{\beta_{[\ell+1,k]}}$, and there exists $i$ with $\ell+1<i\le k$ such that $\br{\chi_h,\beta_{[\ell+1,i-1]}}>0$ and $\br{\chi_h,\beta_{[i,k]}}<0$.
\end{itemize} 
We see that there are only three possibilities for $\lambda$ and all of them are contained in the imaginary cone of $\F_{B^T}$ spanned by $\nu_c(\delta)=\kappa_{[1,k]}$, $\kappa_{[1,\ell-1]}$, and $\kappa_{[\ell+1,k-1]}$.
If $\ell=1$ and/or $\ell=k-1$, then $\kappa_{[1,\ell-1]}=0$ and/or $\kappa_{[\ell+1,k-1]}=0$, but these special cases do not alter the remainder of the argument, except to preemptively eliminate some scenarios.

Our fourth step is to adopt a specific choice of chi sequence.
Since all possibilities for $\lambda$ are in a common imaginary cone, we can choose the same chi sequence for all possible lambda.
There are two convenient choices, related by reversing the roles of $0$ and $\ell$, but we only need one.
Let $L=\sett{\kappa_{[1,k]},\kappa_{[\ell+1,\ell+k-1]},\kappa_{[1,\ell-1]},\kappa_{[\ell+1,k-1]}}$ and let $C$ be the imaginary cone spanned by~$L$.
Choose a chi sequence starting with this choice of $C$, so that the chi sequence limits to the relative interior of a cone $C'$ containing $C$.
Using Proposition~\ref{nuc phip phi}, we make the following observations:
\begin{itemize}
\item 
$\br{\chi_h,\beta_{[1,\ell]}}>0$ because $\br{\xi,\beta_{[1,\ell]}}\ge0$ for all $\xi\in L$ and $\br{\xi,\beta_{[1,\ell]}}=1$ for $\xi=\kappa_{[\ell+1,\ell+k-1]}\in L$.
\item 
$\br{\chi_h,\beta_{[\ell+1,k]}}<0$ because $\br{\xi,\beta_{[\ell+1,k]}}\le0$ for all $\xi\in L$ and $\br{\xi,\beta_{[\ell+1,k]}}=-1$ for $\xi=\kappa_{[\ell+1,\ell+k-1]}\in L$.
\item
$\br{\chi_h,\beta_{[1,i-1]}}<0$ and $\br{\chi_h,\beta_{[i,\ell]}}>0$ for all $1<i\le\ell$.  
This is because $\br{\xi,\beta_{[1,i-1]}}\le0$ for all $\xi\in L$ and ${\br{\xi,\beta_{[1,i-1]}}=-1}$ for $\xi=\kappa_{[1,\ell-1]}\in L$ and because $\br{\xi,\beta_{[i,\ell]}}\ge0$ for all $\xi\in L$ and $\br{\xi,\beta_{[i,\ell]}}=1$ for ${\xi=\kappa_{[\ell+1,\ell+k-1]}\in L}$.
\item
$\br{\chi_h,\beta_{[\ell+1,i-1]}}<0$ and $\br{\chi_h,\beta_{[i,k]}}>0$ for all $\ell+1<i\le k$.  
This is because $\br{\xi,\beta_{[\ell+1,i-1]}}\le0$ for all $\xi\in L$ and ${\br{\kappa_{[\ell+1,\ell+k-1]},\beta_{[\ell+1,i-1]}}=-1}$ and because $\br{\xi,\beta_{[i,k]}}\ge0$ for all $\xi\in L$ and ${\br{\kappa_{[\ell+1,k-1]},\beta_{[i,k]}}=1}$.
\end{itemize}
Comparing these observations with the summary of scenarios above, we rule out all scenarios except Scenarios 0+0, 1+0, and 2+0, for this choice of chi sequence.

It may appear that the proof is finished, but recall that each of Cases 1 and 2 specify more that one broken line.  
Our fifth step is to complete the proof by determining which broken lines $\bl_\ell$, for this choice of chi sequence, can appear in Scenarios 1+0 and 2+0.

In Scenario 1+0, $\bl_0$ ends on the positive side of $\beta_{[1,\ell]}^\perp$, but may have bent multiple times.
Suppose that $\bl_0$ bends more than once.
By inspection of Case 1, we see that for the last bend, $\bl_0$ starts on the negative side of $\beta_{[j+1,\ell]}^\perp$ and the positive side of $\beta_{[1,j]}^\perp$ for some $0<j<\ell$ and ends up still on the positive side of $\beta_{[1,j]}^\perp$ and passes to the positive side of $\beta_{[j+1,\ell]}^\perp$.
This is a contradiction to the observations above, which say that $\br{\chi_h,\beta_{[1,j]}}<0$.
We conclude that $\bl_0$ bends only once, on the wall normal to $\beta_{[1,\ell]}$.
Since $\bl_\ell$ never bends, there is a unique pair $(\bl_0,\bl_\ell)$ in Scenario~1+0.

Similarly, in Scenario 2+0, if $\bl_\ell$ bends more than once, its last bend starts and ends on the negative side of $\beta_{[i,k]}^\perp$ for some $\ell+1<i\le k$, contradicting the observations above, which say that $\br{\chi_h,\beta_{[i,k]}}>0$.
We conclude that $\bl_\ell$ bends only once, on the wall normal to $\beta_{[\ell+1,k]}$, and that there is a unique pair $(\bl_0,\bl_\ell)$ in Scenario 2+0.

In all, there are three pairs, which contribute the three desired terms in the expansion of $\thet_{\kappa_{[1,k-1]}}\cdot\thet_{\kappa_{[\ell+1,k+\ell-1]}}$.
\qed

\subsection{Proof of Theorems~\ref{subalgebra} and~\ref{tube subalgebra}}  
The first assertion of Theorem~\ref{subalgebra} is a restatement of Theorem~\ref{expand d inf}.
The second assertion follows from Theorems~\ref{thet xi}, \ref{delta cheby} and~\ref{expansion prod}, because $\d_\infty$ is a union of cones of $\F_{B^T}$.
The second assertion of Theorem~\ref{tube subalgebra} follows from the first assertion in the same way.
The first assertion of Theorem~\ref{tube subalgebra} is a restatement of the following theorem, which refines Theorem~\ref{expand d inf} by decomposing $\SimplesT{c}$ into $c$-orbits.
Recall that $\d_\infty$ is the nonnegative linear span of the vectors $\sett{\nu_c(\beta):\beta\in\SimplesT{c}}$.

\begin{theorem}\label{expand tube}
Suppose $B$ is acyclic of affine type and let $\tB$ be an extension of~$B$. 
Suppose $\SimplesT{c}_o=\set{\beta_1,\ldots,\beta_k}$ is a $c$-orbit in $\SimplesT{c}$ and suppose $\v$ is a monomial in a finite set of theta functions $\thet_\lambda$, with each $\lambda$ in the nonnegative integer span of~$\nu_c(\SimplesT{c}_o)$.
Then $\v$ is a finite $\k[y]$-linear combination of theta functions $\thet_\kappa$, with each $\kappa$ in the nonnegative integer span of $\nu_c(\SimplesT{c}_o)$.
\end{theorem}

Theorem~\ref{expand tube} is readily proved using results that we have already proved.  
It is enough to prove the theorem for a product of two theta functions $\thet_{p_1}$ and $\thet_{p_2}$ with $p_1,p_2$ in the nonnegative integer span of $\nu_c(\SimplesT{c}_o)$.
Theorem~\ref{boundary with delta} says that $\thet_{p_1}\cdot\thet_{p_2}$ is $\thet_{m_1\nu_c(\delta)}\cdot\thet_{m_2\nu_c(\delta)}\cdot\thet_{q_1}\cdot\thet_{q_2}$, where $q_1=p_1-m_1\nu_c(\delta)$ and $q_2=p_2-m_2\nu_c(\delta)$ are both in the relative boundary of~$\d_\infty$.

Since $\delta$ is the sum of the orbit $\SimplesT{c}_o$, $q_1$ and $q_2$ are also in the nonnegative integer span of $\nu_c(\SimplesT{c}_o)$.
Theorem~\ref{simplification for boundary dinf} implies that all broken lines used to compute structure constants for $\thet_{q_1}\cdot\thet_{q_2}$ only bend on walls whose normal vectors are in~$\APTre{c}$.
But if ${\gamma\in\APTre{c;o'}}$ with $o\neq o'$, then, by Proposition~\ref{nuc phip phi}, $\br{\nu_c(q),\gamma}=0$ for any $q$ in the integer span of $\nu_c(\SimplesT{c}_o)$.
Thus the first bend of a broken line used to compute structure constants for $\thet_{q_1}\cdot\thet_{q_2}$ is on a wall whose normal vector is in~$\APTre{c;o}$.
Furthermore, Proposition~\ref{crucial prop} implies that, after the first bend, the monomial on the broken line is still in the integer span of $\nu_c(\SimplesT{c}_o)$.
We see that the broken line only bends on walls whose normal vectors are in~$\APTre{c;o}$, and the monomial obtained from the broken line is in the integer span of $\nu_c(\beta_1),\ldots,\nu_c(\beta_k)$.
Thus $\thet_{q_1}\cdot\thet_{q_2}$ expands as a $\k[y]$-linear combination of theta functions $\thet_\kappa$, with each~$\kappa$ in the integer span of $\nu_c(\SimplesT{c}_o)$.
Theorem~\ref{expand d inf} now says that all of these~$\kappa$ are in the \emph{nonnegative} integer span of $\nu_c(\SimplesT{c}_o)$.

Theorem~\ref{imag general} says that $\thet_{m_1\nu_c(\delta)}\cdot\thet_{m_2\nu_c(\delta)}$ is a finite $\k[y]$-linear combination of theta functions $\thet_\kappa$, with each $\kappa$ a multiple of $\nu_c(\delta)$.
Applying Theorem~\ref{boundary with delta} again, we see that $\thet_{m_1\nu_c(\delta)}\cdot\thet_{m_2\nu_c(\delta)}\cdot\thet_{q_1}\cdot\thet_{q_2}$ is a finite $\k[y]$-linear combination of theta functions $\thet_\kappa$, with each $\kappa$ in the nonnegative integer span of $\nu_c(\SimplesT{c}_o)$.
\qed

\subsection{Proof of Theorems~\ref{gen clus alg prop}, \ref{gen clus alg} and~\ref{gen clus alg almost}}
Most of the work remaining for the proof of these theorems is accomplished by two lemmas that we now work to prove (Lemmas~\ref{J mut} and~\ref{x mut}).
We continue to use the combinatorics of compatibility among roots in~$\APTre{c}$, as reviewed in Section~\ref{agaf sec}.

\begin{lemma}\label{J mut}
Suppose $J_o$ and $J_o'$ are distinct maximal sets of pairwise compatible real roots in $\APTre{c;o}$ with $J_o'=(J_o\setminus\set\gamma)\cup\set{\gamma'}$.
The generalized seed $(\x'_{J_o},\pp'_{J_o},B'_{J_o})$ obtained from $(\x_{J_o},\pp_{J_o},B_{J_o})$ by mutating $x_\gamma$ to obtain $x'_\gamma$ coincides with the generalized seed $(\x_{J_o'},\pp_{J_o'},B_{J_o'})$, identifying $x'_\gamma$ with $x_{\gamma'}$.
\end{lemma}

The following lemmas simplify the proof of Lemma~\ref{J mut}.
Suppose $J_o$ is a maximal set of pairwise compatible real roots in $\APTre{c;o}$ and choose $\gamma\in J_o$.
Let $R_\gamma$ be the subset of $J_o$ consisting of $\gamma$, the zero, one, or two next smaller roots from $\gamma$ in~$J_o$, the zero or one next larger roots from $\gamma$ in $J_o$ and the zero or one roots in $J_o$ that have the same next larger root as $\gamma$.
The definition of $B_{J_o}$ indicates that if $\psi\not\in R_\gamma$, then $b_{\psi\gamma}=0$.
Thus the skew-symmetrizability of $B_{J_o}$ implies that also $b_{\gamma\psi}=0$.

\begin{lemma}\label{unchanging exch}
Suppose $J_o'$ is obtained by exchanging $\gamma$ from~$J_o$.
\begin{enumerate}[label=\bf\arabic*., ref=\arabic*] 
\item \label{unch B exch}   
If $\psi,\phi\in J_o$ have $\set{\psi,\phi}\not\subseteq R_\gamma$, then the $\psi,\phi$-entry of $B_{J_o}$ is the same as the $\psi,\phi$-entry of $B_{J_o'}$.
\item\label{unch p exch}
If $\psi\in J_o$ has $\psi\not\in R_\gamma$, then $p_\psi$ is the same in $\pp_{J_o}$ as in $\pp_{J_o'}$.
\end{enumerate}
\end{lemma}

\begin{proof}
Exchanging $\gamma$ from $J_o$ does not change any next-smallest-root relationships among roots in $J_o$ except among roots in $R_\gamma$.
Thus the lemma follows from the four observations below for $\psi,\phi\in J_o$.

First, if $\SuppT(\psi)\not\subseteq\SuppT(\phi)$ and $\SuppT(\phi)\not\subseteq\SuppT(\psi)$, then $b_{\psi\phi}$ and $b_{\phi\psi}$ are nonzero if and only if $\psi$ and $\phi$ have the same next larger root in~$J_o$, or equivalently, they are next smallest roots from the same root.
Furthermore, when the values of $b_{\psi\phi}$ and $b_{\phi\psi}$ are nonzero, they are determined by where their support sits inside that next larger root.
Second, if $\SuppT(\psi)\subseteq\SuppT(\phi)$, then $b_{\psi\phi}$ and $b_{\phi\psi}$ are nonzero if and only if $\psi$ is a next smaller root from $\phi$ in $J_o$ (or equivalently, $\phi$ is the next larger root from $\psi$).
Furthermore, when the values of $b_{\psi\phi}$ and $b_{\phi\psi}$ are nonzero, they are determined by where $\SuppT(\psi)$ sits inside $\SuppT(\phi)$.
Third, if $\psi$ is maximal in $J_o$, then $p_\psi$ is determined by the next smaller roots from $\psi$ in $J_o$.
Fourth, if $\psi$ is not maximal in $J_o$, then $p_\psi$ is determined by the next smaller roots from $\psi$ in $J_o$, from the next larger root from $\psi$ in $J_o$, and from the root that has the same next larger root from $\psi$.
\end{proof}

\begin{lemma}\label{unchanging mut}
Suppose $(\x'_{J_o},\pp'_{J_o},B'_{J_o})$ is obtained from $(\x_{J_o},\pp_{J_o},B_{J_o})$ by mutating at~$\gamma$.
\begin{enumerate}[label=\bf\arabic*., ref=\arabic*] 
\item \label{unch B mut}   
If $\psi,\phi\in J_o$ have $\set{\psi,\phi}\not\subseteq R_\gamma$, then the $\psi,\phi$-entry of $B_{J_o}$ is the same as the $\psi,\phi$-entry of $B'_{J_o}$.
\item\label{unch p mut}
If $\psi\in J_o$ has $\psi\not\in R_\gamma$, then $p_\psi$ in $\pp_{J_o}$ is the same as $p_\psi$ in $\pp'_{J_o}$.
\end{enumerate}
\end{lemma}
\begin{proof}
Suppose $\set{\psi,\phi}\not\subseteq R_\gamma$.
If $\psi=\gamma$, then $\phi\not\in R_\gamma$, so $b_{\psi\phi}=0$.
Also, $b'_{\psi\phi}=-b_{\psi\phi}=0$.
Similarly, if $\phi=\gamma$, then $b'_{\psi\phi}=-b_{\psi\phi}=0$.
If $\gamma\not\in\set{\psi,\phi}$, then since at least one of $\psi$ or $\phi$ is not in $R_\gamma$, at least one of $b_{\psi\gamma}$ or $b_{\gamma\phi}$ is zero.
Therefore $b'_{\psi\phi}=-b_{\psi\phi}$.

If $\psi\not\in R_\gamma$, then $b_{\gamma\psi}=0$ and $p'_{\psi;\ell}= \frac{p_{\psi;\ell}}{p_{\psi;0}\oplus p_{\psi;d_\psi}}=p_{\psi;\ell}$ for all $\ell$  from $0$ to $d_\psi$.
\end{proof}

\begin{proof}[Proof of Lemma~\ref{J mut}]
Identifying $x'_\gamma$ with $x_{\gamma'}$, the clusters $\x'_{J_o}$ and $\x_{J_o'}$ coincide by construction.
We must check that $(\pp'_{J_o},B'_{J_o})$ and $(\pp_{J_o'},B_{J_o'})$ coincide.
There are three cases.
The relevant information for the first two cases is shown in Figures~\ref{proof fig max} and~\ref{proof fig nonmax}.  
The third case will follow immediately from the second.
In each figure, we omit the entries of exchange matrices and the coefficients that don't change when~$\gamma$ is exchanged out of $J_o$, according to Lemma~\ref{unchanging exch}.
According to Lemma~\ref{unchanging mut}, those entries and coefficients also don't change when $(\x_{J_o},\pp_{J_o},B_{J_o})$ is mutated at position~$\gamma$.
Exchange matrix entries are shown in tables.
In each figure, information for the seed determined by $J_o$ is in the top row and information for $J_o'$ in the bottom row.
In each case, we need to show that the matrix entries and coefficients shown in the second row agree with the entries and coefficients obtained from the data in the first row by generalized seed mutation at position~$\gamma$.
In both cases, the matrix entries are easily checked to be obtained by matrix mutation as desired.
\begin{figure}
\begin{tabular}{|ccc|}\hline&&\\[-5pt]
\begin{minipage}[m]{139pt}
\scalebox{0.95}{\includegraphics{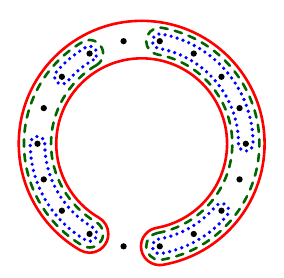}
\begin{picture}(0,0)(68,-64)
\put(-12,-36){$\beta$}
\put(-15,29){$\beta'$}
\put(23,-13){$\beta''$}
\put(-39,11){$\beta'''$}
\put(10,63){\textcolor{red}{$\gamma$}}
\put(-71,10){\textcolor{darkgreen}{$\phi$}}
\put(55,-23){\textcolor{darkgreen}{$\phi'$}}
\put(-63,-42){\textcolor{blue}{$\psi_1$}}
\put(-53,50){\textcolor{blue}{$\psi_2$}}
\put(45,42){\textcolor{blue}{$\psi_3$}}
\put(27,-60){\textcolor{blue}{$\psi_4$}}
\put(-7.5,-2){\large$J_o$}
\end{picture}}
\end{minipage}
&
\begin{minipage}[m]{115pt}
$
\begin{array}{l|rrr}
		&\phi		&\phi'	&\gamma\\\hline
\phi		&0		&-1		&2\\
\phi'		&1		&0		&-2\\
\gamma	&-1		&1		&0
\end{array}
$
\end{minipage}
& 
\begin{minipage}[m]{70pt}
$\begin{aligned}
p_{\gamma;0}&=z^{\phi'+\beta}\\
p_{\gamma;1}&=z_*\\
p_{\gamma;2}&=z^{\phi+\beta'}\\
p_{\phi;0}&=z^{\psi_2+\beta'}\\
p_{\phi;1}&=1\\
p_{\phi';0}&=1\\
p_{\phi';1}&=z^{\psi_3+\beta''}
\end{aligned}$
\end{minipage}\\[64pt]\hline&&\\[-5pt]
\begin{minipage}[m]{139pt}
\scalebox{0.95}{\includegraphics{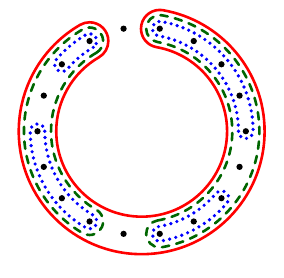}
\begin{picture}(0,0)(68,-70)
\put(-12,-36){$\beta$}
\put(-15,29){$\beta'$}
\put(23,-13){$\beta''$}
\put(-39,11){$\beta'''$}
\put(-27,-69){\textcolor{red}{$\gamma'$}}
\put(-71,10){\textcolor{darkgreen}{$\phi$}}
\put(55,-23){\textcolor{darkgreen}{$\phi'$}}
\put(-63,-42){\textcolor{blue}{$\psi_1$}}
\put(-53,50){\textcolor{blue}{$\psi_2$}}
\put(45,42){\textcolor{blue}{$\psi_3$}}
\put(27,-60){\textcolor{blue}{$\psi_4$}}
\put(-7.5,-2){\large$J_o'$}
\end{picture}}
\end{minipage}
&
\begin{minipage}[m]{115pt}
$
\begin{array}{l|rrr}
		&\phi		&\phi'	&\gamma'\\\hline
\phi		&0		&1		&-2\\
\phi'		&-1		&0		&2\\
\gamma'	&1		&-1		&0
\end{array}
$
\end{minipage}
& 
\begin{minipage}[m]{70pt}
$\begin{aligned}
p_{\gamma';0}&=z^{\phi+\beta'}\\
p_{\gamma';1}&=z_*\\
p_{\gamma';2}&=z^{\phi'+\beta}\\
p_{\phi;0}&=1\\
p_{\phi;1}&=z^{\psi_1+\beta'''}\\
p_{\phi';0}&=z^{\psi_4+\beta}\\
p_{\phi';1}&=1
\end{aligned}$
\end{minipage}\\[62pt]
\hline
\end{tabular}
\caption{Seeds for $J_o$ and $J_o'=(J_o\setminus\set\gamma)\cup\gamma'$ for $\gamma$ maximal}
\label{proof fig max}
\vspace{-0.5em}
\end{figure}
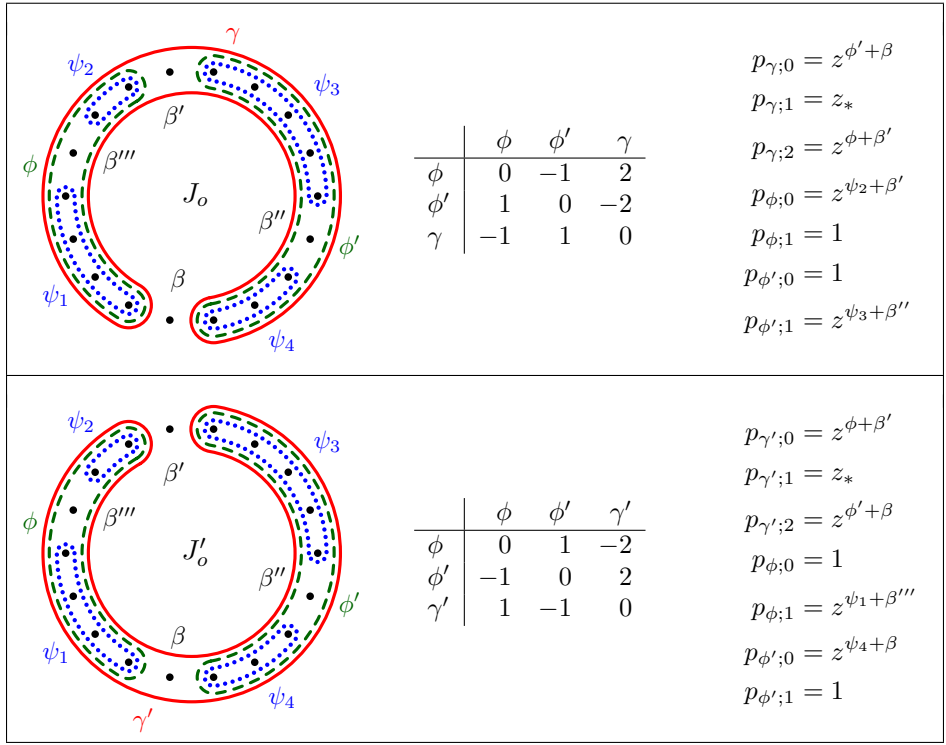

In the case where $\gamma$ is maximal in $J_o$ (Figure~\ref{proof fig max}), we now check the coefficients by writing down the formulas defining generalized seed mutation.
\begin{align*}
p'_{\gamma';0}&=p_{\gamma;2}=z^{\phi+\beta'}\\
p'_{\gamma';1}&=p_{\gamma;1}=z_*\\
p'_{\gamma';2}&=p_{\gamma;0}=z^{\phi'+\beta}\\
p'_{\phi;0}
&=\frac{p_{\phi;0}\,p_{\gamma;0}^{[b_{\gamma \phi}]_+}\,p_{\gamma;2}^{0}}{p_{\phi;0}\,p_{\gamma;0}^{[b_{\gamma \phi}]_+}\oplus p_{\phi;1}\,p_{\gamma;2}^{[-b_{\gamma \phi}]_+}}
=\frac{z^{\psi_2+\beta'}}{z^{\psi_2+\beta'}\oplus z^{\phi+\beta'}}
=1\\
p'_{\phi;1}
&=\frac{p_{\phi;1}\,p_{\gamma;0}^{0}\,p_{\gamma;2}^{[-b_{\gamma \phi}]_+}}{p_{\phi;0}\,p_{\gamma;0}^{[b_{\gamma \phi}]_+}\oplus p_{\phi;1}\,p_{\gamma;2}^{[-b_{\gamma \phi}]_+}}
=\frac{z^{\phi+\beta'}}{z^{\psi_2+\beta'}\oplus z^{\phi+\beta'}}
=z^{\psi_1+\beta'''}\\
p'_{\phi';0}
&=\frac{p_{\phi';0}\,p_{\gamma;0}^{[b_{\gamma \phi'}]_+}\,p_{\gamma;2}^{0}}{p_{\phi';0}\,p_{\gamma;0}^{[b_{\gamma \phi'}]_+}\oplus p_{\phi';1}\,p_{\gamma;2}^{[-b_{\gamma \phi'}]_+}}
=\frac{z^{\phi'+\beta}}{z^{\phi'+\beta}\oplus z^{\psi_3+\beta''}}
=z^{\psi_4+\beta}\\
p'_{\phi';1}
&=\frac{p_{\phi';1}\,p_{\gamma;0}^{0}\,p_{\gamma;2}^{[-b_{\gamma \phi'}]_+}}{p_{\phi';0}\,p_{\gamma;0}^{[b_{\gamma \phi'}]_+}\oplus p_{\phi';1}\,p_{\gamma;2}^{[-b_{\gamma \phi'}]_+}}
=\frac{z^{\psi_3+\beta''}}{z^{\phi'+\beta}\oplus z^{\psi_3+\beta''}}
=1.
\end{align*}

\begin{figure}
\begin{tabular}{|cc|}\hline&\\[-5pt]
\begin{minipage}[m]{175pt}
\qquad\scalebox{1.0}{\includegraphics{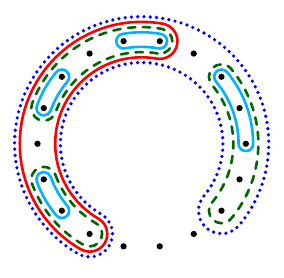}
\begin{picture}(0,0)(68,-64)
\put(-40,-2){\small$\beta$}
\put(8,24){\small$\beta'$}
\put(-27,25){\small$\beta''$}
\put(19,-12){\small$\beta'''$}
\put(17,-54){\small$\beta^{(4)}$}
\put(-17,-59){\small$\beta^{(5)}$}
\put(-71,13){\textcolor{red}{$\gamma$}}
\put(16,62){\textcolor{blue}{$\phi$}}
\put(-60,-50){\textcolor{darkgreen}{$\phi'$}}
\put(-45,53){\textcolor{darkgreen}{$\phi''$}}
\put(57,-25){\textcolor{darkgreen}{$\phi'''$}}
\put(-72,-29){\textcolor{lightblue}{$\psi_1$}}
\put(-71,30){\textcolor{lightblue}{$\psi_2$}}
\put(-7,66){\textcolor{lightblue}{$\psi_3$}}
\put(56,23){\textcolor{lightblue}{$\psi_4$}}
\put(-7.5,-5){\large$J_o$}
\end{picture}}\qquad
\end{minipage}
&
\begin{minipage}[m]{175pt}
$
\begin{array}{l|rrrrr}
		&\phi		&\phi'	&\phi''	&\phi'''	&\gamma\\\hline
\phi		&0		&0		&0		&1		&-1\\
\phi'		&0		&0		&-1		&0		&1\\
\phi''		&0		&1		&0		&0		&-1\\
\phi'''		&-d_\phi	&0		&0		&0		&1\\
\gamma	&d_\phi	&-1		&1		&-1		&0
\end{array}$\\[10pt]

$\begin{array}{ll}
\begin{aligned}
p_{\gamma;0}&=z^{\phi''+\beta'}\\
p_{\gamma;1}&=1
\end{aligned}
\quad&
\begin{aligned}
p_{\phi';0}&=z^{\psi_1+\beta}\\
p_{\phi';1}&=1
\end{aligned}
\\[20pt]
\begin{aligned}
p_{\phi'';0}&=1\\
p_{\phi'';1}&=z^{\psi_2+\beta''}
\end{aligned}
\quad&
\begin{aligned}
p_{\phi''';0}&=1\\
p_{\phi''';1}&=z^{\psi_4+\beta'''}
\end{aligned}
\end{array}$
\end{minipage}
\\[80pt]
\multicolumn{2}{|c|}{If $\beta^{(4)}\neq\beta^{(5)}$, write $\bar\phi$ for the next larger root from $\phi$ in $J_o$ (not pictured).}\\[5pt]
\multicolumn{2}{|c|}{
$\left\{
\begin{array}{lllll}
p_{\phi;0}=z^{\phi'''+\beta^{(4)}}&
p_{\phi;1}=z_*&
p_{\phi;2}=z^{\gamma+\beta'}&&
\text{if }\beta^{(4)}=\beta^{(5)}\\[5pt]
p_{\phi;0}=z^{\phi'''+\beta^{(4)}}&
p_{\phi;1}=1&&&
\text{if }\beta^{(4)}\in\SuppT(\bar\phi)\\[5pt]
p_{\phi;0}=1&
p_{\phi;1}=z^{\gamma+\beta'}&&&
\text{if }\beta^{(5)}\in\SuppT(\bar\phi)
\end{array}\right.$
}\\[25pt]\hline&\\[-5pt]
\begin{minipage}[m]{139pt}
\scalebox{1.0}{\includegraphics{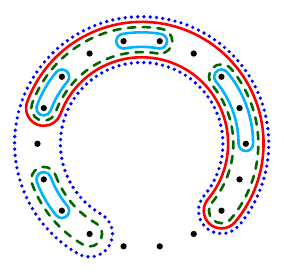}
\begin{picture}(0,0)(68,-64)
\put(-40,-2){\small$\beta$}
\put(8,24){\small$\beta'$}
\put(-27,25){\small$\beta''$}
\put(19,-12){\small$\beta'''$}
\put(17,-54){\small$\beta^{(4)}$}
\put(-17,-59){\small$\beta^{(5)}$}
\put(42,45){\textcolor{red}{$\gamma'$}}
\put(16,62){\textcolor{blue}{$\phi$}}
\put(-60,-50){\textcolor{darkgreen}{$\phi'$}}
\put(-45,53){\textcolor{darkgreen}{$\phi''$}}
\put(57,-25){\textcolor{darkgreen}{$\phi'''$}}
\put(-72,-29){\textcolor{lightblue}{$\psi_1$}}
\put(-71,30){\textcolor{lightblue}{$\psi_2$}}
\put(-7,66){\textcolor{lightblue}{$\psi_3$}}
\put(56,23){\textcolor{lightblue}{$\psi_4$}}
\put(-7.5,-5){\large$J_o'$}
\end{picture}}\qquad
\end{minipage}
&
\begin{minipage}[m]{175pt}
$
\begin{array}{l|rrrrr}
		&\phi		&\phi'	&\phi''	&\phi'''	&\gamma'\\\hline
\phi		&0		&-1		&0		&0		&1\\
\phi'		&d_\phi	&0		&0		&0		&-1\\
\phi''		&0		&0		&0		&-1		&1\\
\phi'''		&0		&0		&1		&0		&-1\\
\gamma'	&-d_\phi	&1		&-1		&1		&0
\end{array}$\\[10pt]

$\begin{array}{ll}
\begin{aligned}
p_{\gamma';0}&=1\\
p_{\gamma';1}&=z^{\phi''+\beta'}
\end{aligned}
\quad&
\begin{aligned}
p_{\phi';0}&=z^{\psi_1+\beta}\\
p_{\phi';1}&=1
\end{aligned}
\\[20pt]
\begin{aligned}
p_{\phi'';0}&=z^{\psi_3+\beta'}\\
p_{\phi'';1}&=1
\end{aligned}
\quad&
\begin{aligned}
p_{\phi''';0}&=1\\
p_{\phi''';1}&=z^{\psi_4+\beta'''}
\end{aligned}
\end{array}$
\end{minipage}
\\[80pt]
\multicolumn{2}{|c|}{If $\beta^{(4)}\neq\beta^{(5)}$, then $\bar\phi$ is also the next larger root from $\phi$ in $J_o'$.}\\[5pt]
\multicolumn{2}{|c|}{
$\left\{
\begin{array}{lllll}
p_{\phi;0}=z^{\gamma'+\beta^{(4)}}&
p_{\phi;1}=z_*&
p_{\phi;2}=z^{\phi'+\beta}&&
\text{if }\beta^{(4)}=\beta^{(5)}\\[5pt]
p_{\phi;0}=z^{\gamma'+\beta^{(4)}}&
p_{\phi;1}=1&&&
\text{if }\beta^{(4)}\in\SuppT(\bar\phi)\\[5pt]
p_{\phi;0}=1&
p_{\phi;1}=z^{\phi'+\beta}&&&
\text{if }\beta^{(5)}\in\SuppT(\bar\phi)
\end{array}\right.$
}\\[25pt]\hline\end{tabular}
\caption{Seeds for $J_o$ and $J_o'=(J_o\setminus\set\gamma)\cup\gamma'$ for $\gamma$ not maximal}
\label{proof fig nonmax}
\end{figure}

\clearpage

We now check coefficients for one of the cases where $\gamma$ is not maximal (Figure~\ref{proof fig nonmax}).
\begin{align*}
p'_{\gamma';0}&=p_{\gamma;1}=1\\
p'_{\gamma';1}&=p_{\gamma;0}=z^{\phi''+\beta'}\\
p'_{\phi';0}&
=\frac{p_{\phi';0}p_{\gamma ;0}^{[b_{\gamma \phi'}]_+}}{p_{\phi';0}\,p_{\gamma ;0}^{[b_{\gamma \phi'}]_+}\oplus p_{\phi';1}\,p_{\gamma ;1 }^{[-b_{\gamma \phi'}]_+}}
=\frac{z^{\psi_1+\beta}}{z^{\psi_1+\beta}\oplus 1}
=z^{\psi_1+\beta}\\
p'_{\phi';1}&
=\frac{p_{\phi';1}p_{\gamma ;1 }^{[-b_{\gamma \phi'}]_+}}{p_{\phi';0}\,p_{\gamma ;0}^{[b_{\gamma \phi'}]_+}\oplus p_{\phi';1}\,p_{\gamma ;1 }^{[-b_{\gamma \phi'}]_+}}
=\frac{1}{z^{\psi_1+\beta}\oplus 1}
=1\\
p'_{\phi'';0}&
=\frac{p_{\phi'';0}p_{\gamma ;0}^{[b_{\gamma \phi''}]_+}}{p_{\phi'';0}\,p_{\gamma ;0}^{[b_{\gamma \phi''}]_+}\oplus p_{\phi'';1}\,p_{\gamma ;1 }^{[-b_{\gamma \phi''}]_+}}
=\frac{z^{\phi''+\beta'}}{z^{\phi''+\beta'}\oplus z^{\psi_2+\beta''}}
=z^{\psi_3+\beta'}\\
p'_{\phi'';1}&
=\frac{p_{\phi'';1}p_{\gamma ;1 }^{[-b_{\gamma \phi''}]_+}}{p_{\phi'';0}\,p_{\gamma ;0}^{[b_{\gamma \phi''}]_+}\oplus p_{\phi'';1}\,p_{\gamma ;1 }^{[-b_{\gamma \phi''}]_+}}
=\frac{z^{\psi_2+\beta''}}{z^{\phi''+\beta'}\oplus z^{\psi_2+\beta''}}
=1\\
p'_{\phi''';0}&
=\frac{p_{\phi''';0}p_{\gamma ;0}^{[b_{\gamma \phi'''}]_+}}{p_{\phi''';0}\,p_{\gamma ;0}^{[b_{\gamma \phi'''}]_+}\oplus p_{\phi''';1}\,p_{\gamma ;1 }^{[-b_{\gamma \phi'''}]_+}}
=\frac{1}{1\oplus z^{\psi_4+\beta'''}}
=1\\
p'_{\phi''';1}&
=\frac{p_{\phi''';1}p_{\gamma ;1 }^{[-b_{\gamma \phi'''}]_+}}{p_{\phi''';0}\,p_{\gamma ;0}^{[b_{\gamma \phi'''}]_+}\oplus p_{\phi''';1}\,p_{\gamma ;1 }^{[-b_{\gamma \phi'''}]_+}}
=\frac{z^{\psi_4+\beta'''}}{1\oplus z^{\psi_4+\beta'''}}
=z^{\psi_4+\beta'''}
\end{align*}

If $\beta^{(4)}=\beta^{(5)}$, then
\begin{align*}
p'_{\phi;0}&
=\frac{p_{\phi;0}\,p_{\gamma ;0}^{[b_{\gamma \phi}]_+}}{p_{\phi;0}\,p_{\gamma ;0}^{[b_{\gamma \phi}]_+}\oplus p_{\phi;2}\,p_{\gamma ;1 }^{[-b_{\gamma \phi}]_+}}
=\frac{z^{\phi'''+\beta^{(4)}}z^{2\phi''+2\beta'}}{z^{\phi'''+\beta^{(4)}}z^{2\phi''+2\beta'}\oplus z^{\gamma+\beta'}}
=z^{\gamma'+\beta^{(4)}}\\
p'_{\phi;1}&
=\frac{p_{\phi;1}\,p_{\gamma ;0}^{\frac{1}{2}[b_{\gamma \phi}]_+}\,p_{\gamma ;1 }^{\frac{1}{2}[-b_{\gamma \phi}]_+}}{p_{\phi;0}\,p_{\gamma ;0}^{[b_{\gamma \phi}]_+}\oplus p_{\phi;2}\,p_{\gamma ;1 }^{[-b_{\gamma \phi}]_+}}
=\frac{z_*z^{\phi''+\beta'}}{z^{\phi'''+\beta^{(4)}}z^{2\phi''+2\beta'}\oplus z^{\gamma+\beta'}}
=z_*\\
p'_{\phi;2}&
=\frac{p_{\phi;2}p_{\gamma ;1 }^{[-b_{\gamma \phi}]_+}}{p_{\phi;0}\,p_{\gamma ;0}^{[b_{\gamma \phi}]_+}\oplus p_{\phi;2}\,p_{\gamma ;1 }^{[-b_{\gamma \phi}]_+}}
=\frac{z^{\gamma+\beta'}}{z^{\phi'''+\beta^{(4)}}z^{2\phi''+2\beta'}\oplus z^{\gamma+\beta'}}
=z^{\phi'+\beta}
\end{align*}

If $\beta^{(4)}\in\SuppT(\bar\phi)$, then
\begin{align*}
p'_{\phi;0}&
=\frac{p_{\phi;0}p_{\gamma ;0}^{[b_{\gamma \phi}]_+}}{p_{\phi;0}\,p_{\gamma ;0}^{[b_{\gamma \phi}]_+}\oplus p_{\phi;1}\,p_{\gamma ;1 }^{[-b_{\gamma \phi}]_+}}
=\frac{z^{\phi'''+\beta^{(4)}}z^{\phi''+\beta'}}{z^{\phi'''+\beta^{(4)}}z^{\phi''+\beta'}\oplus 1}
=z^{\gamma'+\beta^{(4)}}\\
p'_{\phi;1}&
=\frac{p_{\phi;1}p_{\gamma ;1 }^{[-b_{\gamma \phi}]_+}}{p_{\phi;0}\,p_{\gamma ;0}^{[b_{\gamma \phi}]_+}\oplus p_{\phi;1}\,p_{\gamma ;1 }^{[-b_{\gamma \phi}]_+}}
=\frac{1}{z^{\phi'''+\beta^{(4)}}z^{\phi''+\beta'}\oplus 1}
=1
\end{align*}

If $\beta^{(5)}\in\SuppT(\bar\phi)$, then
\begin{align*}
p'_{\phi;0}&
=\frac{p_{\phi;0}p_{\gamma ;0}^{[b_{\gamma \phi}]_+}}{p_{\phi;0}\,p_{\gamma ;0}^{[b_{\gamma \phi}]_+}\oplus p_{\phi;1}\,p_{\gamma ;1 }^{[-b_{\gamma \phi}]_+}}
=\frac{z^{\phi''+\beta'}}{z^{\phi''+\beta'}\oplus z^{\gamma+\beta'}}
=1\\
p'_{\phi;1}&
=\frac{p_{\phi;1}p_{\gamma ;1 }^{[-b_{\gamma \phi}]_+}}{p_{\phi;0}\,p_{\gamma ;0}^{[b_{\gamma \phi}]_+}\oplus p_{\phi;1}\,p_{\gamma ;1 }^{[-b_{\gamma \phi}]_+}}
=\frac{z^{\gamma+\beta'}}{z^{\phi''+\beta'}\oplus z^{\gamma+\beta'}}
=z^{\phi'+\beta}
\end{align*}
In the other case where $J_o$ is not maximal, we could obtain the analog of Figure~\ref{proof fig nonmax} by swapping the top row and the bottom row, swapping the symbols $\gamma$ and $\gamma'$, and swapping the symbols $J_o$ and $J_o'$.
Since generalized seed mutation is an involution, the computations above imply that the analogous computations work in the other direction.
\end{proof}

We consider the map $\mapchar_o$ as in Section~\ref{imag sec}, but temporarily write $\mapchar_o^{J_o}$ to record the choice of~$J_o$.
Thus~$\mapchar_o^{J_o}$ is the map from $\set{z_\beta:\beta\in\SimplesT{c}_o}\cup\set{z_*}\cup\set{x_\gamma:\gamma\in J_o}$ to~$\T_o(\tB)$ that sends $z_\beta$ to~$y^\beta$ for all $\beta\in\SimplesT{c}_o$, sends $z_*$ to $\thet_{\nu_c(\delta)}$, and sends $x_\gamma$ to~$\thet_{\nu_c(\gamma)}$ for all~$\gamma\in J_o$.
Then~$\mapchar_o^{J_o}$ extends uniquely to a homomorphism from the field of rational functions in $\set{x_\gamma:\gamma\in J_o}$ with coefficients in $\k[z_*,\,z_\beta^{\pm1}]_{\beta\in\SimplesT{c}_o}$ to the field of rational functions in $\set{\thet_{\nu_c(\gamma)}:\gamma\in J_o\cup\set\delta}$ with coefficients in $\k[y^{\pm\beta}]_{\beta\in\SimplesT{c}_o}$, and then restricts uniquely to a homomorphism from $\A(\x_{J_o},\pp_{J_o},B_{J_o})$.
We will write~$\mapchar_o^{J_o}$ for all of these maps.

The next key ingredient in the proof is the following lemma, which validates the identification of $x'_\gamma$ with $x_{\gamma'}$ in Lemma~\ref{J mut}.

\begin{lemma}\label{x mut}
Suppose $J_o$ and $J_o'$ are distinct maximal sets of pairwise compatible real roots in $\APTre{c;o}$ with $J_o'=(J_o\setminus\set\gamma)\cup\set{\gamma'}$.
If $x'_\gamma$ is the cluster variable obtained by mutating $(\x_{J_o},\pp_{J_o},B_{J_o})$ at $\gamma$, then $\mapchar_o^{J_o}$ sends $x_{\gamma'}$ to $\thet_{\nu_c(\gamma')}$.
\end{lemma}

Theorem~\ref{im exch} is one ingredient of the proof of Lemma~\ref{x mut}.
We will also need a description of the exchange relations that expand $\thet_{\nu_c(\gamma)}\cdot\thet_{\nu_c(\gamma')}$ for pairs $\gamma,\gamma'\in\APTre{c}$ that are $c$\nobreakdash-real-exchangeable.
These pairs are characterized by the following result that is the concatenation of \cite[Proposition~4.8]{affdenom} with part of \cite[Theorem~7.2]{affdenom}.
(See also \cite[Definition~4.3]{affdenom}.)

\begin{theorem}\label{real exchangeable}  
For $\gamma,\gamma'\in\APTre{c}$, the following are equivalent:
\begin{enumerate}[label=\rm(\roman*), ref=(\roman*)]
\item \label{real ex cond}
$\gamma$ and $\gamma'$ are $c$-real-exchangeable.
\item 
There exists a $c$-orbit (of size $k$) in $\SimplesT{c}$, a choice of $\beta_{[0]}$ in the orbit such that $\gamma=\kappa_{[i,j]}$ and $\gamma'=\kappa_{[i',j']}$ with $1\le i\le j<j'\le k$ and $1\le i<i'\le j'\le k$ and $i'\le j+1$ (or the same with $\gamma$ and $\gamma'$ swapped).
\end{enumerate}
\end{theorem}

Condition~\ref{real ex cond} of Theorem~\ref{real exchangeable} fails unless $\gamma$ and $\gamma'$ are in the span of $\SimplesT{c}_o$ for some $c$-orbit, in which case it is equivalent to the existence of a maximal set $J_o$ of pairwise compatible roots in $\APTre{c;o}$ that contains $\gamma$ as a non-maximal root such that $\gamma'$ is the root obtained by exchanging $\gamma$ out of~$J_o$.
Thus, we can describe the exchange relation in the notation of Figure~\ref{proof fig nonmax}.
The following fact was established as \cite[(4.3)]{affdomreg} in the proof of \cite[Proposition~4.42]{affdomreg}.

\begin{proposition}\label{real exch}
If $\gamma,\gamma'\in\APTre{c}$ are $c$-real-exchangeable, then
\[\thet_{\nu_c(\gamma)}\cdot\thet_{\nu_c(\gamma')}=\thet_{\nu_c(\phi)}\thet_{\nu_c(\phi'')}+y^{\phi''+\beta'}\thet_{\nu_c(\phi')}\thet_{\nu_c(\phi''')}.\]
\end{proposition}

Proposition~\ref{real exch} appears in \cite{affdomreg} in the notation of Theorem~\ref{real exchangeable} as an exchange relation between cluster variables with principal coefficients.
Since, in the principal coefficients case, each cluster variable equals the theta function with the same $\g$-vector, we obtain the proposition in the principal coefficients case, and thus in the general case by Proposition~\ref{struct plus plus}.

\begin{proof}[Proof of Lemma~\ref{x mut}]
Suppose $\gamma$ is maximal in $J_o$ and adopt the notation of Figure~\ref{proof fig max}.
The exchange relation in $\A(\x_{J_o},\pp_{J_o},B_{J_o})$ is 
\[x_\gamma x_{\gamma'}=
p_{\gamma;0}\prod_{\psi\in J_o}x_\psi^{[b_{\psi\gamma}]_+}
+
p_{\gamma;1}\prod_{\psi\in J_o}x_\psi^{\frac12[b_{\psi\gamma}]_++\frac12[-b_{\psi\gamma}]_+}
+
p_{\gamma;2}\prod_{\psi\in J_o}x_\psi^{[-b_{\psi\gamma}]_+,}
\]
which evaluates to 
\[x_\gamma x_{\gamma'}=
z^{\phi'+\beta}x_\phi^2
+
z_*x_\phi x_{\phi'}
+
z^{\phi+\beta'}x_{\phi'}^2.
\]
Applying $\mapchar_o^{J_o}$ to that relation and appealing to Theorem~\ref{im exch}, we see that $\mapchar_o^{J_o}$ sends $x'_\gamma$ to $\thet_{\nu_c(\gamma')}$.
(In the first place, this is true of $\mapchar_o^{J_o}$ as a map on rational functions, but since $x'_\gamma\in\A(\x_{J_o},\pp_{J_o},B_{J_o})$, it is also true of the restriction of $\mapchar_o^{J_o}$ to $\A(\x_{J_o},\pp_{J_o},B_{J_o})$.)

Next, suppose $\gamma$ is not maximal in $J_o$.
Consider first the case where $\gamma$ and $\gamma'$ are as in Figure~\ref{proof fig nonmax} (as opposed to swapped in that figure), and adopt the notation of that figure.
The exchange relation in $\A(\x_{J_o},\pp_{J_o},B_{J_o})$ is 
\[x_\gamma x_{\gamma'}
=
p_{\gamma;0}\prod_{\psi\in J_o}x_\psi^{[b_{\psi\gamma}]_+}
+
p_{\gamma;1}\prod_{\psi\in J_o}x_\psi^{[-b_{\psi\gamma}]_+},\]
which evaluates to $x_\gamma x_{\gamma'}=z^{\phi''+\beta'}x_{\phi'}x_{\phi'''}+x_{\phi}x_{\phi''}$.
As in the argument above, we conclude (using Proposition~\ref{real exch} rather than Theorem~\ref{im exch}) that $\mapchar_o^{J_o}$ sends $x'_\gamma$ to~$\thet_{\nu_c(\gamma')}$.
If instead $\gamma$ and $\gamma'$ are swapped in the figure, the exchange relation is the same and we again conclude that $\mapchar_o^{J_o}$ sends $x'_\gamma$ to $\thet_{\nu_c(\gamma')}$.
\end{proof}

We now prove Theorems~\ref{gen clus alg prop} and~\ref{gen clus alg}.
To begin, assume that~$\tB$ has nondegenerate coefficients.
Under this assumption, $\set{y^\beta:\beta\in\SimplesT{c}_o}$ (the image of $\set{z_\beta:\beta\in\SimplesT{c}_o}$) is algebraically independent.
The assumption also implies that elements of $\A(\tB)$ have well-defined $\g$-vectors.
The $\g$-vectors $\set{\nu_c(\gamma):\gamma\in J_o\cup\set{\delta}}$ are linearly independent because they span an imaginary cone in~$\F^{B^T}$.
Thus, $\set{\thet_{\nu_c(\gamma)}:\gamma\in J_o\cup\set{\delta}}$ (the image of ${\set{z_*}\cup\set{x_\gamma:\gamma\in J_o}}$) is algebraically independent.
We see that the image of $\set{z_\beta:\beta\in\SimplesT{c}_o}\cup\set{z_*}\cup\set{x_\gamma:\gamma\in J_o}$  under $\mapchar_o^{J_o}$ is an algebraically independent set in $\T_o(\tB)$.

Since the set map $\mapchar_o^{J_o}$ sends the tropical variables and the initial cluster to an algebraically independent set, the map $\mapchar_o^{J_o}$ on the field of rational functions in $\set{x_\gamma:\gamma\in J_o}$ with coefficients in $\k[z_*,\,z_\beta^{\pm1}]_{\beta\in\SimplesT{c}_o}$ is an isomorphism to the field of rational functions in $\set{\thet_{\nu_c(\gamma)}:\gamma\in J_o}$ with coefficients in $\k[\thet_{\nu_c(\delta)},\,y^{\pm\beta}]_{\beta\in\SimplesT{c}}$.
Therefore, also, the restriction of $\mapchar_o^{J_o}$ to $\A(\x_{J_o},\pp_{J_o},B_{J_o})$ is an isomorphism to its image.

In particular, $\mapchar_o^{J_o}$ restricts to a one-to-one map on cluster variables.
Now Lemmas~\ref{J mut} and~\ref{x mut} and a simple induction show that $\mapchar_o^{J_o}$ sends every cluster variable to a theta function $\thet_\gamma$ for $\gamma\in\APTre{c}$.
We index the cluster variables as $x_\gamma$ accordingly.
The same induction also shows that every seed in $\A(\x_{J_o},\pp_{J_o},B_{J_o})$ is of the form $(\x_{J_o'},\pp_{J_o'},B_{J_o'})$ for some maximal set~$J_o'$ of pairwise compatible roots in~$\APTre{c;o}$.
Using the combinatorial description, in Section~\ref{agaf sec}, of imaginary cones in terms of sets of compatible roots in $\APT{c}$, we conclude that the map $J_o'\mapsto(\x_{J_o'},\pp_{J_o'},B_{J_o'})$ is bijection from maximal sets of pairwise compatible roots to seeds.
This completes Assertion~\ref{gca prop bij} of Theorem~\ref{gen clus alg prop}, and Assertion~\ref{gca prop cyclo} also follows from the combinatorial description.
To complete the proof of Theorem~\ref{gen clus alg prop}, note that we can number the roots in $\SimplesT{c}_o$ as $0,1,\ldots,|J_o|$ in cyclic order and take $J_o$ to be $\set{\beta_{[1,j]}:j=1,\ldots,|J_o|}$ to make $B_{J_o}$ an acyclic exchange matrix of type $C_{|J_o|}$.

We continue to prove Theorem~\ref{gen clus alg} under the assumption that~$\tB$ has nondegenerate coefficients.
The induction above also shows that all the maps $\mapchar_o^{J_o}$ for various~$J_o$ coincide, and this is Assertion~\ref{gca no J} of Theorem~\ref{gen clus alg}.
We return to the notation $\mapchar_o$ for this map.
In the paragraph above, we established Assertion~\ref{gca bij} of Theorem~\ref{gen clus alg}.
Before that, we established that the map $\mapchar_o$ on $\A(\x_{J_o},\pp_{J_o},B_{J_o})$ is an isomorphism to its image.
Since $\mapchar_o$ maps $\set{z_\beta:\beta\in\SimplesT{c}_o}\cup\set{z_*}\cup\set{x_\gamma:\gamma\in J_o}$ into $\T_o(\tB)$, the image of $\mapchar_o$ is contained in $\T_o(\tB)$.
The image contains $\set{y^\beta:\beta\in\SimplesT{c}_o}$ and $\set{\thet_{\nu_c(\gamma)}:\gamma\in\APTre{c;o}}$, so Theorem~\ref{tube subalgebra} implies that the image is all of $\T_o(\tB)$.
This proves Assertion~\ref{gca nondegen}.

We have proved some assertions of Theorem~\ref{gen clus alg} under the assumption that~$\tB$ has nondegenerate coefficients.
Now remove that assumption on $\tB$, but suppose~$\tB'$ is another extension of $B$ with nondegenerate coefficients.
As in Section~\ref{coeffs sec}, construct theta functions for $\tB$ and $\tB'$ with the same indeterminates $x_1,\ldots,x_n$, use different tropical variables for each, and write $y$ for coefficients for $\tB$ and $y'$ for coefficients for~$\tB'$.
Proposition~\ref{where prin plus} implies that $\T_o(\tB)$ can be obtained from $\T_o(\tB')$ by replacing each $y'$ by $y$ throughout.
Write $\mapchar_o'$ for the map from $\A(\x_{J_o},\pp_{J_o},B_{J_o})$ to~$\T_o(\tB')$, which we have proved is an isomorphism.

The map $\mapchar_o$ is the composition of $\mapchar_o'$ followed by the specialization map from $\T_o(\tB')$ to $\T_o(\tB)$.
Since $\mapchar_o'$ is an isomorphism, Assertion~\ref{gca spec} follows.
Since also specialization is surjective, we have Assertion~\ref{gca hom}.
(Recall that the existence of a unique extension was proved above, before we imposed the assumption of nondegenerate coefficients.)
Assertions~\ref{gca no J} and~\ref{gca bij} for $\mapchar_o$ also follow from the same assertions for $\mapchar'_i$, which we already proved.
This completes the proof of Theorem~\ref{gen clus alg}.

We conclude with the proof of Theorem~\ref{gen clus alg almost}.
Let $J=\cup_oJ_o$ and write $\A_o$ for $\A(\x_{J_o},\pp_{J_o},B_{J_o})$.
The set map $\mapchar$ extends uniquely to a ring homomorphism ${\mapchar:\A(\x_J,\pp_J,B_J)\to\I(\tB)}$ exactly as in the proof of Theorem~\ref{gen clus alg}.
The set map~$\mapchar$ also restricts to set maps $\mapchar_o$.
Theorem~\ref{gen clus alg} says that each set map $\mapchar_o$ extends uniquely to a surjective homomorphism ${\mapchar_o:\A_o\to\T_o(\tB)\subseteq\I(\tB)}$  and that each is independent of the choice of $J_o$.
These maps together define a $\k[z_*]$-multilinear homomorphism from the product $\bigtimes_o\A_o$ to $\I(\tB)$ sending a tuple $(f_o:f_o\in\A_o)$ to $\prod_o\mapchar_o(f_o)$.
The universal property of the tensor product gives a unique $\k[z_*]$-linear homomorphism from $\bigotimes_{\k[z_*]}\A_o$ to $\I(\tB)$ that sends $\otimes_o f_o$ to $\prod_o\mapchar_o(f_o)$.
But this map extends the set map $\mapchar$, so the uniqueness of $\mapchar$ says that this homomorphism is~$\mapchar$, and thus ${\mapchar(\otimes_o f_o)=\prod_o\mapchar_o(f_o)}$.

The second assertion of Theorem~\ref{subalgebra} can be rephrased as the statement that $\I(\tB)$ is spanned, over $\k$, by monomials that are Laurent in the coefficients $\set{y^\beta:\beta\in\SimplesT{c}}$ and ordinary in $\set{\thet_{\nu_c(\gamma)}:\gamma\in\APTre{c}}$.
Each such monomial can be factored into monomials contained in the $\T_o(\tB)$, so the surjectivity in Assertion~\ref{gca almost hom} of Theorem~\ref{gen clus alg almost} follows from ${\mapchar(\otimes_o f_o)=\prod_o\mapchar_o(f_o)}$, the second assertion of Theorem~\ref{tube subalgebra}, and the surjectivity of each $\mapchar_o$.
Assertion~\ref{gca almost no J} of Theorem~\ref{gen clus alg almost} follows from the independence of the $\mapchar_o$ of the choice of the $J_o$.
Assertion~\ref{gca almost bij} follows from the analogous results for the $\mapchar_o$ because the set of cluster variables of $\A(\x_J,\pp_J,B_J)$ is the disjoint union of the cluster variables in the $\A(\x_{J_o},\pp_{J_o},B_{J_o})$ and because $\APTre{c}$ is the disjoint union of the $\APTre{c;o}$.

Now assume $\tB$ has nondegenerate coefficients.
As in the proof of Theorem~\ref{gen clus alg}, the set $\set{\thet_{\nu_c(\gamma)}:\gamma\in J\cup\set\delta}$ is algebraically independent.
The only linear dependences among the vectors $\SimplesT{c}$ come from the fact that $y^\delta=\prod_{\beta\in\SimplesT{c}_o}y^\beta$ for all $o$ (because $\sum_{\beta\in\SimplesT{c}_o}\beta=\delta$).
Thus since $\tB$ has nondegenerate coefficients the only algebraic dependences among the functions $\set{y^\beta:\beta\in\SimplesT{c}}$ come from the identities $\sett{\,\prod_{\beta\in\SimplesT{c}_o}z_\beta-\prod_{\beta\in\SimplesT{c}_{o'}}z_\beta:o\neq o'}$.
This is Assertion~\ref{gca almost ker}.

Finally, as in the proof of Theorem~\ref{gen clus alg}, let $\tB$ be an arbitrary extension and let~$\tB'$ be an extension with nondegenerate coefficients, so that $\mapchar'$ is the homomorphism from $\A(\x_J,\pp_J,B_J)$ to $\I(\tB')$ whose kernel generated is by the identities $\sett{\,\prod_{\beta\in\SimplesT{c}_o}z_\beta-\prod_{\beta\in\SimplesT{c}_{o'}}z_\beta:o\neq o'}$.
We can obtain $\I(\tB)$ from $\I(\tB')$ by replacing each $y'$ by $y$ throughout, and $\mapchar$ is the composition of $\mapchar'$ followed by the specialization map from $\I(\tB')$ to $\I(\tB)$.
Because $\prod_{\beta\in\SimplesT{c}_o}z_\beta=y^\delta$ for all $o$, the kernel of the specialization homomorphism on $\A(\x_J,\pp_J,B_J)$ that replaces each $z_\beta$ by $y^\beta$ contains $\sett{\,\prod_{\beta\in\SimplesT{c}_o}z_\beta-\prod_{\beta\in\SimplesT{c}_{o'}}z_\beta:o\neq o'}$.
We see that the kernel of $\mapchar$ equals the kernel of the specialization map on~$\A(\x_J,\pp_J,B_J)$.  
This is Assertion~\ref{gca almost spec}.

We have completed the proofs of Theorems~\ref{gen clus alg prop}, \ref{gen clus alg} and~\ref{gen clus alg almost}.\qed

\bibliographystyle{plain}
\bibliography{bibliography}

@article {ga,
    AUTHOR = {Fomin, Sergey and Zelevinsky, Andrei},
     TITLE = {{$Y$}-systems and generalized associahedra},
   JOURNAL = {Ann. of Math. (2)},
  FJOURNAL = {Annals of Mathematics. Second Series},
    VOLUME = {158},
      YEAR = {2003},
    NUMBER = {3},
     PAGES = {977--1018},
      IS:SN = {0003-486X},
     CODEN = {ANMAAH},
   MRCLASS = {17B20 (20F55 52B12)},
  MRNUMBER = {2031858 (2004m:17010)},
MRREVIEWER = {Ivan V. Arzhantsev},
       DOI = {10.4007/annals.2003.158.977},
       URL = {http://0-dx.doi.org.ilsprod.lib.neu.edu/10.4007/annals.2003.158.977},
}

@article {ca2,
    AUTHOR = {Fomin, Sergey and Zelevinsky, Andrei},
     TITLE = {Cluster algebras. {II}. {F}inite type classification},
   JOURNAL = {Invent. Math.},
  FJOURNAL = {Inventiones Mathematicae},
    VOLUME = {154},
      YEAR = {2003},
    NUMBER = {1},
     PAGES = {63--121},
      ISSN = {0020-9910},
     CODEN = {INVMBH},
   MRCLASS = {17B20 (05E15 16S99 52B12)},
  MRNUMBER = {2004457 (2004m:17011)},
MRREVIEWER = {Eric N. Sommers},
       DOI = {10.1007/s00222-003-0302-y},
       URL = {http://0-dx.doi.org.ilsprod.lib.neu.edu/10.1007/s00222-003-0302-y},
}

@article {camb_fan,
    AUTHOR = {Reading, Nathan and Speyer, David E.},
     TITLE = {Cambrian fans},
   JOURNAL = {J. Eur. Math. Soc. (JEMS)},
  FJOURNAL = {Journal of the European Mathematical Society (JEMS)},
    VOLUME = {11},
      YEAR = {2009},
    NUMBER = {2},
     PAGES = {407--447},
      ISSN = {1435-9855},
   MRCLASS = {20F55 (13F60)},
  MRNUMBER = {2486939 (2011a:20102)},
MRREVIEWER = {Laurent Demonet},
       DOI = {10.4171/JEMS/155},
       URL = {http://0-dx.doi.org.ilsprod.lib.neu.edu/10.4171/JEMS/155},
}

@article {ca4,
    AUTHOR = {Fomin, Sergey and Zelevinsky, Andrei},
     TITLE = {Cluster algebras. {IV}. {C}oefficients},
   JOURNAL = {Compos. Math.},
  FJOURNAL = {Compositio Mathematica},
    VOLUME = {143},
      YEAR = {2007},
    NUMBER = {1},
     PAGES = {112--164},
      ISSN = {0010-437X},
   MRCLASS = {16S99 (05E15 14M17 22E46)},
  MRNUMBER = {2295199 (2008d:16049)},
MRREVIEWER = {Christof Gei{\ss}},
       DOI = {10.1112/S0010437X06002521},
       URL = {http://0-dx.doi.org.ilsprod.lib.neu.edu/10.1112/S0010437X06002521},
}

@article {cambrian,
    AUTHOR = {Reading, Nathan},
     TITLE = {Cambrian lattices},
   JOURNAL = {Adv. Math.},
  FJOURNAL = {Advances in Mathematics},
    VOLUME = {205},
      YEAR = {2006},
    NUMBER = {2},
     PAGES = {313--353},
      ISSN = {0001-8708},
     CODEN = {ADMTA4},
   MRCLASS = {05E25 (06B10 20F55 52C07)},
  MRNUMBER = {2258260 (2007g:05195)},
MRREVIEWER = {Axel Hultman},
       DOI = {10.1016/j.aim.2005.07.010},
       URL = {http://0-dx.doi.org.ilsprod.lib.neu.edu/10.1016/j.aim.2005.07.010},
}

@article{ca3,
    AUTHOR = {Berenstein, Arkady and Fomin, Sergey and Zelevinsky, Andrei},
     TITLE = {Cluster algebras. {III}. {U}pper bounds and double {B}ruhat
              cells},
   JOURNAL = {Duke Math. J.},
  FJOURNAL = {Duke Mathematical Journal},
    VOLUME = {126},
      YEAR = {2005},
    NUMBER = {1},
     PAGES = {1--52},
      ISSN = {0012-7094},
   MRCLASS = {16S99 (05E15 14M17 22E46)},
  MRNUMBER = {2110627},
       DOI = {10.1215/S0012-7094-04-12611-9},
       URL = {https://doi-org.prox.lib.ncsu.edu/10.1215/S0012-7094-04-12611-9},
}

@article {MSW2,
    AUTHOR = {Musiker, Gregg and Schiffler, Ralf and Williams, Lauren},
     TITLE = {Bases for cluster algebras from surfaces},
   JOURNAL = {Compos. Math.},
  FJOURNAL = {Compositio Mathematica},
    VOLUME = {149},
      YEAR = {2013},
    NUMBER = {2},
     PAGES = {217--263},
      ISSN = {0010-437X,1570-5846},
   MRCLASS = {13F60 (05C70 05E15)},
  MRNUMBER = {3020308},
MRREVIEWER = {Olga\ Kravchenko},
       DOI = {10.1112/S0010437X12000450},
       URL = {https://doi-org.prox.lib.ncsu.edu/10.1112/S0010437X12000450},
}

@article {ca1,
    AUTHOR = {Fomin, Sergey and Zelevinsky, Andrei},
     TITLE = {Cluster algebras. {I}. {F}oundations},
   JOURNAL = {J. Amer. Math. Soc.},
  FJOURNAL = {Journal of the American Mathematical Society},
    VOLUME = {15},
      YEAR = {2002},
    NUMBER = {2},
     PAGES = {497--529 (electronic)},
      ISSN = {0894-0347},
   MRCLASS = {16S99 (14M99 17B99)},
  MRNUMBER = {1887642 (2003f:16050)},
MRREVIEWER = {Eric N. Sommers},
       DOI = {10.1090/S0894-0347-01-00385-X},
       URL = {http://0-dx.doi.org.ilsprod.lib.neu.edu/10.1090/S0894-0347-01-00385-X},
}

@ARTICLE{FeShThTu12,
    AUTHOR = {Felikson, Anna and Shapiro, Michael and Thomas, Hugh and
              Tumarkin, Pavel},
     TITLE = {Growth rate of cluster algebras},
   JOURNAL = {Proc. Lond. Math. Soc. (3)},
  FJOURNAL = {Proceedings of the London Mathematical Society. Third Series},
    VOLUME = {109},
      YEAR = {2014},
    NUMBER = {3},
     PAGES = {653--675},
      ISSN = {0024-6115,1460-244X},
   MRCLASS = {13F60},
  MRNUMBER = {3260289},
MRREVIEWER = {Calin\ Chindris},
       DOI = {10.1112/plms/pdu010},
       URL = {https://doi-org.prox.lib.ncsu.edu/10.1112/plms/pdu010},
}

@article {Seven,
    AUTHOR = {Seven, Ahmet I.},
     TITLE = {Cluster algebras and semipositive symmetrizable matrices},
   JOURNAL = {Trans. Amer. Math. Soc.},
  FJOURNAL = {Transactions of the American Mathematical Society},
    VOLUME = {363},
      YEAR = {2011},
    NUMBER = {5},
     PAGES = {2733--2762},
      ISSN = {0002-9947},
     CODEN = {TAMTAM},
   MRCLASS = {13F60 (05E15 15B48)},
  MRNUMBER = {2763735},
MRREVIEWER = {Kyungyong Lee},
       DOI = {10.1090/S0002-9947-2010-05255-9},
       URL = {http://dx.doi.org/10.1090/S0002-9947-2010-05255-9},
}

@article {typefree,
    AUTHOR = {Reading, Nathan and Speyer, David E.},
     TITLE = {Sortable elements in infinite {C}oxeter groups},
   JOURNAL = {Trans. Amer. Math. Soc.},
  FJOURNAL = {Transactions of the American Mathematical Society},
    VOLUME = {363},
      YEAR = {2011},
    NUMBER = {2},
     PAGES = {699--761},
      ISSN = {0002-9947},
     CODEN = {TAMTAM},
   MRCLASS = {20F55},
  MRNUMBER = {2728584 (2011j:20100)},
MRREVIEWER = {Alessandro Conflitti},
       DOI = {10.1090/S0002-9947-2010-05050-0},
       URL = {http://dx.doi.org/10.1090/S0002-9947-2010-05050-0},
}

@article {Howlett,
    AUTHOR = {Howlett, Robert B.},
     TITLE = {Coxeter groups and {$M$}-matrices},
   JOURNAL = {Bull. London Math. Soc.},
  FJOURNAL = {The Bulletin of the London Mathematical Society},
    VOLUME = {14},
      YEAR = {1982},
    NUMBER = {2},
     PAGES = {137--141},
      ISSN = {0024-6093},
     CODEN = {LMSBBT},
   MRCLASS = {20F05 (14B05 20H15)},
  MRNUMBER = {647197 (83g:20032)},
MRREVIEWER = {Wim H. Hesselink},
       DOI = {10.1112/blms/14.2.137},
       URL = {http://dx.doi.org/10.1112/blms/14.2.137},
}

@article {sortable,
    AUTHOR = {Reading, Nathan},
     TITLE = {Clusters, {C}oxeter-sortable elements and noncrossing
              partitions},
   JOURNAL = {Trans. Amer. Math. Soc.},
  FJOURNAL = {Transactions of the American Mathematical Society},
    VOLUME = {359},
      YEAR = {2007},
    NUMBER = {12},
     PAGES = {5931--5958},
      ISSN = {0002-9947},
     CODEN = {TAMTAM},
   MRCLASS = {20F55 (05A18 05E15)},
  MRNUMBER = {2336311 (2009d:20093)},
MRREVIEWER = {Pavlo Pylyavskyy},
       DOI = {10.1090/S0002-9947-07-04319-X},
       URL = {http://dx.doi.org/10.1090/S0002-9947-07-04319-X},
}

@article {MRZ,
    AUTHOR = {Marsh, Robert and Reineke, Markus and Zelevinsky, Andrei},
     TITLE = {Generalized associahedra via quiver representations},
   JOURNAL = {Trans. Amer. Math. Soc.},
  FJOURNAL = {Transactions of the American Mathematical Society},
    VOLUME = {355},
      YEAR = {2003},
    NUMBER = {10},
     PAGES = {4171--4186},
      ISSN = {0002-9947},
     CODEN = {TAMTAM},
   MRCLASS = {52B11 (05E15 16G20 17B20 17B37)},
  MRNUMBER = {1990581},
MRREVIEWER = {Ivan V. Arzhantsev},
       DOI = {10.1090/S0002-9947-03-03320-8},
       URL = {http://dx.doi.org/10.1090/S0002-9947-03-03320-8},
}

@article{afframe,
    AUTHOR = {Reading, Nathan and Speyer, David E.},
     TITLE = {Cambrian frameworks for cluster algebras of affine type},
   JOURNAL = {Trans. Amer. Math. Soc.},
  FJOURNAL = {Transactions of the American Mathematical Society},
    VOLUME = {370},
      YEAR = {2018},
    NUMBER = {2},
     PAGES = {1429--1468},
      ISSN = {0002-9947},
   MRCLASS = {13F60 (20F55)},
  MRNUMBER = {3729507},
       DOI = {10.1090/tran/7193},
       URL = {https://doi-org.prox.lib.ncsu.edu/10.1090/tran/7193},
}

@article {CarrDevadoss,
    AUTHOR = {Carr, Michael P. and Devadoss, Satyan L.},
     TITLE = {Coxeter complexes and graph-associahedra},
   JOURNAL = {Topology Appl.},
  FJOURNAL = {Topology and its Applications},
    VOLUME = {153},
      YEAR = {2006},
    NUMBER = {12},
     PAGES = {2155--2168},
      ISSN = {0166-8641},
     CODEN = {TIAPD9},
   MRCLASS = {52B11 (05B45 14H10 14P25)},
  MRNUMBER = {2239078},
MRREVIEWER = {Seth Sullivant},
       DOI = {10.1016/j.topol.2005.08.010},
       URL = {http://dx.doi.org/10.1016/j.topol.2005.08.010},
}

@article{GHKK,
    AUTHOR = {Gross, Mark and Hacking, Paul and Keel, Sean and Kontsevich,
              Maxim},
     TITLE = {Canonical bases for cluster algebras},
   JOURNAL = {J. Amer. Math. Soc.},
  FJOURNAL = {Journal of the American Mathematical Society},
    VOLUME = {31},
      YEAR = {2018},
    NUMBER = {2},
     PAGES = {497--608},
      ISSN = {0894-0347},
   MRCLASS = {13F60 (14J33)},
  MRNUMBER = {3758151},
MRREVIEWER = {Ralf Schiffler},
       DOI = {10.1090/jams/890},
       URL = {https://doi-org.prox.lib.ncsu.edu/10.1090/jams/890},
}

@article {Rupel,
    AUTHOR = {Rupel, Dylan and Stella, Salvatore},
     TITLE = {Some consequences of categorification},
   JOURNAL = {SIGMA Symmetry Integrability Geom. Methods Appl.},
  FJOURNAL = {SIGMA. Symmetry, Integrability and Geometry. Methods and
              Applications},
    VOLUME = {16},
      YEAR = {2020},
     PAGES = {Paper No. 007, 8},
      ISSN = {1815-0659},
   MRCLASS = {13F60 (16G20)},
  MRNUMBER = {4057626},
MRREVIEWER = {Juan\ Bosco\ Fr\'ias-Medina},
       DOI = {10.3842/SIGMA.2020.007},
       URL = {https://doi-org.prox.lib.ncsu.edu/10.3842/SIGMA.2020.007},
}

@article {affdenom,
    AUTHOR = {Reading, Nathan and Stella, Salvatore},
     TITLE = {An affine almost positive roots model},
   JOURNAL = {J. Comb. Algebra},
  FJOURNAL = {Journal of Combinatorial Algebra},
    VOLUME = {4},
      YEAR = {2020},
    NUMBER = {1},
     PAGES = {1--59},
      ISSN = {2415-6302,2415-6310},
   MRCLASS = {20F55 (05E16 13F60 17B22)},
  MRNUMBER = {4073889},
       DOI = {10.4171/jca/37},
       URL = {https://doi.org/10.4171/jca/37},
}

@article {afforb,
    AUTHOR = {Reading, Nathan and Stella, Salvatore},
     TITLE = {The action of a {C}oxeter element on an affine root system},
   JOURNAL = {Proc. Amer. Math. Soc.},
  FJOURNAL = {Proceedings of the American Mathematical Society},
    VOLUME = {148},
      YEAR = {2020},
    NUMBER = {7},
     PAGES = {2783--2798},
      ISSN = {0002-9939,1088-6826},
   MRCLASS = {20F55},
  MRNUMBER = {4099768},
MRREVIEWER = {Neil\ J. Y. Fan},
       DOI = {10.1090/proc/14769},
       URL = {https://doi.org/10.1090/proc/14769},
}

@article {affscat,
    AUTHOR = {Reading, Nathan and Stella, Salvatore},
  journal = {arXiv:2205.05125},
    title={Cluster scattering diagrams of acyclic affine type},
    year={2022},
}

@article {scatcomb,
    AUTHOR = {Reading, Nathan},
     TITLE = {A combinatorial approach to scattering diagrams},
   JOURNAL = {Algebr. Comb.},
  FJOURNAL = {Algebraic Combinatorics},
    VOLUME = {3},
      YEAR = {2020},
    NUMBER = {3},
     PAGES = {603--636},
   MRCLASS = {13F60},
  MRNUMBER = {4113600},
MRREVIEWER = {Ibrahim Saleh},
       DOI = {10.5802/alco.107},
       URL = {https://doi-org.prox.lib.ncsu.edu/10.5802/alco.107},
}

@article {scatfan,
    AUTHOR = {Reading, Nathan},
     TITLE = {Scattering fans},
   JOURNAL = {Int. Math. Res. Not. IMRN},
  FJOURNAL = {International Mathematics Research Notices. IMRN},
      YEAR = {2020},
    NUMBER = {23},
     PAGES = {9640--9673},
      ISSN = {1073-7928,1687-0247},
   MRCLASS = {13F60 (05E14 14J33)},
  MRNUMBER = {4182806},
MRREVIEWER = {Fan\ Qin},
       DOI = {10.1093/imrn/rny260},
       URL = {https://doi.org/10.1093/imrn/rny260},
}

@book {Kac,
    author = {V. Kac},
    title = {Infinite-dimensional {L}ie algebras},
    edition = {Third},
    publisher = {Cambridge University Press, Cambridge},
    year = {1990},
    pages = {xxii+400},
}

@article {RSDom,
    AUTHOR = {Rupel, Dylan and Stella, Salvatore},
     TITLE = {Dominance regions for rank two cluster algebras},
   JOURNAL = {Ann. Comb.},
  FJOURNAL = {Annals of Combinatorics},
    VOLUME = {27},
      YEAR = {2023},
    NUMBER = {4},
     PAGES = {873--894},
      ISSN = {0218-0006,0219-3094},
   MRCLASS = {13F60},
  MRNUMBER = {4657334},
       DOI = {10.1007/s00026-023-00636-4},
       URL = {https://doi.org/10.1007/s00026-023-00636-4},
}

@article {FanQin,
    AUTHOR = {Qin, Fan},
     TITLE = {Bases for upper cluster algebras and tropical points},
   JOURNAL = {J. Eur. Math. Soc. (JEMS)},
  FJOURNAL = {Journal of the European Mathematical Society (JEMS)},
    VOLUME = {26},
      YEAR = {2024},
    NUMBER = {4},
     PAGES = {1255--1312},
      ISSN = {1435-9855,1435-9863},
   MRCLASS = {13F60},
  MRNUMBER = {4721032},
       DOI = {10.4171/jems/1308},
       URL = {https://doi.org/10.4171/jems/1308},
}

@article{affdomreg,
      title={Dominance regions for affine cluster algebras}, 
      author={Nathan Reading and Dylan Rupel and Salvatore Stella},
      year={2025},
     journal={arXiv:2512.02218},
      eprint={2512.02218},
      archivePrefix={arXiv},
      primaryClass={math.RT},
      url={https://arxiv.org/abs/2512.02218}, 
}

@article {ChekhovShapiro,
    AUTHOR = {Chekhov, Leonid and Shapiro, Michael},
     TITLE = {Teichm\"uller spaces of {R}iemann surfaces with orbifold
              points of arbitrary order and cluster variables},
   JOURNAL = {Int. Math. Res. Not. IMRN},
  FJOURNAL = {International Mathematics Research Notices. IMRN},
      YEAR = {2014},
    NUMBER = {10},
     PAGES = {2746--2772},
      ISSN = {1073-7928,1687-0247},
   MRCLASS = {32G15 (13F60 30F60)},
  MRNUMBER = {3214284},
       DOI = {10.1093/imrn/rnt016},
       URL = {https://doi-org.prox.lib.ncsu.edu/10.1093/imrn/rnt016},
}

@article{MandelQin,
      title={Bracelets bases are theta bases}, 
      author={Travis Mandel and Fan Qin},
      year={2023},
      journal={arXiv:2301.11101},
      eprint={2301.11101},
      archivePrefix={arXiv},
      primaryClass={math.QA},
      url={https://arxiv.org/abs/2301.11101}, 
}

@article {affncA,
    AUTHOR = {Brestensky, Laura G. and Reading, Nathan},
     TITLE = {Noncrossing partitions of an annulus},
   JOURNAL = {Comb. Theory},
  FJOURNAL = {Combinatorial Theory},
    VOLUME = {5},
      YEAR = {2025},
    NUMBER = {1},
     PAGES = {Paper No. 12, 49},
      ISSN = {2766-1334},
   MRCLASS = {20F55 (05E16 20F36)},
  MRNUMBER = {4882076},
}

@article {McSul,
    AUTHOR = {McCammond, Jon and Sulway, Robert},
     TITLE = {Artin groups of {E}uclidean type},
   JOURNAL = {Invent. Math.},
  FJOURNAL = {Inventiones Mathematicae},
    VOLUME = {210},
      YEAR = {2017},
    NUMBER = {1},
     PAGES = {231--282},
      ISSN = {0020-9910,1432-1297},
   MRCLASS = {20F36},
  MRNUMBER = {3698343},
MRREVIEWER = {Kisnney\ Emiliano de Almeida},
       DOI = {10.1007/s00222-017-0728-2},
       URL = {https://doi-org.prox.lib.ncsu.edu/10.1007/s00222-017-0728-2},
}

@misc{affncD,
      title={Symmetric noncrossing partitions of an annulus with double points}, 
      author={Nathan Reading},
      year={2025},
      eprint={2312.17331},
      archivePrefix={arXiv},
      primaryClass={math.CO},
      url={https://arxiv.org/abs/2312.17331}, 
}

@incollection {Markl,
    AUTHOR = {Markl, Martin},
     TITLE = {Simplex, associahedron, and cyclohedron},
 BOOKTITLE = {Higher homotopy structures in topology and mathematical
              physics ({P}oughkeepsie, {NY}, 1996)},
    SERIES = {Contemp. Math.},
    VOLUME = {227},
     PAGES = {235--265},
 PUBLISHER = {Amer. Math. Soc., Providence, RI},
      YEAR = {1999},
      ISBN = {0-8218-0913-X},
   MRCLASS = {57P99 (18C10)},
  MRNUMBER = {1665469},
MRREVIEWER = {Donald\ M.\ Davis},
       DOI = {10.1090/conm/227/03259},
       URL = {https://doi-org.prox.lib.ncsu.edu/10.1090/conm/227/03259},
}

@incollection {BottTaubes,
    AUTHOR = {Bott, Raoul and Taubes, Clifford},
     TITLE = {On the self-linking of knots},
      NOTE = {Topology and physics},
   JOURNAL = {J. Math. Phys.},
  FJOURNAL = {Journal of Mathematical Physics},
    VOLUME = {35},
      YEAR = {1994},
    NUMBER = {10},
     PAGES = {5247--5287},
      ISSN = {0022-2488,1089-7658},
   MRCLASS = {57M25 (81T18 81T40)},
  MRNUMBER = {1295465},
MRREVIEWER = {Zhenghan\ Wang},
       DOI = {10.1063/1.530750},
       URL = {https://doi-org.prox.lib.ncsu.edu/10.1063/1.530750},
}

@article{AkagiNakanishi,
      title={Relation between generalized and ordinary cluster algebras}, 
      author={Ryota Akagi and Tomoki Nakanishi},
      year={2026},
      journal={arXiv:2512.21062},
      eprint={2512.21062},
      archivePrefix={arXiv},
      primaryClass={math.RT},
      url={https://arxiv.org/abs/2512.21062}, 
}

@article{RamosWhiting,
      title={Generalized cluster algebras are subquotients of cluster algebras}, 
      author={Rolando Ramos and David Whiting},
      year={2025},
      journal={arXiv:2504.20931},
      eprint={2504.20931},
      archivePrefix={arXiv},
      primaryClass={math.RA},
      url={https://arxiv.org/abs/2504.20931}, 
}

@article{canonical,
      title={Mutation of theta functions},
      author={Nathan Reading and Salvatore Stella},
      year={2026},
      journal={arXiv:2603.19391},
      eprint={2603.19391},
      archivePrefix={arXiv},
      primaryClass={math.CO},
      url={https://arxiv.org/abs/2603.19391},
}

@article {Tomoki,
    AUTHOR = {Nakanishi, Tomoki},
     TITLE = {Structure of seeds in generalized cluster algebras},
   JOURNAL = {Pacific J. Math.},
  FJOURNAL = {Pacific Journal of Mathematics},
    VOLUME = {277},
      YEAR = {2015},
    NUMBER = {1},
     PAGES = {201--217},
      ISSN = {0030-8730,1945-5844},
   MRCLASS = {13F60},
  MRNUMBER = {3393688},
MRREVIEWER = {Hugh\ Ross\ Thomas},
       DOI = {10.2140/pjm.2015.277.201},
       URL = {https://doi-org.prox.lib.ncsu.edu/10.2140/pjm.2015.277.201},
}

@article{PlamondonStella,
      title={On the growth of friezes via theta functions},
      author={Pierre-Guy Plamondon and Salvatore Stella},
      year={2026},
      journal={In preparation},
}
\vspace{-0.175 em}

\end{document}